\documentclass[reqno,11pt]{amsart}
\usepackage{hyperref}

\usepackage{enumerate}

\usepackage[all]{xy}
\usepackage[usenames,dvipsnames]{xcolor}
\usepackage{tikz}
\usepackage{tikz-cd}

\usepackage{pgfplots}
\usetikzlibrary{3d}
\usetikzlibrary{shapes,decorations.pathreplacing}
\usetikzlibrary{calc}
\usetikzlibrary{backgrounds}
\usetikzlibrary{automata}
\usetikzlibrary{positioning}
\usetikzlibrary{decorations.markings}
\usetikzlibrary{arrows}
\usetikzlibrary{scopes}
\usetikzlibrary{intersections}
\tikzstyle myBG=[line width=3pt,opacity=1]

\def\ThCon#1{\mathop{
\unitlength=1mm
\begin{picture}(6,4)
\put(0,0){$\leftrightarrow$} \put(1.1,1.2){$*$}\put(3.5,-1){{\tiny
$#1$}}
\end{picture}}}

\def\ssr#1{\mathop{
\unitlength=1mm
\begin{picture}(5,4)
\put(0,0){$\rightarrow$} 
\put(3,-0.6){{\tiny $ _#1$}}
\end{picture}}}

\def\RhCon#1{\mathop{
\unitlength=1mm
\begin{picture}(7,4)
\put(0,0){$\longrightarrow$} \put(2,1.2){$*$}\put(5.1,-1){{\tiny
$#1$}}
\end{picture}}}

\newcommand{\E}{\mathfrak{E}}
\newcommand{\C}{\mathfrak{C}}
\newcommand{\R}{\mathfrak{R}}
\newcommand{\B}{\mathfrak{B}}
\newcommand{\ii}{\iota}
\newcommand{\cc}{\kappa}

\newcommand{\EMr}{\overrightarrow{EM}}
\newcommand{\EMl}{\overleftarrow{EM}}
\newcommand{\EMb}{\overleftrightarrow{EM}}

\newcommand{\ZG}{{\mathbb{Z}G}}

\newcommand{\ZM}{{\mathbb{Z}M}}

\newcommand{\Z}{\mathbb{Z}}

\newcommand{\gu}{guarded }
\newcommand{\gud}{guarded}

\newcommand{\FPn}{{\rm FP}\sb n}
\newcommand{\Fn}{{\rm F}\sb n}
\newcommand{\F}{{\rm F}}

\newcommand{\FP}{{\rm FP}}
\newcommand{\FPinfty}{{\rm FP}\sb \infty}
\newcommand{\Finfty}{{\rm F}\sb \infty}

\definecolor{cof}{RGB}{219,144,71}
\definecolor{pur}{RGB}{186,146,162}
\definecolor{greeo}{RGB}{91,173,69}
\definecolor{greet}{RGB}{52,111,72}

\usepackage{amssymb}
\usepackage{amsgen}
\usepackage{amsmath}
\usepackage{amsthm}
\usepackage{mathrsfs}
\usepackage{cite}
\usepackage{amsfonts}

\newcommand{\p}{\varphi}

\newcommand{\ov}[1]{\ensuremath{\overline {#1}}}
\newcommand{\til}[1]{\ensuremath{\widetilde {#1}}}

\newcommand{\wh}{\widehat}

\newcommand{\Hom}{\mathop{\mathrm{Hom}}\nolimits}

\newtheorem{Thm}{Theorem}[section]
\newtheorem{Prop}[Thm]{Proposition}
\newtheorem{theorem}[Thm]{Theorem}

\newtheorem{Lemma}[Thm]{Lemma}
{\theoremstyle{definition}
}
{\theoremstyle{remark}
\newtheorem{Rmk}[Thm]{Remark}}
{\theoremstyle{remark}
\newtheorem{remark}[Thm]{Remark}}
\newtheorem{Cor}[Thm]{Corollary}
\newtheorem{corollary}[Thm]{Corollary}
{\theoremstyle{remark}
\newtheorem{Example}[Thm]{Example}}
{\theoremstyle{remark}
}

\newtheorem{proposition}[Thm]{Proposition}

{\theoremstyle{definition}
\newtheorem{definition}[Thm]{Definition}}
{\theoremstyle{remark}
}

{\theoremstyle{remark}
}
{\theoremstyle{remark}
}
{\theoremstyle{remark}
}
{\theoremstyle{remark}
\newtheorem{example}[Thm]{Example}}
{\theoremstyle{remark}
}
\numberwithin{equation}{section}

\usepackage{graphicx}

\title[Topological finiteness properties]{Topological finiteness
  properties of monoids \\  Part 1: Foundations}
  \date{\today }

\subjclass[2010]{20M50, 20M05, 20J05, 57M07, 20F10, 20F65}
  \keywords{ CW complex,
  monoid, equivariant homotopy theory, homological finitenss property
  $\FPn$, cohomological dimension, Hoschild cohomology, rewriting
  system, collapsing scheme, discrete Morse theory.}
  \thanks{This work was supported by the
EPSRC grant EP/N033353/1 `Special inverse monoids: subgroups, structure, geometry, rewriting systems and the word problem'.  The second author was supiported in part by
United States-Israel Binational Science Foundation \#2012080 and NSA MSP \#H98230-16-1-0047.
}

\author{ROBERT D. GRAY\and BENJAMIN STEINBERG}
\address{School of Mathematics,\\ University of East Anglia,\\ Norwich NR4 7TJ,\\ England
\and
Department of Mathematics,\\ City College of New York,\\ Convent Avenue at 138th Street,\\ New York, New York 10031,\\ USA}
\email{Robert.D.Gray@uea.ac.uk\and bsteinberg@ccny.cuny.edu}
\begin{document}

\begin{abstract}
  We initiate the study of higher dimensional topological finiteness
  properties of monoids. This is done by developing the theory of
  monoids acting on CW complexes. For this we establish the
  foundations of $M$-equivariant homotopy theory where $M$ is a
  discrete monoid. For projective $M$-CW complexes we prove several
  fundamental results such as the homotopy extension and lifting
  property, which we use to prove the $M$-equivariant Whitehead
  theorems. We define a left equivariant classifying space as a
  contractible projective $M$-CW complex. We prove that such a space
  is unique up to $M$-homotopy equivalence and give a canonical model
  for such a space via the nerve of the right Cayley graph category of
  the monoid. The topological finiteness conditions left-$\Fn$ and
  left geometric dimension are then defined for monoids in terms of
  existence of a left equivariant classifying space satisfying
  appropriate finiteness properties. We also introduce the bilateral
  notion of $M$-equivariant classifying space, proving uniqueness and
  giving a canonical model via the nerve of the two-sided Cayley graph
  category, and we define the associated finiteness properties
  bi-$\Fn$ and geometric dimension. We explore the connections between
  all of the these topological finiteness properties and several
  well-studied homological finiteness properties of monoids which are
  important in the theory of string rewriting systems, including
  $\FPn$, cohomological dimension, and Hochschild cohomological
  dimension. We also develop the corresponding theory of
  $M$-equivariant collapsing schemes (that is, $M$-equivariant
  discrete Morse theory), and among other things apply it to give
  topological proofs of results of Anick, Squier and Kobayashi that monoids
  which admit presentations by complete rewriting systems are left-
  right- and bi-$\FPinfty$.
\end{abstract}
\maketitle

\section{Introduction}


The study of the higher dimensional finiteness properties of groups was
initiated fifty years ago by C. T. C. Wall \cite{Wall1965}  and Serre
\cite{Serre1971}.
An Eilenberg--MacLane complex $K(G,1)$ for a discrete group $G$, also
called a classifying space, is an aspherical CW complex with
fundamental group $G$. Such a space can always be constructed for any
group $G$ (e.g. via the bar construction) and it is unique up to
homotopy equivalence. While useful for theoretical purposes, this
canonical $K(G,1)$-complex is very big and is often not useful for
practical purposes, specifically if one wants to compute the homology
of the group. It is therefore natural to seek a `small' $K(G,1)$ for a
given group by imposing various finiteness conditions on the
space. Two of the most natural and well-studied such conditions are
the topological finiteness property $\Fn$ and the geometric dimension
$\mathrm{gd}(G)$ of the group.

Property $\Fn$ was introduced by C. T. C. Wall in \cite{Wall1965}.  A
group $G$ is said to be of type $\Fn$ if it has an Eilenberg--MacLane
complex $K(G,1)$ with finite $n$-skeleton. It is easily
verified that a group is finitely generated if and only if it is of
type $\F_1$ and is finitely presented if and only if it is of type
$\F_2$. Thus property $\Fn$ generalises the two fundamental finiteness
properties of being finitely generated, or finitely presented, to higher
dimensions.
The geometric dimension of $G$, denoted $\mathrm{gd}(G)$, is the smallest
non-negative integer $n$ such that there exists an $n$-dimensional $K(G,
1)$ complex. If no such $n$ exists, then we set $\mathrm{gd}(G) =
\infty$.
For more general background on higher
dimensional finiteness properties of groups we refer the reader to the
books \cite[Chapter~8]{BrownCohomologyBook}, \cite[Chapters~6-9]{GeogheganBook}, or the
survey article \cite{Brown2010}.


Each of these topological finiteness properties has a natural
counterpart in homological algebra given in terms of the existence of
projective resolutions of $\ZG$-modules. The analogue of $\Fn$ in this
context is the homological finiteness property $\FPn$, while geometric
dimension corresponds to the cohomological dimension of the
group. Recall that a group $G$ is said to be of type $\FPn$ (for a
positive integer $n$) if there is a projective resolution $P$ of $\Z$
over $\ZG$ such that $P_i$ is finitely generated for $i \leq n$. We
say that $G$ is of type $\FPinfty$ if there is a projective resolution
$P$ of $\Z$ over $\ZG$ with $P_i$ finitely generated for all $i$.  The
property $\FPn$ was introduced for groups by Bieri in \cite{Bieri1976}
and since then has received a great deal of attention in the
literature; see
\cite{Bestvina1997,Bieri2001,Brady1999,Bux2007,Leary2006}. For groups,
$\Fn$ and $\FPn$ are equivalent for $n=0,1$, while important results
of Bestvina and Brady \cite{Bestvina1997} show that $\FP_2$ is
definitely weaker than $\F_2$. For higher $n$ there are no further
differences, in that a group $G$ is of type $\Fn$
$(2 \leq n \leq \infty)$ if and only if it is finitely presented and
of type $\FPn$. One natural line of investigation has been the study
of the closure properties of $\FPn$. Examples include results about
the behaviour of $\FPn$ under taking: finite index subgroups or
extensions, direct (and semidirect) products, wreath products, HNN
extensions, amalgamated free products, and quasi-isometry invariance;
see for example \cite{Alonso1994, Baumslag1998, Bieri1976}.
In \cite{Bridson2002} it is shown that if $G$ is a subgroup of
a direct product of $n$ surface groups, then if $G$ is of type $\FPn$
then $G$ has a subgroup of finite index which is the direct product of
at most $n$ finitely generated surface groups. Thompson's groups, and several interesting generalisations of these groups, have all be shown to be of type $\FPinfty$; see for example \cite{Brown1987, Fluch2013, Stein1992, GubaSapirDirected}.

The cohomological dimension of a group $G$, denoted $\mathrm{cd}(G)$,
is the smallest non-negative integer $n$ such that there exists a
projective resolution $P = (P_i)_{i \geq 0}$ of $\Z$ over $\ZG$ of
length $\leq n$, i.e., satisfying $P_i = 0$ for $i>n$. (Or, if no such
$n$ exists, then we set $\mathrm{cd}(G) = \infty$.) The geometric
dimension of a group provides an upper bound for the cohomological
dimension. It is easily seen that $\mathrm{gd}(G) = \mathrm{cd}(G) =
0$ if and only if $G$ is trivial. It follows from important results of
Stallings \cite{Stallings1968} and Swan \cite{Swan1969} that $\mathrm{gd}(G) = \mathrm{cd}(G) =
1$ if and only if $G$ is non-trivial free group. Eilenberg and Ganea \cite{Eilenberg1957} proved that
for $n \geq 3$ the cohomological and the geometric dimension of a group are the same. The
famous Eilenberg--Ganea problem asks whether this also holds in dimension two.


Working in the more general context of monoids, and projective resolutions of left $\ZM$-modules, gives the notion of left-$\FPn$, and left cohomological dimension, of a monoid $M$. There is an obvious dual notion of monoids of type right-$\FPn$, working with right $\ZM$-modules. Working instead with bimodules resolutions of the $(\ZM, \ZM)$-bimodule $\ZM$ one obtains the notion \emph{bi-$\FPn$} introduced and studied in \cite{KobayashiOtto2001}. Property bi-$\FP_n$  is of interest from the point of view of Hochschild cohomology, which is the standard notion of cohomology for rings; \cite{HochschildCoh},\cite[Chapter~9]{Weibel1994}, or~\cite{Mitchell1972}. For monoids all these notions of $\FPn$ are known to be different, while for groups they are all equivalent; see \cite{Cohen1992, Pride2006}. Similarly there is a dual notion of the right cohomological dimension of a monoid which again is in general not equal to the left cohomological dimension; see \cite{Guba1998}.  The two-sided notion is the Hochschild cohomological dimension~\cite{Mitchell1972}.

In monoid and semigroup theory the property $\FPn$ arises naturally in the
study of string rewriting systems (i.e. semigroup presentations).
The history of rewriting systems in monoids and groups is long and distinguished, and has roots in fundamental work of Dehn and Thue.
A central topic in this area is the study of complete rewriting systems and in methods for computing normal forms. A finite complete rewriting system is a finite presentation for a monoid of a particular form (both confluent and Noetherian) which in particular gives a solution of the word problem for the monoid; see \cite{BookAndOtto}.
It is therefore of considerable interest to develop an understanding of which monoids are presentable by such rewriting systems.
Many important classes of groups are known to be presentable by finite complete rewriting systems, including surface groups, Coxeter groups, and  many closed three-manifold groups.
Rewriting systems continue to receive a lot of attention in the literature; see \cite{Silva2009_2,Silva2009_1,Chouraqui2006,Miasnikov2009,Goodman2008,Hermiller1999,Pride2005}.  The connection between complete rewriting systems and homological finiteness properties is given by a result of
Anick \cite{Anick1986} (see also \cite{Brown1989}) which shows that a monoid that admits such a presentation must be of type left- and right-$\FPinfty$; the special case of $\FP_3$ was also handled by Squier~\cite{Squier1987}. More generally Kobayashi \cite{Kobayashi2005} proved that any such monoid is of type bi-$\FPinfty$. A range of other interesting homotopical and homological finiteness properties have been studied in relation to monoids defined by compete rewriitng systems including finite homological type, finite derivation type, and higher dimensional analogues $\mathrm{FDT}_n$; see \cite{Squier1994, Pride2004, Pride2005, Guiraud2012}.
More background on the importance the property $\FP_n$ (and other related finiteness conditions) in semigroup theory, and the connections with the theory of string rewriting systems may be found in the survey articles \cite{Cohen1997,Otto1997}.
Results on cohomology, and cohomological dimension, of monoids include \cite{Adams1967, Guba1996, Nico1969, Nico1972, Cheng1980, Nunes1995, Guba1998} and \cite{Novikov1998}. The cohomological dimension of left regular bands was recently considered in \cite{Margolis2015} and~\cite{MSS2015} where connections with the Leray number of simplicial complexes~\cite{Leray0} and the homology of cell complexes was obtained.

It is often easier to establish the topological finiteness properties
$\Fn$ for a group than the homological finiteness properties $\FPn$,
especially if there is a suitable geometry or topological space available
on which the group acts cocompactly. The desired homological
finiteness properties can then be derived by the above-mentioned
result for groups, that $\Fn$ (for $n \geq 2$) is equivalent to being  
finitely presented and 
of type 
$\FPn$.
In contrast, no
corresponding theory of $\Fn$ for monoids currently exists. Similarly,
there is currently no analogue of geometric dimension of monoids in
the literature. The study of homological  finiteness properties of
monoids should greatly profit from the development of a corresponding
theory of topological finiteness properties of monoids. The central
aim of the present article is to lay the foundations of such a
theory.

For such a theory to be useful in the study of homological finiteness
properties of monoids there are certain properties that any definition
of left-$\Fn$, and left geometric dimension, should certainly
satisfy. Specifically left-$\Fn$ should imply left-$\FPn$, and the
left geometric dimension should provide an upper bound for the left
cohomological dimension. The fundamental question that needs to be
addressed when developing this theory is to determine the correct 
analogue of the $K(G,1)$-complex in the theory for monoids?  There is a natural
notion of classifying space $|BM|$ of a monoid $M$. This is obtained
by viewing $M$ as a one-point category, letting $BM$ denote the nerve
of this category, and setting $|BM|$ as the geometric realisation of
the nerve; see Section~\ref{sec_cspaces} for full details of this
construction. For a group $G$ this space $|BG|$ is a $K(G,1)$-complex,
it is the canonical complex for $G$ mentioned earlier in the
introduction. Since $K(G,1)$s are unique up to homotopy equivalence
the finiteness conditions $\Fn$ and cohomological dimension can all be
defined in terms of existence of CW complexes homotopy equivalent to
$|BG|$ satisfying the appropriate finiteness property. Indeed in group
theory it is a common approach in the study of these topological
finiteness properties to begin with the space $|BG|$ and then seek
transformations on the space which preserve the homotopy equivalence
class, but make the space smaller. This is, for example, the basic
idea behind the theory of collapsing schemes (which will be discussed
in more detail in Section~\ref{sec_collapse}).  It could be regarded
as natural therefore to try define and study topological finiteness
properties of a monoid $M$ in terms of the space $|BM|$. We note that
there is an extensive literature on the study of classifying spaces
$|BM|$ of monoids and related topics; see for instance \cite{
  Fiedorowicz1984, Segal1978, Nunes1995, McCord1969, McDuff1979,
  McDuffSegal1975, LearyNucinkis2001, Puppe1958, Puppe1959,
  KanThurston1976, Hurwitz1989}.

It turns out, however, that using $|BM|$ to define topological finiteness properties of monoids is not the right approach in the sense that it would lead to a definition of $\Fn$ for monoids which does not imply left- or right-$\FPn$, and there are similar issues for the corresponding definition of geometric dimension. This essentially comes does to the fact that the space $|BM|$ does not contain enough information about the monoid $M$ to recover the corresponding homological finiteness properties.

 In more detail, by applying results of MacDuff \cite{McDuff1979} it is possible to show that there are examples of monoids which are not of type left-$\FP_1$ even though $|BM|$ is contractible. For example, if $M$ is an infinite left zero semigroup (a semigroup with multiplication $xy=x$ for all elements $x$ and $y$) with an identity adjoined then by \cite[Lemma~5]{McDuff1979} the space $|BM|$ is contractible while it is straightforward to show that $M$ does not even satisfy the property left-$\FP_1$ (this also follows from Theorem~\ref{t:f1} below). This shows that one should not define property $\Fn$ for monoids using the space $|BM|$. Similar comments apply to attempts to define geometric dimension--if one tries to define geometric dimension using $|BM|$ then if $M$ is any monoid with a left zero element but no right zero element, the left $\mathrm{cd}(M)$ would not equal zero (by Proposition~\ref{p:coh.dim.zero}) while the geometric dimension would be zero.

This issue in fact arose in work of Brown \cite{Brown1989} when he introduced the theory of collapsing schemes. In that paper Brown shows that if a monoid $M$ admits a presentation by a finite complete rewriting system, then $|BM|$ has the homotopy type of a CW complex with only finitely many cells in each dimension. When $M$ is a group this automatically implies that the group is of type $\FPinfty$. Brown goes on to comment

\begin{quote}
\textit{``We would like, more generally, to construct a `small' resolution of this type for any monoid $M$ with a good set of normal forms, not just for groups. I do not know any way to formally deduce such a resolution from the existence of the homotopy equivalence for $|BM|$ above''.}
\end{quote}

As the comments above show, just knowing about the homotopy equivalence class of $|BM|$ will never suffice in order to deduce that the monoid is left- (or right-) $\FPinfty$. It turns out that the correct framework for studying topological finiteness properties of monoids is to pass to the universal bundle $|\EMr|$ over $|BM|$, which has a concrete description as the geometric realisation of the right Cayley graph category of the monoid (this will be explained in detail in Section~\ref{sec_cspaces}). The space $|\EMr|$ is contractable and the monoid has a natural action by left multiplication on this space. This action is free and sends $n$-cells to $n$-cells. It turns out that this space is the correct canonical model of what we shall call a left equivariant classifying space for the monoid. In fact we are able to work in the more general context of projective $M$-sets, and shall define a left equivariant classifying space as a contractable projective $M$-CW complex (see Section~\ref{sec_projM} for the definition of projective $M$-CW complex). The corresponding finiteness conditions left-$\Fn$ and left geometric dimension are then defined in the obvious natural way in terms of the existence of a left equivariant classifying space satisfying appropriate finiteness properties. Consequently, in order to define and study the topological finiteness properties of monoids we are interested in, it will first be necessary for us to develop some of the foundations of $M$-equivariant homotopy theory. Equivariant homotopy theory, and cohomology theory, for groups is an important and well-established area; see for example \cite{May1996} for an introduction. In this way, we are interested in studying homological finiteness properties of monoids by investigating their actions on CW complexes.

The paper is structured as follows.
%
%
The notions of free and projective $M$-CW complexes are introduced in Section~\ref{sec_projM}. For projective $M$-CW complexes we then go on to prove an $M$-equivariant version of HELP (homotopy extension and lifting property), from which we deduce the $M$-equivariant Whitehead theorems. We also apply HELP to prove that every $M$-equivariant continuous mapping of projective $M$-CW complexes is $M$-homotopy equivalent to a cellular one (the cellular approximation theorem). In Section~\ref{sec_base} some useful base change theorems are established. Section~\ref{sec_simplicial} is concerned with monoids acting on simplicial sets. In particular we show how (rigid) $M$-CW complexes arise from (rigid) $M$-simplicial sets by taking the geometric realisation. In Section~\ref{sec_cspaces} we recall the definition of the nerve of a category. By considering three natural categories associated with a monoid $M$, namely the one-point category, the right Cayley graph category, and the two-sided Cayley graph category, we associate three CW complexes with the monoid, denoted $|BM|$, $|\EMr|$ and $|\EMb|$. We go through these constructions in detail in that section. In particular the spaces $|\EMr|$ and $|\EMb|$ are the canonical models of $M$-equivariant (respectively two-sided $M$-equivariant) classifying spaces of the monoid.
In Section~\ref{sec_onesided} we will define left-equivariant, and dually right-equivariant, classifying spaces for a monoid $M$. We prove that such spaces always exist, and that they are unique up to $M$-homotopy equivalence. Using the notion of left-equivariant classifying space we define property left-$\Fn$ for monoids. We prove several fundamental results about this property, including results showing its relationship with property left-$\FPn$, and its connection with the properties of being finitely generated and finitely presented. We prove some results about the closure properties of property left-$\Fn$, in particular results relating the property holding in $M$ to it holding in certain maximal subgroups of $M$.  We also introduce the notion of left geometric dimension of a monoid in this section. We show that it is possible to find an $M$-finite equivariant classifying space for $M$ which is projective when no $M$-finite free equivariant classifying space exists, justifying our choice to work with projective $M$-CW complexes. The geometric dimension is proved to provide an upper bound for the cohomological dimension, and we characterize monoids with left geometric dimension equal to zero.
In Section~\ref{subsec_bi-equi_CS} we introduce the bilateral notion of a classifying space in order to introduce the stronger property of bi-$\F_n$.
We prove results for bi-$\Fn$ analogous to those previously established for left- and right-$\Fn$. In particular we show that bi-$\F_n$ implies bi-$\FP_n$ which is of interest from the point of view of Hochschild cohomology. We also define the geometric dimension as the minimum dimension of a bi-equivariant classifying space and show how it relates to the Hochschild dimension.
In Sections~\ref{sec_collapse} and \ref{sec_MeqCS}  we develop the theory of $M$-equivariant collapsing schemes \cite{Brown1989}. Equivalently, this may be viewed as the development of $M$-equivariant discrete Morse theory in the sense of Forman \cite{Forman2002}.
We show in Section~\ref{sec_guarded} that this theory can be applied to monoids which admit a, so-called, guarded collapsing scheme. Then, in Section~\ref{sec_admitting}, we identify some classes of monoids which admit guarded collapsing schemes, and in particular recover a topological proof of the Anick's theorem, and more generally Kobayashi's theorem, that monoids defined by finite complete rewriting systems are of type bi-$\FPinfty$ (by proving that they are all of type bi-$\Finfty$).  Brown proved Anick's theorem by developing a theory of collapsing schemes (or discrete Morse theory) for chain complexes, which has been rediscovered by other authors~\cite{MorseChain} later on.  Our approach obviates the need to develop a chain complex analogue in this setting.

Applications of the topological approach set out in this paper will be
given in a future article by the current authors
\cite{GraySteinbergInPrep}. Among other things in that paper we shall
show how our topological approach can be used to prove results about
the closure properties of left-$\Fn$ and bi-$\Fn$ for (i) amalgamated
free products of monoids (simplifying and vastly improving on some results of
Cremanns and Otto \cite{Cremanns1998}) (ii) HNN-like extensions in the
sense of Otto and Pride \cite{Pride2004} (in particular establishing
results which generalise \cite[Theorem 1]{Pride2004} to higher
dimensions), and (iii) HNN extensions of the sort considered by Howie
\cite{Howie1963}. For example, we prove that if $A,B$ are monoids containing a common submonoid $C$ such that $A$ and $B$ are free right $C$-set, then if $A$ and $B$ are of type left-$\Fn$ and $C$ is of type left-$\F_{n-1}$, then $A\ast_C B$ is of type left-$\Fn$.  An analogous result is proved for the homological finiteness property $\FPn$, under the weaker hypothesis that $\mathbb ZA$ and $\mathbb ZB$ are flat as $\mathbb ZC$-modules.
Monoid amalgamated products are much more complicated than group ones. For instance, the amalgamated free product of finite monoids can have an undecidable word problem, and the factors do not have to embed or intersect in the base monoid. 

Additionally, we shall give a topological proof that
a free inverse monoid on one or more generators is neither of type
left-$\FP_2$ nor right-$\FP_2$ generalising a classical result of Schein
\cite{Schein1975} that such monoids are not finitely
presented.

Finally, in \cite{GraySteinbergInPrep} we shall apply our methods to prove results about
the homological finiteness properties of special monoids, that is,
monoids defined by finite presentations of the form
$\langle A\mid w_1=1,\ldots,w_k=1\rangle$. We prove that if $M$ is a
special monoid with group of units $G$ then if $G$ is of type $\FPn$
with $1\leq n\leq \infty$, then $M$ is of type left-  and right-$\FPn$; and moreover that
$\mathop{\mathrm{cd}} G\leq \mathop{\mathrm{cd}} M\leq
\max\{2,\mathop{\mathrm{cd}} G\}$.
As a corollary we obtain that all special one-relation monoids are of
type left- and right-$\FPinfty$, answering a special case of a question of Kobayashi
\cite{Kobayashi2000}, and we recover Kobayashi's result that if the
relator is not a proper power then the cohomological dimension is at
most 2.  Specifically we show that if $M$ is a special one-relation
monoid then $M$ is of type left- and right-$\FP_{\infty}$, and if
$M=\langle A\mid w=1\rangle$ with $w$ not a proper power, then
$\mathop{\mathrm{cd}} M\leq 2$; otherwise
$\mathop{\mathrm{cd}} M=\infty$.

\section{Projective $M$-CW complexes and $M$-homotopy theory}
\label{sec_projM}

\subsection{CW complexes and topological finiteness properties of groups}
\label{sec_tfpOfgroups}

For background on CW complexes, homotopy theory, and algebraic topology for group theory, we refer the reader to \cite{GeogheganBook} and \cite{May1999}. Throughout $B^n$ will denote the closed unit ball in $\mathbb{R}^n$, $S^{n-1}$ the $(n-1)$-sphere which is the boundary $\partial B^n$ of the $n$-ball, and $I = [0,1]$ the unit interval. We use $e^n$ to denote an open $n$-cell, homeomorphic to the open $n$ ball $\mathring{B}^n = B^n - \partial B^n$, $\partial e$ denotes the boundary of $e$ and $\bar{e} = cl(e)$ the closure of $e$, respectively.
We identify $I^r = I^r \times 0 \subset I^{r+1}$. 

A \emph{CW complex} is a space $X$ which is a union of subspaces $X_n$ such that, inductively, $X_0$ is a discrete set of points, and $X_{n}$ is obtained from $X_{n-1}$ by attaching balls $B^{n}$ along \emph{attaching maps} $j\colon  S^{n-1} \rightarrow X_{n-1}$.
The resulting maps $B^{n} \rightarrow X_{n}$ are called the \emph{characteristic maps}.
So $X_{n}$ is the quotient space obtained from $X_{n-1} \cup (J_{n} \times B_{n})$ by identifying $(j,x)$ with $j(x)$ for $x \in S^n$, where $J_{n}$ is the discrete set of such attaching maps. Thus $X_n$ is obtained as a pushout of spaces:
\[
\begin{tikzcd}[ampersand replacement=\&]
J_n\times S^{n-1} \ar{r}\ar[d,hook] \& X_{n-1}\ar[d,hook]\\
J_n\times B^n\ar{r} \& X_n.
\end{tikzcd}
\]
The topology of $X$ should be that of the inductive limit $X=\varinjlim X_n$

A CW complex $X$ is then equal as a set to the disjoint union of (open) cells $X = \bigcup_{\alpha}{e_{\alpha}}$ where the $e_\alpha$ are the images of $\mathring{B}^n$ under the characteristic maps. Indeed, an alternative way of defining CW complex, which shall be useful for us to use later on, is as follows.
A CW complex is a Hausdorff space $K$ along with a family $\{e_\alpha\}$ of open cells of various dimensions such that, letting
\[K^j = \bigcup \{ e_{\alpha}\colon\dim{e_{\alpha}} \leq j \},\]
the following conditions are satisfied
\begin{enumerate}
\item[(CW1)] $\displaystyle K = \bigcup_{\alpha}{e_{\alpha}}$ and $e_\alpha \cap e_\beta = \varnothing$ for $\alpha \neq \beta$.
\item[(CW2)] For each cell $e_\alpha$ there is a map $\varphi_\alpha\colon B^n \rightarrow K$ (called the characteristic map) where $B^n$ is a topological ball of dimension $n = \dim{e_\alpha}$, such that

\

	\begin{enumerate}
	
	\item[(a)] $\varphi_\alpha|_{\mathring{B}^n}$ is a homeomorphism onto $e_\alpha$;
	
	\
	
	\item[(b)] $\varphi_\alpha(\partial B^n) \subset K^{n-1}$.
	
	\end{enumerate}
	
	\
	
\item[(CW3)] Each $\overline{e_{\alpha_0}}$ is contained in a union of finitely many $e_\alpha$.

\

\item[(CW4)] A set $A \subset K$ is closed in $K$ if and only if $A \cap \overline{e}_\alpha$ is closed in $\overline{e}_\alpha$ for all $e_\alpha$.
\end{enumerate}

Note that each characteristic map $\varphi\colon B^n \rightarrow K$ gives rise to a characteristic map $\varphi'\colon I^n \rightarrow K$ be setting $\varphi' = \varphi h$ for some homeomorphism $h\colon  I^n \rightarrow B^n$. So we can restrict our attention to characteristic maps with domain $I^n$ when convenient. If $\varphi\colon  B^n \rightarrow K$ is a characteristic map for a cell $e$ then $\varphi|_{\partial B^n}$ is called an attaching map for $e$. A \emph{subcomplex} is a subset $L \subset K$ with a subfamily $\{e_\beta\}$ of cells such that $\displaystyle L = \bigcup e_\beta$ and every $\overline{e_\beta}$ is contained in $L$. If $L$ is a subcomplex of $K$ we write $L < K$ and call $(K,L)$ a CW pair. If $e$ is a cell of $K$ which does not lie in (and hence does not meet) $L$ we write $e \in K - L$. An isomorphism between CW complexes is a homeomorphism that maps cells to cells.

Let $M$ be a monoid.  We shall define notions of free and projective $M$-CW complexes and then use these to study topological finiteness properties of $M$.  The notion of a free $M$-CW complex is a special case of a free $C$-CW complex for a category $C$ considered by Davis and L\"uck in \cite{DavisLuck1998} and so the cellular approximation theorem, HELP Theorem and Whitehead Theorem in this case can be deduced from their results.  The HELP Theorem and Whitehead Theorem for for projective $M$-CW complexes can be extracted with some work from the more general results of Farjoun \cite{DrorFarjounZabrodsky1986} on diagrams of spaces but to keep things elementary and self-contained we present them here	.

\subsection{The category of $M$-sets}

A left $M$-set consists of a set $X$ and a mapping
$M \times X \rightarrow X$ written $(m,x) \mapsto mx$ called a left
action, such that $1x=x$ and $m(nx) = (mn)x$ for all $m,n \in M$ and
$x \in X$. Right $M$-sets are defined dually, they are the same thing
as left $M^{op}$-sets. A \emph{bi-$M$-set} is an $M\times M^{op}$-set.
There is a category of $M$-sets and $M$-equivariant mappings, where
$f\colon X \rightarrow Y$ is $M$-equivariant if $f(mx) = mf(x)$ for all
$x \in X$, $m \in M$.

A (left) $M$-set $X$ is said to be \emph{free} on a set $A$ if there is a mapping $\iota\colon A\to X$ such that for any mapping $f\colon A\to
Y$ with $Y$ an $M$-set, there is a unique $M$-equivariant map $F\colon X\to Y$ such that
\[\begin{tikzcd}
A\ar{r}{\iota}\ar{dr}[swap]{f} & X\ar[dashed]{d}{F}\\
 & Y	
\end{tikzcd}\]
commutes.  The mapping $\iota$ is necessarily injective.  If $X$ is an $M$-set and $A\subseteq X$, then $A$ is a free basis for $X$ if and only if each element of $X$ can be uniquely expressed as $ma$ with $m\in M$ and $a\in A$.

The free left $M$-set on $A$ exists and can be realised as the set
$M \times A$ with action $m(m',a) = (mm',a)$ and $\iota$ is the map
$a \mapsto (1,a)$. Note that if $G$ is a  group, then a left $G$-set $X$ is free if and only if $G$ acts freely on $X$, that is, each element of $X$ has trivial stabilizer.  
In this case, any set of orbit representatives is a basis.

An  $M$-set $P$ is \emph{projective} if any $M$-equivariant surjective mapping $f\colon X\to P$ has an $M$-equivariant section $s\colon P\to X$ with $f\circ s=1_P$.  Free $M$-sets are projective and an $M$-set is projective if and only if it is a retract of a free one.

Each projective $M$-set $P$ is isomorphic to an
$M$-set of the form $\coprod_{a\in A} Me_a$ (disjoint union, which is the coproduct in the category of $M$-sets) with $e_a\in E(M)$.  Here
$E(M)$ denotes the set of idempotents of the monoid $M$.  In
particular, projective $G$-sets are the same thing as free $G$-sets
for a group $G$. (See \cite{Knauer1971} for more details.)

\subsection{Equivariant CW complexes}
A \emph{left $M$-space} is a topological space $X$ with a continuous left action $M\times X\to X$ where $M$ has the discrete topology.  A right $M$-space is the same thing as an $M^{op}$-space and a \emph{bi-$M$-space} is an $M\times M^{op}$-space.  Each $M$-set can be viewed as a discrete $M$-space.
%
%
Note that colimits in the category of $M$-spaces are formed by taking colimits in the category of spaces and observing that the result has a natural $M$-action.

Let us define a (projective) \emph{$M$-cell} of dimension $n$ to be an $M$-space of the form $Me\times B^n$ where $e\in E(M)$ and $B^n$ has the trivial action; if $e=1$, we call it a \emph{free $M$-cell}.  We will define a projective $M$-CW complex in an inductive fashion by imitating the usual definition of a CW complex but by attaching $M$-cells $Me\times B^n$ via $M$-equivariant maps from $Me\times S^{n-1}$ to the $(n-1)$-skeleton.

Formally, a \emph{projective (left) relative $M$-CW complex} is a pair $(X,A)$ of $M$-spaces such that $X=\varinjlim X_n$ with $i_n\colon X_n\to X_{n+1}$ inclusions, $X_{-1}=A$, $X_0 = P_0\cup A$ with $P_0$ a projective $M$-set and where $X_n$ is obtained as a pushout of $M$-spaces
\begin{equation}\label{eq:pushout}
\begin{tikzcd}P_n\times S^{n-1}\ar{r}\ar[d,hook] & X_{n-1}\ar[d,hook]\\ P_n\times B^n\ar{r} & X_n \end{tikzcd}
\end{equation}
with $P_n$ a projective $M$-set and $B^n$ having a trivial $M$-action for $n\geq 1$.  As usual, $X_n$ is called the \emph{$n$-skeleton} of $X$ and if $X_n=X$ and $P_n\neq \emptyset$, then $X$ is said to have \emph{dimension} $n$.
Notice that since $P_n$ is isomorphic to a coproduct of $M$-sets of the form $Me$ with $e\in E(M)$, we are indeed attaching $M$-cells at each step. If $A=\emptyset$, we call $X$ a \emph{projective $M$-CW complex}. Note that a projective $M$-CW complex is a CW complex and the $M$-action is cellular (in fact, takes $n$-cells to $n$-cells).  We can define projective right $M$-CW complexes and projective bi-$M$-CW complexes by replacing $M$ with $M^{op}$ and $M\times M^{op}$, respectively. We say that $X$ is a \emph{free $M$-CW complex} if each $P_n$ is a free $M$-set. If $G$ is a group, a CW complex with a $G$-action is a free $G$-CW complex if and only if $G$ acts freely, cellularly, taking cells to cells, and the setwise stabilizer of each cell is trivial~\cite[Appendix of Section~4.1]{GeogheganBook}.

More generally we define an $M$-CW complex in the same way as above
except that the $P_i$ are allowed to be arbitrary $M$-sets. Most of
the theory developed below is only valid in the projective setting,
but there will be a few occasions (e.g. when we discuss $M$-simplicial
sets) where it will be useful for us to be able to refer to $M$-CW
complexes in general. For future reference we should note here that,
just as for the theory of $G$-CW complexes, there is alternative way
of defining $M$-CW complex in terms of monoids acting on CW
complexes. This follows the same lines as that of groups, see for
example \cite[Section~3.2 and page 110]{GeogheganBook} or
\cite{May1996}. Let $Y$ be a left $M$-space where $M$ is a monoid and
$Y = \bigcup_{\alpha}{e_{\alpha}}$ is a CW complex with characteristic
maps $\varphi_\alpha\colon B^n \rightarrow Y$. We say that $Y$ is a
\emph{rigid left $M$-CW complex} if it is:

\

\noindent $\bullet$ Cellular and dimension preserving: For every $e_\alpha$ and $m \in M$ there exists an
$e_\beta$ such that $m e_\alpha = e_\beta$ and $\mathrm{dim}(e_\beta) = \mathrm{dim}(e_\alpha)$; and

\

\noindent $\bullet$ Rigid on cells: If $me_\alpha = e_\beta$ then $m \varphi_\alpha(k') = \varphi_\beta(k')$ for all $k' \in B^n - \partial B^n$.

\

%
%
%
%
\noindent If the action of $M$ on the set of $n$-cells is free
(respectively projective) then we call $Y$ a free (respectively
projective) rigid left $M$-CW complex. The inductive process described
above for building (projective, free) left $M$-CW complexes is easily
seen to give rise to rigid (projective, free) left $M$-CW complexes,
in the above sense. Conversely every rigid (projective, free) left
$M$-CW complex arises in this way. In other words, the two definitions
are equivalent. For an explanation of this in the case of $G$-CW
complexes see, for example, \cite[page 110]{GeogheganBook}. The proof for
monoids is analogous and is omitted. Similar comments apply for
rigid right $M$-CW complexes and rigid bi-$M$-CW complexes.

%
%
%
%

We say that a projective $M$-CW complex $X$ is of \emph{$M$-finite type} if $P_n$ is a finitely generated projective $M$-set for each $n$ and we say that $X$ is \emph{$M$-finite} if it is finite dimensional and of $M$-finite type (i.e., $X$ is constructed from finitely many $M$-cells).

Notice that if $m\in M$, then $mX$ is a subcomplex of $X$ for all $m\in M$ with $n$-skeleton $mX_n$.  Indeed, $mX_0=mP_0$ is a discrete set of points and $mX_n$ is obtained from $mX_{n-1}$ via the pushout diagam
\[\begin{tikzcd}mP_n\times S^{n-1}\ar{r}\ar[d,hook] & mX_{n-1}\ar[d,hook]\\ mP_n\times B^n\ar{r} & mX_n.\end{tikzcd}\]

A \emph{projective $M$-CW subcomplex} of $X$ is an $M$-invariant subcomplex $A\subseteq X$ which is a union of $M$-cells of $X$.  In other words, each $P_n$ (as above) can be written $P_n=P_n'\coprod P_n''$ with the images of the $P_n'\times B^n$ giving the cells of $A$.  Notice that if $A$ is a projective $M$-CW subcomplex of $X$, then $(X,A)$ can be viewed as a projective relative $M$-CW complex in a natural way. Also note that a cell of $X$ belongs to $A$ if and only if each of its translates do.

A projective $\{1\}$-CW complex is the same thing as a CW complex and $\{1\}$-finite type ($\{1\}$-finite) is the same thing as finite type (finite).

If $e\in E(M)$ is an idempotent and $m\in Me$, then left multiplication by $m$ induces an isomorphism $H_n(\{e\}\times B^n,\{e\}\times S^{n-1})\to H_n(\{m\}\times B^n, \{m\}\times S^{n-1})$ (since it induces a homeomorphism $\{e\}\times B^n/\{e\}\times S^{n-1}\to \{m\}\times B^n/\{m\}\times S^{n-1}$) and so if we choose an orientation for the $n$-cell $\{e\}\times B^n$, then we can give $\{m\}\times B^n$ the orientation induced by this isomorphism.  If $m\in M$ and $m'\in Me$, then the  isomorphism  \[H_n(\{e\}\times B^n,\{e\}\times S^{n-1})\to H_n(\{mm'\}\times B^n, \{mm'\}\times S^{n-1})\] induced by $mm'$ is the composition of the isomorphism \[H_n(\{e\}\times B^n,\{e\}\times S^{n-1})\to H_n(\{m'\}\times B^n, \{m'\}\times S^{n-1})\] induced by $m'$ and the isomorphism \[H_n(\{m'\}\times B^n, \{m'\}\times S^{n-1})\to H_n(\{mm'\}\times B^n, \{mm'\}\times S^{n-1})\] induced by $m$ and so  the action of $m$ preserves orientation.  We conclude that the degree $n$ component of the cellular chain complex for $X$ is isomorphic to $\mathbb ZP_n$ as a $\mathbb ZM$-module and hence is projective (since $\mathbb Z\left[\coprod_{a\in A} Me_a\right]\cong \bigoplus_{a\in A}\mathbb ZMe_a$ and $\mathbb ZM\cong \mathbb ZMe\oplus \mathbb ZM(1-e)$ for any idempotent $e\in E(M)$).

If $X$ is a projective $M$-CW complex then so is $Y=M\times I$ where $I$ is given the trivial action.  If we retain the above notation, then $Y_0=X_0\times \partial I\cong X_0\coprod X_0$.  The $n$-cells for $n\geq 1$ are obtained from attaching $P_n\times B^n\times \partial I\cong (P_n\coprod P_n)\times B^n$ and $P_{n-1}\times B^{n-1}\times I$.  Notice that $X\times \partial I$ is a projective $M$-CW subcomplex of $X\times I$.

If $X,Y$ are $M$-spaces, then an \emph{$M$-homotopy} between $M$-equivariant continuous maps $f,g\colon X\to Y$ is an $M$-equivariant mapping $H\colon X\times I\to Y$ with $H(x,0)=f(x)$ and $H(x,1)=g(x)$ for $x\in X$ where $I$ is viewed as having the trivial $M$-action.  We write $f\simeq_M g$ in this case.  We say that $X,Y$ are \emph{$M$-homotopy equivalent}, written $X\simeq_M Y$, if there are $M$-equivariant continuous mappings (called \emph{$M$-homotopy equivalences}) $f\colon X\to Y$ and $g\colon Y\to X$ with $gf\simeq_M 1_X$ and $fg\simeq_M 1_Y$.  We write $[X,Y]_M$ for the set of $M$-homotopy classes of $M$-equivariant continuous mappings $X\to Y$.

\begin{Lemma}\label{l:adjunction}
Let $X,Y$ be projective $M$-CW complexes and $A$ a projective $M$-CW subcomplex of $X$. Let  $f\colon A\to Y$ be a continuous $M$-equivariant cellular map.  Then the pushout $X\coprod_A Y$ is a projective $M$-CW complex.
\end{Lemma}
\begin{proof}
It is a standard result that $X\coprod_A Y$ is a CW complex whose $n$-cells are the $n$-cells of $Y$ together with the $n$-cells of $X$ not belonging to $A$.  In more detail, let $q\colon X\to X\coprod_A Y$ be the canonical mapping and view $Y$ as a subspace of the adjunction space.
Then the attaching map of a cell  coming from $Y$ is the original attaching map, whereas the attaching map of a cell of $X$ not belonging to $A$ is the composition of $q$ with its original attaching mapping.  It follows from the definition of a projective $M$-CW subcomplex and the construction that $X\coprod_A Y$ is a projective $M$-CW complex. Here it is important that a translate by $M$ of a cell from $X\setminus A$ is a cell of $X\setminus A$.
\end{proof}

A \emph{free $M$-CW subcomplex} of a free $M$-CW complex $X$ is an $M$-invariant subcomplex $A\subseteq X$ which is a union of $M$-cells of $X$.

The proof of Lemma~\ref{l:adjunction} yields the following.

\begin{Lemma}\label{l:adjunction_free_pushout}
Let $X,Y$ be free $M$-CW complexes and $A$ a free $M$-CW subcomplex of $X$. Let  $f\colon A\to Y$ be a continuous $M$-equivariant cellular map.  Then the pushout $X\coprod_A Y$ is a free $M$-CW complex.
\end{Lemma}

A continuous mapping $f\colon X\to Y$ of spaces is an \emph{$n$-equivalence} if
\[f_\ast\colon \pi_q(X,x)\to \pi_q(Y,f(x))\] is a bijection for $0\leq q<n$ and a surjection for $q=n$ where $\pi_0(Z,z)=\pi_0(Z)$ (viewed as a pointed set with base point the component of $z$).  It is a \emph{weak equivalence} if it is an $n$-equivalence for all $n$, i.e., $f_\ast$ is a bijection for all $q\geq 0$.  We will consider a weak equivalence as an $\infty$-equivalence.  We shall see later that an $M$-equivariant weak equivalence of projective $M$-CW complexes is an $M$-homotopy equivalence.

Let $\mathrm{Top}(X,Y)$ denote the set of continuous maps $X\to Y$ for spaces $X,Y$ and $\mathrm{Top}_M(X,Y)$ denote the set of continuous $M$-equivariant maps $X\to Y$ between $M$-spaces $X,Y$.

\begin{Prop}\label{p:cell-maps}
Let $X$ be a space with a trivial $M$-action, $e\in E(M)$ and $Y$ an $M$-space.  Then there is a bijection between $\mathrm{Top}_M(Me\times X,Y)$ and $\mathrm{Top}(X,eY)$. The bijection sends $f\colon Me\times X\to Y$ to $\ov f\colon X\to eY$ given by $\ov f(x) = f(e,x)$ and $g\colon X\to eY$ to $\wh g\colon Me\times X\to Y$ given by $\wh g(m,x) = mg(x)$.
\end{Prop}
\begin{proof}
If $x\in X$, then $\ov f(x)=f(e,x) = f(e(e,x))=ef(e,x)\in eY$.  Clearly, $\ov f$ is continuous.  As $\wh g$ is the composition of $1_{Me}\times g$ with the action map, it follows that $\wh g$ is continuous. We show that the two constructions are mutually inverse.  First we check that
\[\wh{\ov f}(m,x) = m\ov f(x)=mf(e,x) = f(m(e,x))=f(me,x)=f(m,x)\] for $m\in Me$ and $x\in X$.  Next we compute that
\[\ov{\wh{g}}(x) = \wh g(e,x) = eg(x)=g(x)\] since $g(x)\in eY$.  This completes the proof.
\end{proof}

Proposition~\ref{p:cell-maps} is the key tool to transform statements about projective $M$-CW complexes into statement about CW complexes.  We shall also need the following lemma relating equivariant $n$-equivalences and $n$-equivalences.

\begin{Lemma}\label{l:relate.equiv}
Let $Y,Z$ be $M$-spaces and let $k\colon Y\to Z$ be an $M$-equivariant $n$-equivalence with $0\leq n\leq \infty$.  Let $e\in E(M)$ and $k'=k|_{eY}\colon eY\to eZ$.  Then $k'$ is an $n$-equivalence.
\end{Lemma}
\begin{proof}
First note that $k(ey) =ek(y)$ and so $k|_{eY}$ does indeed have image contained in $eZ$.  Let $y\in eY$ and $q\geq 0$. Let $\alpha\colon eY\to Y$ and $\beta\colon eZ\to Z$ be the inclusions.  Then note that the action of $e$ gives retractions $Y\to eY$ and $Z\to eZ$.  Hence we have a commutative diagram
\[\begin{tikzcd}\pi_q(Y,y)\ar{r}{k_*}\ar[transform canvas={xshift=0.7ex}]{d}{e_*} & \pi_q(Z,k(y))\ar[transform canvas={xshift=0.7ex}]{d}{e_*}\\ \pi_q(eY,y)\ar[transform canvas={xshift=-0.7ex}]{u}{\alpha_*}\ar{r}{k'_*} &\pi_q(eZ,k(y))\ar[transform canvas={xshift=-0.7ex}]{u}{\beta_*} \end{tikzcd}\] with $e_*\alpha_*$ and $e_*\beta_*$ identities.  Therefore, if $k_*$ is surjective, then $k'_*$ is surjective and if $k_*$ is injective, then $k'_*$ is injective.  The lemma follows.
\end{proof}

\subsection{Whitehead's theorem}

With Lemma~\ref{l:relate.equiv} in hand, we can prove an $M$-equivariant version of HELP (homotopy extension and lifting property)~\cite[Page~75]{May1999}, which underlies most of the usual homotopy theoretic results about CW complexes. If $X$ is a space, then $i_j\colon X\to X\times I$, for $j=0,1$, is defined by $i_j(x)=(x,j)$.

\begin{Thm}[HELP]\label{t:HELP}
Let $(X,A)$ be a projective relative $M$-CW complex of dimension at most $n\in \mathbb N\cup \{\infty\}$ and $k\colon Y\to Z$ an $M$-equivariant $n$-equivalence of $M$-spaces.  Then given $M$-equivariant continuous mappings $f\colon X\to Z$, $g\colon A\to Y$ and $h\colon A\times I\to Z$ such that $kg=hi_1$ and $fi=hi_0$ (where $i\colon A\to X$ is the inclusion), there exist $M$-equivariant continuous mappings $\til g$ and $\til h$ making the diagram
\[\begin{tikzcd}A\ar{rr}{i_0}\ar[hook]{dd}[swap]{i} & & A\times I\ar{dl}[swap]{\ov h}\ar[hook]{dd} & & A\ar{ll}[swap]{i_1}\ar[hook]{dd}{i}\ar{dl}{\ov g}\\
                              &Z&           &Y\ar[crossing over]{ll}[swap]{k}&   \\
                X\ar{rr}{i_0}\ar{ur}{\ov f}             & & X\times I\ar[dashrightarrow]{ul}[swap]{\til h} & & X\ar{ll}[swap]{i_1}\ar[dashrightarrow]{ul}[swap]{\til g}\end{tikzcd}\]
commute.
\end{Thm}
\begin{proof}
Proceeding by induction on the skeleta and adjoining an $M$-cell at a time, it suffices to handle the case that \[(X,A)=(Me\times B^q,Me\times S^{q-1})\] with $0\leq q\leq n$.  By Proposition~\ref{p:cell-maps} it suffices to find continuous mappings   $\til g$ and $\til h$ making the diagram
\[\begin{tikzcd}S^{q-1}\ar{rr}{i_0}\ar[hook]{dd}[swap]{i} & & S^{q-1}\times I\ar{dl}[swap]{h}\ar[hook]{dd} & & S^{q-1}\ar{ll}[swap]{i_1}\ar[hook]{dd}{i}\ar{dl}{g}\\
                              &eZ&           &eY\ar[crossing over]{ll}[swap]{k}&   \\
                B^q\ar{rr}{i_0}\ar{ur}{f}             & & B^q\times I\ar[dashrightarrow]{ul}[swap]{\til h} & & B^q\ar{ll}[swap]{i_1}\ar[dashrightarrow]{ul}[swap]{\til g}\end{tikzcd}\]
commute where we have retained the notation of Proposition~\ref{p:cell-maps}.  The mapping $k\colon eY\to eZ$ is an $n$-equivalence by Lemma~\ref{l:relate.equiv} and so we can apply the usual HELP theorem~\cite[Page~75]{May1999} for CW complexes to deduce the existence of $\til g$ and $\til h$.  This completes the proof.
\end{proof}

As a consequence we may deduce the $M$-equivariant Whitehead theorems.

\begin{Thm}[Whitehead]\label{t:whitehead.vers1}
If $X$ is a projective $M$-CW complex and $k\colon Y\to Z$ is an $M$-equivariant $n$-equivalence of $M$-spaces, then the induced mapping $k_*\colon [X,Y]_M\to [X,Z]_M$ is a bijection if $\dim X<n$ or $n=\infty$ and a surjection if $\dim X=n<\infty$.
\end{Thm}
\begin{proof}
For surjectivity we apply Theorem~\ref{t:HELP} to the pair $(X,\emptyset)$.  If $f\colon X\to Z$, then $\til g\colon X\to Y$ satisfies $kg\simeq_M f$.  For injectivity, we apply Theorem~\ref{t:HELP} to the pair $(X\times I,X\times \partial I)$ and note that $X\times I$ has dimension one larger than $X$.  Suppose that $p,q\colon X\to Y$ are such that $kp\simeq_M kq$ via a homotopy $f\colon X\times I\to Z$. Put $g=p\coprod q\colon X\times \partial I\to Y$ and define $h\colon X\times \partial I\times I\to Z$ by $h(x,s,t)=f(x,s)$.  Then $\til g\colon X\times I\to Y$ is a homotopy between $p$ and $q$.
\end{proof}

\begin{Cor}[Whitehead]\label{c:whitehead}
If $k\colon Y\to Z$ an $M$-equivariant weak equivalence ($n$-equivalence) between projective $M$-CW complexes (of dimension less than $n$), then $k$ is an $M$-homotopy equivalence.
\end{Cor}
\begin{proof}
Under either hypothesis, $k_*\colon [Z,Y]_M\to [Z,Z]_M$ is a bijection by Theorem~\ref{t:whitehead.vers1} and so $kg\simeq_M 1_Z$ for some $M$-equivariant $g\colon Z\to Y$.  Then $kgk\simeq_M k$ and hence, since $k_*\colon [Y,Y]\to [Y,Z]$ is a bijection by Theorem~\ref{t:whitehead.vers1}, we have that $gk\simeq_M 1_Y$.  This completes the proof.
\end{proof}

\subsection{Cellular approximation}
Our next goal is to show that every $M$-equivariant continuous mapping of projective $M$-CW complexes is $M$-ho\-mo\-to\-py equivalent to a cellular one. We shall need the well-known fact that if $Y$ is a CW complex, then the inclusion $Y_n\hookrightarrow Y$ is an $n$-equivalence for all $n\geq 0$~\cite[Page~76]{May1999}.

\begin{Thm}[Cellular approximation]\label{t:cell.approx}
Let $f\colon X\to Y$ be a continuous  $M$-equivariant mapping with $X$ a projective $M$-CW complex and $Y$ a CW complex with a continuous action of $M$ by cellular mappings.  Then $f$ is $M$-homotopic to a continuous $M$-equivariant cellular mapping. Any two cellular approximations are homotopy equivalent via a cellular $M$-homotopy.
\end{Thm}
\begin{proof}
We prove only the first statement.  The second is proved using a relative version of the first whose statement and proof we omit.  Note that $Y_n$ is $M$-equivariant for all $n\geq 0$ because $M$ acts by cellular mappings.  We construct by induction $M$-equivariant continuous mappings $f_n\colon X_n\to Y_n$ such that $f|_{X_n}\simeq_M f_n|_{X_n}$ via an $M$-homotopy $h_n$ and $f_n,h_n$ extend $f_{n-1}$, $h_{n-1}$, respectively (where we take $X_{-1}=\emptyset$).  We have, without loss of generality, $X_0=\coprod_{a\in A} Me_a$. Since $e_aY$ is a subcomplex of $Y$ with $0$-skeleton $e_aY_0$ and $f(e_a)\in e_aY$, we can find a path $p_a$ in $e_aY$ from $f(e_a)$ to an element $y_a\in e_aY_0$.  Define $f_0(me_a)=my_a$ and $h_0(me_a,t) = mp_a(t)$, cf.~Proposition~\ref{p:cell-maps}.

Assume now that $f_n,h_n$ have been defined.  Since the inclusion $Y_{n+1}\to Y$ is an $M$-equivariant $(n+1)$-equivalence, Theorem~\ref{t:HELP} gives a commutative diagram
\[\begin{tikzcd}X_n\ar{rr}{i_0}\ar[hook]{dd}[swap]{i} & & X_n\times I\ar{dl}[swap]{h_n}\ar[hook]{dd} & & X_n\ar{ll}[swap]{i_1}\ar[hook]{dd}{i}\ar{dl}{f_n}\\
                              &Y&           &Y_{n+1}\ar[crossing over]{ll}&   \\
                X_{n+1}\ar{rr}{i_0}\ar{ur}{f}             & & X_{n+1}\times I\ar[dashrightarrow]{ul}[swap]{h_{n+1}} & & X_{n+1}\ar{ll}[swap]{i_1}\ar[dashrightarrow]{ul}[swap]{f_{n+1}}\end{tikzcd}\]
 thereby establishing the inductive step.
We obtain our desired cellular mapping and $M$-homotopy by taking the colimit of the $f_n$ and $h_n$.
\end{proof}

\section{Base change}
\label{sec_base}
If $A$ is a right $M$-set and $B$ is a left $M$-set, then $A\otimes_M B$ is the quotient of $A\times B$ by the least equivalence relation $\sim$ such that $(am,b)\sim (a,mb)$ for all  $a\in A$, $b\in B$ and $m\in M$.  We write $a\otimes b$ for the class of $(a,b)$ and note that the mapping $(a,b)\mapsto a\otimes b$ is universal for mappings $f\colon A\times B\to X$ with $X$ a set and $f(am,b)=f(a,mb)$.  If $M$ happens to be a group, then $M$ acts on $A\times B$ via $m(a,b)=(am^{-1},mb)$ and $A\otimes_M B$ is just the set of orbits of this action. The tensor product $A\otimes_M()$ preserves all colimits because it is a left adjoint to the functor $X\mapsto X^A$.

If $B$ is a left $M$-set there is a natural preorder relation $\leq$ on $B$ where $x \leq y$ if and only if $Mx \subseteq My$. Let $\approx$ denote the symmetric-transitive closure of $\leq$. That is, $x \approx y$ if there is a sequence $z_1, z_2, \ldots, z_n$ of elements of $B$ such that for each $0 \leq i \leq n-1$ either $z_i \leq z_{i+1}$ or $z_i \geq z_{i+1}$. This is clearly an equivalence relation and we call the $\approx$-classes of $B$ the \emph{weak orbits} of the $M$-set. This corresponds to the notion of the weakly connected components of a directed graph. If $B$ is a right $M$-set then we use $B/M$ to denote the set of weak orbits of the $M$-set. Dually, if $B$ is a left $M$-set we use $M\backslash B$ to denote the set of weak orbits. Note that if $1$ denotes the trivial right $M$-set and $B$ is  a left $M$-set, then we have $M\backslash B=1\otimes_M B$.

Let $M,N$ be monoids.  An \emph{$M$-$N$-biset} is an $M\times N^{op}$-set.
If $A$ is an $M$-$N$-biset and $B$ is a left $N$-set, then the equivalence relation defining $A\otimes_N B$ is left $M$-invariant and so $A\otimes_N B$ is a left $M$-set with action $m(a\otimes b) = ma\otimes b$.

\begin{Prop}\label{p:base.change.proj}
Let $A$ be an $M$-$N$-biset that is (finitely generated) projective as an $M$-set and let $B$ be a (finitely generated) projective $N$-set.  Then $A\otimes_N B$ is a (finitely generated) projective $M$-set.
\end{Prop}
\begin{proof}
As $B$ is a (finite) coproduct of $N$-sets $Ne$ with $e\in E(N)$, it suffices to handle the case $B=Ne$.  Then $A\otimes_N Ne\cong Ae$ via $a\otimes n\mapsto an$ with inverse $a\mapsto a\otimes e$ for $a\in Ae$.  Now define $r\colon A\to Ae$ by $r(a)=ae$.  Then $r$ is an $M$-equivariant retraction.  So $Ae$ is a retract of a (finitely generated) projective and hence is a (finitely generated) projective.
\end{proof}

If $X$ is a left $M$-space and $A$ is a right $M$-set, then $A\otimes_M X$ is a topological space with the quotient topology.  Again the functor $A\otimes_M ()$ preserves all colimits.  In fact, $A\otimes_M X$ is the coequalizer in the diagram
\[\coprod_{A\times M}X\rightrightarrows \coprod_{A} X\rightarrow A\otimes_M X\]
where the top map sends $x$ in the $(a,m)$-component to $mx$ in the $a$-component and the bottom map sends $x$ in the $(a,m)$-component to $x$ in the $am$-component.

\begin{Cor}\label{c:base.change.cw}
If $A$ is an $M$-$N$-biset that is projective as an $M$-set and $X$ is a projective $N$-CW complex, then $A\otimes_N X$ is a projective $M$-CW complex.  If $A$ is in addition finitely generated as an $M$-set and $X$ is of $N$-finite type, then $A\otimes_N X$ is of $M$-finite type.  Moreover, $\dim A\otimes_N X=\dim X$.
\end{Cor}
\begin{proof}
Since $A\otimes_N()$ preserves colimits, $A\otimes_N X=\varinjlim A\otimes_N X_n$.  Moreover, putting $X_{-1}=\emptyset$, we have that if $X_n$ is obtained as per the pushout square \eqref{eq:pushout}, then $A\otimes_N X_n$ is obtained from the pushout square
\[\begin{tikzcd}(A\otimes_N P_n)\times S^{n-1}\ar{r}\ar[d,hook] & A\otimes_N X_{n-1}\ar[d,hook]\\ (A\otimes_N P_n)\times B^n\ar{r} & A\otimes_N X_n \end{tikzcd}\] by preservation of colimits and the observation that if $C$ is a trivial left $N$-set and $B$ is a left $N$-set, then $A\otimes_N (B\times C)\cong (A\otimes_N B)\times C$ via $a\otimes (b,c)\mapsto (a\otimes b,c)$.  The result now follows from Proposition~\ref{p:base.change.proj}
\end{proof}

By considering the special case where $M$ is trivial and $A$ is a singleton, and observing that a projective $M$-set $P$ is finitely generated if and only if $M\backslash P$ is finite, we obtain the following corollary.

\begin{Cor}\label{c:quotient}
Let $X$ be a projective $M$-CW complex.  Then $M\backslash X$ is a CW complex.  Moreover, $X$ is $M$-finite (of $M$-finite type) if and only if $M\backslash X$ is finite (of finite type).
\end{Cor}

The following observation will be used many times.

\begin{Prop}\label{p:components}
Let $X$ be a locally path connected $N$-space and $A$ an $M$-$N$-biset.  Then $\pi_0(X)$ is an $N$-set and $\pi_0(A\otimes_N X)\cong A\otimes_N\pi_0(X)$.
\end{Prop}
\begin{proof}
Note that the functor $X\mapsto \pi_0(X)$ is left adjoint to the inclusion of the category of $N$-sets into the category of locally path connected $M$-spaces and hence it preserves all colimits.  The result now follows from the description of tensor products as coequalizers of coproducts.
\end{proof}

The advantage of working with $M$-homotopies is that they behave well under base change.

\begin{Prop}\label{p:preservation.hom.equiv}
Let $A$ be an $M$-$N$-biset and let $X,X'$ be $N$-homotopy equivalent $N$-spaces. Then $A\otimes_N X'$ is $M$-homotopy equivalent to $A\otimes_N X'$.  In particular, if $Y,Z$ are $M$-spaces and $Y\simeq_M Z$, then $M\backslash Y\simeq M\backslash Z$.
\end{Prop}
\begin{proof}
It suffices to prove that if $Y,Z$ are $N$-spaces and $f,g\colon Y\to Z$ are $N$-homotopic $N$-equivariant maps, then
\[A\otimes_N f,A\otimes_N g\colon  A\otimes_N Y\to A\otimes_NZ\] are $M$-homotopic.  This follows immediately from the identification of $A\otimes_N(Y\times I)$ with  $(A\otimes_N Y)\times I$.  For if $H\colon Y\times I\to Z$ is an $N$-homotopy between $f$ and $g$, then $A\otimes_N H$ provides the $M$-homotopy between $A\otimes_N f$ and $A\otimes_N g$.
\end{proof}

The following base change lemma, and its dual, is convenient for dealing with bisets.

\begin{Lemma}\label{l:factor.out}
Let $A$ be an $M\times M^{op}$-set and consider the right $M\times M^{op}$-set $M$ with the right action $m(m_L,m_R)=mm_L$.  Then $A/M$ is a left $M$-set and there is an $M$-equivariant isomorphism  $A/M\to M\otimes_{M\times M^{op}} A$.
\end{Lemma}
\begin{proof}
Clearly, $A/M=A\otimes_{M} 1$ is a left $M$-set.  Write $[a]$ for the class of $a$ in $A/M$.  Define $f\colon A/M\to M\otimes_{M\times M^{op}} A$ by $f([a]) = 1\otimes a$.  This is well defined and $M$-equivariant because if $a\in A$ and $m\in M$, then $1\otimes am=1\otimes (1,m)a=1(1,m)\otimes a=1\otimes a$ and $1\otimes ma = 1\otimes (m,1)a=1(m,1)\otimes a=m\otimes a$.  Define $G\colon M\times A\to A/M$ by $G(m,a)  = [ma]$.  If $m,m_L,m_R\in M$, then $G(m(m_L,m_R),a) = G(mm_L,a) = [mm_La]$ and $G(m,(m_L,m_R)a) = [mm_Lam_R]=[mm_La]$.  Therefore, $G$ induces a well defined mapping $g\colon M\otimes_{M\times M^{op}}A\to A/M$.  Then we check that $gf([a])=g(1\otimes a)=[a]$ and $fg(m\otimes a) = f([ma])=1\otimes ma=1\otimes (m,1)a=1(m,1)\otimes a=m\otimes a$.  Thus $f$ and $g$ are inverse isomorphisms.
\end{proof}

The following basic result will be used later.

\begin{Prop}\label{p:easy.tensor}
Let $G$ be a group. Then
$G\times G$ is a $(G\times G^{op})$-$G$-biset that is free as right $G$-set on cardinality of $G$ generators under the right action $(g,g')h=(gh,h^{-1}g')$.
\end{Prop}
\begin{proof}
It is easy to check that right action of $G$ is indeed an action commuting with the left action of $G\times G^{op}$.  Moreover, the right action of $G$ is free and two elements $(g_1,g_2)$ and $(g_1',g_2')$ are in the same right $G$-orbit if and only if $g_1g_2=g_1'g_2'$.  This completes the proof.
\end{proof}

\begin{Cor}\label{c:one-sided}
Let $M$ be a monoid and $X$  a projective $M\times M^{op}$-CW complex.
\begin{enumerate}
\item $X/M$ is a projective $M$-CW complex and $M\backslash X$ is a projective $M^{op}$-CW complex.
\item   If $X$ is of $M\times M^{op}$-finite type, then $X/M$ is of $M$-finite type and dually for $M\backslash X$.
\item  $\dim X/M=\dim X=\dim M\backslash X$.
\item If $X,Y$ are $M\times M^{op}$-homotopic projective $M\times M^{op}$-CW complexes, then $X/M$ and $Y/M$  (respectively, $M\backslash X$ and $M\backslash Y$) are $M$-homotopic projective $M$-CW complexes (respectively, $M^{op}$-homotopic projective $M^{op}$-CW complexes).
\end{enumerate}
\end{Cor}
\begin{proof}
The first three items  follow from  Corollary~\ref{c:base.change.cw} and Lemma~\ref{l:factor.out} (and their duals).  The final statement follows from Lemma~\ref{l:factor.out} and Proposition~\ref{p:preservation.hom.equiv}.
\end{proof}

We shall frequently use without comment that if $A$ is an $M$-$N$-biset and $B$ is an $N$-set, then $\mathbb Z[A\otimes_N B]\cong \mathbb ZA\otimes_{\mathbb ZN}\mathbb ZB$ as left $\mathbb ZM$-modules.  Indeed, there are natural isomorphisms of abelian groups
\begin{align*}
\Hom_{\mathbb ZM}(\mathbb ZA\otimes_{\mathbb ZN} \mathbb ZB,V)&\cong \Hom_{\mathbb ZN}(\mathbb ZB,\Hom_{ZM}(\mathbb ZA,V))\\ &\cong \Hom_N(B,\Hom_M(A,V))\\ &\cong \Hom_M(A\otimes_N B,V)
\\ &\cong \Hom_{\mathbb ZM}(\mathbb Z[A\otimes_N B],V)
\end{align*}
 for a $\mathbb ZM$-module $V$ and so we can apply Yoneda's Lemma.

\section{Simplicial sets}\label{sec_simplicial}
An important source of examples of rigid $M$-CW complexes will come
from simplicial sets which admit suitable monoid actions. In this
section we introduce the notion of a rigid $M$-simplicial set, and
we show how these give rise to rigid $M$-CW complexes via the geometric
realisation functor. For further background on simplicial sets we
refer the reader to \cite[Chapter~8]{Weibel1994} or \cite{MaySimplicialObjects}.


%
%
%
%
Let $\Delta$ denote the \emph{simplicial category}. It has objects all
the finite linearly ordered sets $[n] = \{0,1,\ldots, n-1\}$
($n \geq 0$) and morphisms given by (non-strictly) order-preserving
maps. A \emph{simplicial set} $X$ is then a functor
$X\colon \Delta^{op} \rightarrow {\emph{\textbf{Set}}}$ from $\Delta^{op}$
to the category of sets. For each $n$, the image of $[n]$ under $X$ is
denoted $X_n$ and is called the \emph{set of $n$-simplicies} of the
simplicial set. Any simplicial set $X$ may be defined combinatorially
as a collection of sets $X_n$ ($n \geq 0$) and functions $d_i\colon X_n \rightarrow X_{n-1}$ and
$s_i\colon X_n \rightarrow X_{n+1}$ $(0 \leq i \leq n)$
satisfying
\begin{align*}
& d_i d_j = d_{j-1} d_i \quad (i<j) \\
& s_i s_j = s_{j+1} s_i \quad (i \leq j) \\
& d_i s_j =
\begin{cases}
1 & i=j, \; j+1 \\
s_{j-1} d_i & i<j \\
s_j d_{i-1} & i>j+1.
\end{cases}
\end{align*}
Here the $d_i$ are called the \emph{face maps} and the $s_i$ are
called the \emph{degeneracy maps}.  We say that an $n$-simplex
$x \in X_n$ is \emph{degenerate} if it is the image of some degeneracy
map.

A \emph{simplicial morphism} $f\colon X \rightarrow Y$ between simplicial
sets is a natural transformation between the corresponding functors, i.e., a sequence of functions $f_n\colon X_n \rightarrow Y_n$ for each
$n \geq 0$ such that $f_{n-1} d_i = d_i f_n$ and
$f_{n+1} s_j = s_j f_n$.
%
%
%
%
There is a functor $| \cdot |\colon  \mathbf{SSet} \rightarrow \mathbf{\mathcal{CG}}$,
called the \emph{geometric realization functor}, from the category
$\mathbf{SSet}$ of simplicial sets and the category $\mathbf{\mathcal{CG}}$ of
compactly-generated Hausdorff topological spaces. Let
$K = \bigcup_{i \geq 0} K_i$ be a simplicial set with degeneracy and
face maps $d_i$, $s_i$. The \emph{geometric realisation} $|K|$ of $K$
is the CW complex constructed from $K$ in the following way. Let
\[
\Delta_n =
\left\{
(t_0,\ldots,t_n): 0 \leq t_i \leq 1, \sum t_i = 1
\right\}
\subseteq \mathbb{R}^{n+1}
\]
denote the standard topological $n$-simplex. Define
\[
\begin{array}{cccc}
  \delta_i\colon & \Delta_{n-1} & \rightarrow & \Delta_n \\
            & (t_0,\ldots,t_{n-1}) & \mapsto & (t_0,\ldots,t_{i-1},0,t_i,\ldots,t_{n-1}),
\end{array}
\]
and
\[
\begin{array}{cccc}
  \sigma_i\colon & \Delta_{n+1} & \rightarrow & \Delta_n \\
            & (t_0,\ldots,t_{n+1}) & \mapsto & (t_0,\ldots,t_i + t_{i+1},\ldots,t_{n-1}).
\end{array}
\]
Then
\[
|K| = \left( \bigsqcup_{n \geq 0} K_n \times \Delta_n \right) /{\sim}
\]
where $\sim$ is the equivalence relation generated by
\[
(x,\delta_i(u)) \sim (d_i(x),u), \quad
(x,\sigma_i(v)) \sim (s_i(x),v).
\]
We give
\[
\left( \bigsqcup_{0 \leq n \leq  q} K_n \times \Delta_n \right) /{\sim}
\]
the quotient topology for all $q$ and take the inductive limit of the
resulting topologies. The geometric realisation $|K|$ is a CW complex
whose cells are in natural bijective correspondence with the
non-degenerate simplicies of $K$. To see this, write
\[
\overline{K} = \bigsqcup_{n \geq 0} K_n \times \Delta_n.
\]
Then a point $(k,x) \in \overline{K}$ is called \emph{non-degenerate}
if $k$ is a non-degenerate simplex and $x$ is an interior point.  The following is~\cite[Lemma~3]{Milnor1957}.

\begin{Lemma}
  Each point $(k, x) \in \overline{K}$ is $\sim$-equivalent to a
  unique non-degenerate point.
\end{Lemma}
In each case, the point in question is determined by the maps
$\delta_i, d_i, \sigma_i$ and $s_i$ (see \cite{Milnor1957} for
details). This lemma is the key to proving that that $|K|$ is a
CW complex: we take as $n$-cells of $|K|$ the images of the
non-degenerate $n$-simplices of $\overline{K}$, and the above lemma
shows that the interiors of these cells partition $|K|$. The remaining
properties of a CW complex are then easily verified. The following
lemma shows that geometric realisation defines a functor from
$\mathbf{SSet}$ to $\mathbf{\mathcal{CG}}$.

The next result is~\cite[Lemma~4]{Milnor1957}.

\begin{Lemma}
\label{lem_simplicial}
If $K = \bigcup K_i$ and $L = \bigcup L_i$ are simplicial sets and
$f\colon K \rightarrow L$ is a simplicial morphism then $\overline{f}$ given
by
\[
\overline{f}_n\colon K_n \times \Delta_n \rightarrow L_n \times \Delta_n, \quad
(x,u) \mapsto (f(x),u)
\]
is continuous, and induces a well-defined continuous map
\[
|f|\colon |K| \rightarrow |L|, \quad
(x,u)/{\sim} \ \mapsto \ (f(x),u)/{\sim}
\]
of the corresponding geometric realizations, which is cellular.
\end{Lemma}

%
%
%
%
A left $M$-simplicial set is a simplicial set equipped with a left action of
$M$ by simplicial morphisms. In order to construct rigid $M$-CW
complexes we shall need the following special kind of $M$-simplicial
set.
\begin{definition}[Rigid $M$-simplicial set]
  Let $K = \bigcup_{i \geq 0} K_i$ be a simplicial set with degeneracy
  and face maps $d_i$, $s_i$, and let $M$ be a monoid. We call
  $K$ a \emph{rigid left $M$-simplicial set} if $K$ comes equipped
  with an action of $M \times K \rightarrow K$ such that
\begin{itemize}
\item $M$ is acting by simplicial morphisms, i.e.,  $M$ maps
  $n$-simplicies to $n$-simplicies, and commutes with $d_i$ and $s_i$;
\item $M$ preserves non-degeneracy, i.e., for every non-degenerate
  $n$-simplex $x$ and every $m \in M$ the $n$-simplex $mx$ is also
  non-degenerate.
\end{itemize}
\end{definition}
A \emph{rigid right $M$-simplicial set} is defined dually, and a
\emph{rigid bi-$M$-simplicial set} is simultaneously both a left and
a right $M$-simplicial set, with commuting actions. A
bi-$M$-simplicial set is the same thing as a left $M \times
M^{op}$-simplicial set.
Note that it follows from the condition that $M$ acts by simplicial
morphisms that, under the action of $M$,  degenerate $n$-simplicies are sent to
to degenerate $n$-simplicies.
%
%
%
%
%
%
The geometric realisation construction defines a functor from the
category of left $M$-simplicial sets (with $M$-equivariant simplicial
morphisms) to the category of left $M$-spaces. In particular, this
functor associates with each rigid left $M$-simplicial set a rigid
$M$-CW complex. Corresponding statements hold for both rigid right and
bi-$M$-simplicial sets.
\begin{Lemma}\label{lem_induced}
  For any rigid left $M$-simplicial set $K = \bigcup_{i \geq 0} K_i$
  the geometric realisation $|K|$ is a rigid left $M$-CW complex with
  respect to the induced action given by
\[
m \cdot [(x,u)/{\sim}] = (m \cdot x,u)/{\sim}.
\]
\end{Lemma}
\begin{proof}
  It follows from Lemma~\ref{lem_simplicial} that the action is
  continuous.  By the definition of rigid left $M$-simplicial set the
  $M$-action maps non-degenerate simplices to non-degenerate
  simplices, and the cells of $|K|$ are in natural bijective
  correspondence with the non-degenerate simplicies of $K$.  It
  follows that the action of $M$ on $|K|$ sends $n$-cells to
  $n$-cells. The action is rigid by definition.  Thus $|K|$ is a rigid
  $M$-CW complex.
\end{proof}
There are obvious right-  and bi-$M$-simplicial set analogues of
Lemma~\ref{lem_induced} obtained by replacing $M$ by $M^{op}$ and $M
\times M^{op}$, respectively.

\section{Standard constructions of projective
  $M$-CW complexes}\label{sec_cspaces}

In this section we shall give a fundamental method that, for any
monoid $M$, allows us to construct in a canonical way free left-, right- and
bi-$M$-CW complexes. These constructions will be important when we go on
to discuss $M$-equivariant classifying spaces later on in the
article. Each of the constructions in this section is a special case
of the general notion of the nerve of a category.

To any (small) category $C$ we can associate a simplicial
set $N(C)$ called the \emph{nerve} of the category. For each
$k \geq 0$ we let $N(C)_k$ denote the set of all sequences $(f_1,\ldots,f_k)$ composable arrows
\begin{equation}\label{eq_path}
A_0 \xrightarrow{f_1}
A_1 \xrightarrow{f_2}
\cdots
\xrightarrow{f_k} A_k
\end{equation}
where we allow objects to repeat in these sequences. The objects of $C$ are the $0$-simplices.  The face map
$d_i\colon N(C)_k \rightarrow N(C)_{k-1}$ omits $A_i$, so it carries the above sequence to
\[
A_0 \xrightarrow{f_1}
A_1 \xrightarrow{f_2}
\cdots
\xrightarrow{f_{i-1}} A_{i-1}
\xrightarrow{f_{i+1} \circ  f_{i}}
A_{i+1}
\xrightarrow{f_{i+2}}
\cdots
\xrightarrow{f_k} A_k
\]
while the degeneracy map $s_i\colon N(C)_k \rightarrow N(C)_{k+1}$ carries it to
\[
A_0 \xrightarrow{f_1}
A_1 \xrightarrow{f_2}
\cdots
\xrightarrow{f_{i}} A_{i}
\xrightarrow{\mathrm{id}_{A_i}} A_{i}
\xrightarrow{f_{i+1}}
A_{i+1}
\xrightarrow{f_{i+2}}
\cdots
\xrightarrow{f_k} A_k
\]
The \emph{classifying space} of a (small) category $C$ is the
geometric realisation $|N(C)|$ of the nerve $N(C)$ of $C$.

The nerve is a functor from $\mathrm{\mathbf{Cat}}$ (the category of small categories) to
$\mathrm{\mathbf{SSet}}$ (the category of simplicial sets, with
simplicial morphisms) given by applying the functor to the diagram \eqref{eq_path}.
From this it follows that a functor between
small categories $C$ and $D$ induces a map of simplicial sets
$N(C) \rightarrow N(D)$, which in turn induces a continous map between
the classifying spaces $|N(C)| \rightarrow |N(D)|$. Also, a natural
transformation between two functors between $C$ and $D$ induces a
homotopy between the induced maps on the classifying spaces. In
particular, equivalent categories have homotopy equivalent classifying
spaces. Any functor which is left or right adjoint induces a homotopy
equivalence of nerves. Consequently, $|N(C)|$ is contractible if $C$
admits an initial or final object.
(For a proof of this see \cite[Corollary~3.7]{SrinivasBook}.)

%
%

It is obvious from the nerve
construction that the nerve of a category which is not connected is
the disjoint union of the nerves of the connected components of the
category. Thus, if every component of $C$ admits an initial or final
object, then each of the components of $|N(C)|$ will be contractible.

It is well know that the geometric realisations of the nerve of a  category $C$ and its reversal $C^{op}$ are homeomorphic.

\subsection{The classifying space $|BM|$ of a monoid $M$.}
In the general context above, given a monoid $M$ we can construct a
category with a single object, one arrow for every $m \in M$, and
composition given by multiplication. The \emph{classifying space}
$|BM|$ of the monoid $M$ is then the geometric realisation of the nerve of
the category corresponding to $M^{op}$ (the reversal is for the technical reason of avoiding reversals in the face maps). In more detail, the nerve of this category is the
simplicial set $BM$ with $n$-simplices:
$\sigma = (m_1, m_2, ..., m_n)$, $n$-tuples of elements of $M$. The
face maps are given by
\[
d_i \sigma =
\begin{cases}
(m_2, \ldots, m_n)
& i=0 \\
(m_1, \ldots, m_{i-1}, m_im_{i+1}, m_{i+2}, \ldots, m_n)
& 0 < i < n \\
(m_1, \ldots, m_{n-1})
& i=n,
\end{cases}
\]
and the degeneracy maps are given by
\[
s_i \sigma = (m_1, \ldots, m_i, 1, m_{i+1}, \ldots, m_n)
\quad
(0 \leq i \leq n).
\]
The geometric realisation $|BM|$ is called the \emph{classifying
  space} of $M$.  Then $|BM|$ is a CW complex with one $n$-cell for
every non-degenerate $n$-simplex of $BM$, i.e., for every $n$-tuple all
of whose entries are different from $1$. As mentioned in the
introduction, classifying spaces of monoids have received some
attention in the literature.


\subsection{Right Cayley graph category}

Let $\Gamma_r(M)$ denote the right Cayley graph category for $M$,
which has
\begin{itemize}
\item Objects: $M$;
\item Arrows: $(x,m)\colon x \rightarrow xm$; and
\item Composition of arrows: $(xm,n) \circ (x,m)=(x,mn).$
\end{itemize}
The identity at $x$  is $(x,1)$.
This composition underlies our use of $M^{op}$ in defining $BM$.

Let $\EMr$ be the nerve of the category $\Gamma_r(M)$.  The
$n$-simplies of $\EMr$ may be written using the notation
$m(m_1, m_2, ...,m_n) = m\tau$ where $\tau=(m_1, m_2, ...,m_n)$ is an
$n$-simplex of $BM$.  Here $m(m_1, m_2, ...,m_n)$ denotes the
$n$-tuple of composable arrows in the category $\Gamma_r(M)$ where we
start at $m$ and the follow the path labelled by $m_1, m_2, ..., m_n$.

The face maps in $\EMr$ are given by
\[
d_i( m(m_1, m_2, ...,m_n) )
=
\begin{cases}
mm_1 (m_2, ...,m_n)                   &  i=0 \\
m(m_1, m_2, ...,m_i m_{i+1}, ..., m_n) &  0<i<n \\
m(m_1, m_2, ...,m_{n-1})              &  i=n
\end{cases}
\]
and the degeneracy maps are given by
\[
s_i \sigma = m(m_1, \ldots, m_i, 1, m_{i+1}, \ldots, m_n)
\quad
(0 \leq i \leq n).
\]
where $\sigma = m(m_1, ...,m_n)$.

Let $|\EMr|$ denote the geometric realisation of $\EMr$. So $|\EMr|$
is a CW complex with one $n$-cell for every non-degenerate $n$-simplex
of $\EMr$, that is, for every $m(m_1, m_2, \ldots, m_n)$ with
$m_j \neq 1$ for $1 \leq j \leq n$. As a consequence, by an $n$-cell
of $\EMr$ we shall mean a non-degenerate $n$-simplex.

Consider the right Cayley graph category $\Gamma_r(M)$. For each
$m \in M$ there is precisely one morphism $(1, m)$ from $1$ to
$m$. Since the category has an initial object we conclude that the
geometric realisation of its nerve
$|\EMr|$ is contractible.

Applying the nerve functor to the projection functor from the category
$\Gamma_r(M)$ to the one-point category $M^{op}$, which identifies all the
vertices of $\Gamma_r(M)$ to a point, gives a simplicial morphism
$\pi\colon \EMr \rightarrow BM$ between the corresponding nerves, which maps
\[
m(m_1, m_2, ..., m_n) \mapsto (m_1, m_2, ..., m_n).
\]
Observe that, for each $n$, the projection $\pi$ maps the set of
$n$-cells of $\EMr$ onto the set of $n$-cells $BM$. If we then apply
the geometric realisation functor we obtain a projection
$\pi\colon|\EMr| \rightarrow |BM|$ (we abuse notation slightly by using the
same notation $\pi$ to denote this map).

The monoid $M$ acts by left multiplication on the category $\Gamma_r(M)$. By functoriality
of the nerve, it follows that $M$ acts on the left of $\EMr_n$
by simplicial morphisms via
\[
s \cdot m(m_1, m_2, ...,m_n) = sm(m_1, m_2, ...,m_n).
\]
Under this action $\EMr_n$ is a free left $M$-set with basis $BM_n$. Also,
if we restrict to the $n$-cells (i.e., non-degenerate simplices), then we obtain a free left $M$-set
with basis the set of $n$-cells of $BM$. It is an easy consequence of
the definitions that this is an action by simplicial morphisms and
that it preserves non-degeneracy in the sense that $s \cdot m\sigma $ is
an $n$-cell if and only if $m\sigma$ is an $n$-cell for all $s \in M$
and $m\sigma \in \EMr$. Therefore $\EMr$ is a rigid
left $M$-simplicial set. Combining these observations with Lemma~\ref{lem_induced} we conclude
that $|\EMr|$ is a free left $M$-CW complex which is
contractible.

Dually, we use $\EMl$ to denote the nerve of the left Cayley graph
category $\Gamma_l(M)$. The simplicial set $\EMl$ satisfies all the
obvious dual statements to those above for $\EMr$. In particular $M$
acts freely via right multiplication action on $\EMl$ by simplicial
morphisms, and $|\EMl|$ is a free right $M$-CW complex which is
contractible.

\subsection{Two-sided Cayley graph category} Let
$\overleftrightarrow{\Gamma(M)}$ denote the two-sided Cayley graph
category for $M$, which has

\begin{itemize}
\item Objects: $M \times M$;
\item Arrows: $M \times M \times M$ where
  $(m_L,m,m_R)\colon  (m_L, m m_R) \rightarrow (m_L m, m_R)$; and
\item Composition of arrows: $(n_L,n,n_R) \circ (m_L,m,m_R) = (m_L,mn,n_R)$
where $(m_L m, m_R) = (n_L, n n_R)$. Equivalently this is the same as the composition
$(m_Lm,n,n_R) \circ (m_L,m,nn_R) = (m_L,mn,n_R)$ and corresponds to the path
\[
(m_L,mnn_R) \rightarrow (m_Lm,nn_R) \rightarrow (m_Lmn,n_R).
\]
\end{itemize}
This is in fact the kernel category of the identity map, in the sense
of Rhodes and Tilson \cite{RhodesTilson1989}. There is a natural $M\times M^{op}$ action of the category $\overleftrightarrow{\Gamma(M)}$.

Let $\EMb$ be the nerve of
the category $\overleftrightarrow{\Gamma(M)}$. The simplicial set
$\EMb$ parallels the two-sided geometric bar construction of
J. P. May; see \cite{May1972, May1975}. The $n$-simplies of $\EMb$
may be written using the notation $m(m_1, m_2, ...,m_n)s = m\tau s$
where $\tau=(m_1, m_2, ...,m_n)$ is an $n$-simplex of $BM$.

Here $m(m_1, m_2, ...,m_n)s$ denotes the $n$-tuple of composable
morphisms in the category $\overleftrightarrow{\Gamma(M)}$ where we
start at $(m,m_1m_2\ldots m_ns)$ and follow the path
\[
(m,m_1m_2m_3\ldots m_ns) \rightarrow
(mm_1,m_2m_3\ldots m_ns) \rightarrow \ldots
\]
\[
\ldots (mm_1m_2, m_3\ldots m_ns) \rightarrow
\ldots
(mm_1m_2\ldots m_n, s)
\]
labelled by the edges
\[
(m, m_1, m_2 m_3 \ldots m_ns),
(mm_1, m_2, m_3 \ldots m_ns),
\ldots,
\]
\[
(mm_1m_2\ldots m_{n-2}, m_{n-1}, m_n s),
(mm_1m_2\ldots m_{n-1}, m_n, s)
\]
and finish at $(mm_1m_2\ldots m_n, s)$.
The face maps in the nerve $\EMb$ are given by
\[
d_i( m(m_1, m_2, ...,m_n)s )
=
\begin{cases}
mm_1 (m_2, ...,m_n)s                   &  i=0 \\
m(m_1, m_2, ...,m_i m_{i+1}, ..., m_n)s &  0<i<n \\
m(m_1, m_2, ...,m_{n-1}) m_ns              &  i=n
\end{cases}
\]
and the degeneracy maps are given by
\[
s_i \sigma = m(m_1, \ldots, m_i, 1, m_{i+1}, \ldots, m_n)s
\quad
(0 \leq i \leq n),
\]
where $\sigma = m(m_1, ...,m_n)s$.

Let $|\EMb|$ denote the geometric realisation of $\EMb$. So $|\EMb|$ is
a CW complex with one $n$-cell for every non-degenerate $n$-simplex of
$\EMb$. Observe that $(m_L,m_R)$ and $(m_L',m_R')$ are in the same component
of the two-sided Cayley graph  category $\overleftrightarrow{\Gamma(M)}$ if and
only if $m_L m_R = m_L' m_R'$. Moreover, for each $m \in M$ the vertex
$(1,m)$ is initial in its component. It follows from these
observations that $\pi_0(|\EMb|)\cong M$ as an $M\times M^{op}$-set,
and each component of $|\EMb|$ is contractible.
There is a natural projection from $\overleftrightarrow{\Gamma(M)}$ to
the one-point category $M^{op}$ mapping $(m_L,m,m_R)$ to its middle
component $m$. Applying the nerve functor to this projection gives a
simplicial morphism $\pi\colon  \EMb \rightarrow BM$ given by
\[
m(m_1, m_2, ..., m_n)s \mapsto (m_1, m_2, ..., m_n).
\]
As in the one-sided case, this projection sends $n$-cells to
$n$-cells and induces a map $\pi\colon  |\EMb| \rightarrow |BM|$ between the
corresponding geometric realisations.

The monoid $M$ has a natural two-sided action on
$\EMb_n$ via
\[
x \cdot [m(m_1, m_2, ...,m_n)s] \cdot y = x m(m_1, m_2, ...,m_n)s y.
\]
Under this action $\EMb$ is a free
rigid bi-$M$-simplicial set. Combining these observations with Lemma~\ref{lem_induced} we conclude
that $|\EMb|$ is a free bi-$M$-CW complex such that
$\pi_0(|\EMb|)\cong M$ as an $M\times M^{op}$-set and each component
of $|\EMb|$ is contractible

\section{One-sided classifying spaces and finiteness properties}\label{sec_onesided}
We will define left and right equivariant classifying spaces for a monoid $M$. Two-sided equivariant classifying spaces will be defined in the next section. As we shall see, the examples discussed in Section~\ref{sec_cspaces} will serve as the standard models of such spaces.

We say that a projective $M$-CW complex $X$ is a \emph{(left) equivariant classifying space} for $M$ if it is contractible. A right equivariant classifying space for $M$ will be a left equivariant classifying space for $M^{op}$.  Notice that the augmented cellular chain complex of an equivariant classifying space for $M$ provides a projective resolution of the trivial (left) $\mathbb ZM$-module $\mathbb Z$.

\begin{Example}
The bicyclic monoid is the monoid $B$ with presentation $\langle a,b\mid ab=1\rangle$.  
It is not hard to see that each element of $B$ is uniquely represented by a word of the form $b^i a^j$ where $i,j \geq 0$. Figure~\ref{fi:equiv} shows an equivariant classifying space for $B$.  The $1$-skeleton is the Cayley graph of $B$ and there is a $2$-cell glued in for each path labelled $ab$. 
This example is a special case of far more general results about equivariant classifying spaces of special monoids which will appear in a future paper of the current authors \cite{GraySteinbergInPrep}.  
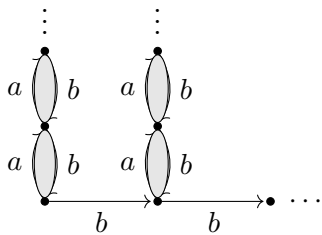
\begin{figure}[htbp]
\begin{center}
\begin{tikzpicture}[shorten >=1pt,scale=1, vertices/.style={draw, fill=black, circle, inner sep=1pt}]
      \node[vertices,label=left:{}] (A) at (0,0) {};
      \node[vertices,label=left:{}] (B) at (0,1) {};
      \node[vertices,label=left:{}] (C) at (0,2) {};
      \node                             at (0,2.5) {$\vdots$};
      \node[vertices,label=left:{}] (D) at (1.5,0) {};
      \node[vertices,label=left:{}] (E) at (1.5,1) {};
      \node[vertices,label=left:{}] (F) at (1.5,2) {};
      \node                             at (1.5,2.5) {$\vdots$};
      \node[vertices,label=left:{}] (G) at (3,0) {};
      \node                             at (3.5,0)  {$\cdots$};
      \path  (A) edge[->, bend left] node [left] {$a$} (B)
             (B) edge[->, bend left] node [left] {$a$} (C);
      \path  (C) edge[->, bend left] node [right] {$b$} (B)
             (B) edge[->, bend left] node [right] {$b$} (A);
      \path  (A) edge[->] node [below] {$b$} (D);
      \path  (D) edge[->, bend left] node [left] {$a$} (E)
             (E) edge[->, bend left] node [left] {$a$} (F);
      \path  (F) edge[->, bend left] node [right] {$b$} (E)
             (E) edge[->, bend left] node [right] {$b$} (D);
      \path  (D) edge[->] node [below] {$b$} (G);
      \draw [black,fill=gray!20] (0,0.5) ellipse (.14cm and .45cm);
      \draw [black,fill=gray!20] (0,1.5) ellipse (.14cm and .45cm);
      \draw [black,fill=gray!20] (1.5,0.5) ellipse (.14cm and .45cm);
      \draw [black,fill=gray!20] (1.5,1.5) ellipse (.14cm and .45cm);
\end{tikzpicture}
\end{center}
\caption{An equivariant classifying space for the bicyclic monoid\label{fi:equiv}}
\end{figure}
\end{Example}

Our first goal is to show that any two equivariant classifying spaces for $M$ are $M$-homotopy equivalent.

\begin{Lemma}\label{l:construct.map}
Let $X$ be an equivariant classifying space for $M$ and let $Y$ be a contractible $M$-space.  Then there exists a continuous $M$-equvariant mapping $f\colon X\to Y$.
\end{Lemma}
\begin{proof}
The proof constructs by induction $M$-equivariant continuous mappings $f_n\colon X_n\to Y$ with $f_n$ extending $f_{n-1}$.  To define $f_0$, observe that $X_0 = \coprod_{a\in A}Me_a$ (without loss of generality) and so, by Proposition~\ref{p:cell-maps}, $\mathbf{Top}_M(X_0,Y)\cong \prod_{a\in A}e_aY\neq \emptyset$ and so we can define $f_0$.  Assume that $f_n\colon X_n\to Y$ has been defined. Let $Z$ be the one-point space with the trivial $M$-action and let $k\colon Y\to Z$ be the unique $M$-equivariant map.  Then $k$ is a weak equivalence.  So by Theorem~\ref{t:HELP} we can construct a commutative diagram
\[\begin{tikzcd}X_n\ar{rr}{i_0}\ar[hook]{dd}[swap]{i} & & X_n\times I\ar{dl}\ar[hook]{dd} & & X_n\ar{ll}[swap]{i_1}\ar[hook]{dd}{i}\ar{dl}{f_n}\\
                              &Z&           &\ar[crossing over]{ll}[swap]{k} Y&   \\
                X_{n+1}\ar{rr}{i_0}\ar{ur}             & & X_{n+1}\times I\ar[dashrightarrow]{ul}[swap]{} & & X_{n+1}\ar{ll}[swap]{i_1}\ar[dashrightarrow]{ul}[swap]{f_{n+1}}\end{tikzcd}\]
with $f_{n+1}$ $M$-equivariant.  Now take $f$ to be the colimit of the $f_n$. \end{proof}

\begin{Thm}\label{t:uniqueness.left}
Let $X,Y$ be equivariant classifying spaces for $M$.  Then $X$ and $Y$ are $M$-homotopy equivalent by a cellular $M$-homotopy equivalence.
\end{Thm}
\begin{proof}
By Corollary~\ref{c:whitehead} and Theorem~\ref{t:cell.approx} it suffices to construct an $M$-equivariant continuous mapping $f\colon X\to Y$.  But Lemma~\ref{l:construct.map} does just that.
\end{proof}

Next we now give an elementary proof that contractible free $M$-CW complexes exist and hence there are equivariant classifying spaces for $M$.  A more canonical construction, using simplicial sets, was given in the previous section.

\begin{Lemma}\label{l:improve}
Let $M$ be a monoid.
\begin{enumerate}
  \item If $X_0$ is a non-empty projective (free) $M$-set, then there is a connected projective (free) $M$-graph $X$ with vertex set $X_0$.
  \item If $X$ is a connected projective (free) $M$-CW complex such that $\pi_q(X)=0$ for $0\leq q<n$, then there exists a projective $M$-CW complex $Y$ containing $X$ as a projective $M$-CW subcomplex and such that $Y_n=X_n$ and $\pi_q(Y)=0$ for $0\leq q\leq n$.
  \item If $X$ is a connected projective (free) $M$-CW complex such that $\pi_q(X)$ is trivial for $0\leq q<n$, then there exists a contractible projective (free) $M$-CW complex $Y$ containing $X$ as a projective $M$-CW subcomplex and such that $Y_n=X_n$.
\end{enumerate}
\end{Lemma}
\begin{proof}
For the first item, fix $x_0\in X_0$.  The edge set of $X$ will be bijection with $M\times X_0$ with the edge corresponding to $(m,x)$ connecting $mx_0$ to $mx$. Then $X$ is a projective (free) $M$-graph and each vertex $x$ is connected to $x_0$ via the edge $(1,x)$.

For the second item, we show that we can add free $M$-cells of dimension $n+1$ to $X$ to obtain a new projective $M$-CW complex $Y$ with $\pi_n(Y)=0$. If $\pi_n(X)=0$, then take $Y=X$.  So assume that $\pi_n(X)$ is non-trivial.  Fix a base point $x_0\in X_0$ and let $f_a\colon S^n\to X$, $a\in A$, be mappings whose based homotopy classes form a set of generators for $\pi_n(X,x_0)$. As $X_n\to X$ is an $n$-equivalence, we may assume without loss of generality that $f_a\colon S^n\to X_n$. Suppose that $X$ is constructed from pushouts as per \eqref{eq:pushout}.  Note that $M\times A$, where $A$ has the trivial action, is a free $M$-set.   Let us define $Y$ by putting $Y_k=X_k$ for $0\leq k\leq n$, defining $Y_{n+1}$ to be the pushout
\[\begin{tikzcd}(P_{n+1}\times S^n)\coprod (M\times A\times S^n)\ar{r}\ar[d,hook] &X_n\ar[d,hook]\\ (P_{n+1}\times B^{n+1})\coprod (M\times A\times B^{n+1})\ar{r} & Y_{n+1},\end{tikzcd}\] where the top map is the union of the attaching map for $X$ with the mapping $(m,a,x)\mapsto mf_a(x)$ (cf.~Proposition~\ref{p:cell-maps}), and putting $Y_k=X_k\cup Y_{n+1}$ for $k>n+1$.  Then $Y$ is a projective $M$-CW complex containing $X$ as a projective $M$-CW subcomplex and with $Y_n=X_n$.  Moreover, since $X_n=Y_n\to Y$ is an $n$-equivalence, it follows that the based homotopy classes of the $f_a\colon S^n\to Y$ generate $\pi_n(Y,x_0)$.  By construction the $f_a$ can be extended to $B^{n+1}\to Y$ and so their classes are trivial in $\pi_n(Y,x_0)$.  Thus $\pi_n(Y)=0$.  Also, because $X_n=Y_n\to Y$ is an $n$-equivalence, we have that $\pi_q(Y)=0$ for $0\leq q<n$.

The final item follows from Whitehead's theorem, iteration of the second item and that $Y_n\to Y$ is an $n$-equivalence for all $n\geq 0$.
\end{proof}

\begin{Cor}\label{c:exist.class}
Let $M$ be a monoid.  Then there exists a contractible free $M$-CW complex.
\end{Cor}
\begin{proof}
Put $X_0=M$.  Then by Lemma~\ref{l:improve} we can find a connected free $M$-graph $X$ with vertex set $X_0$.  Now applying the third item of Lemma~\ref{l:improve} we can find a contractible free $M$-CW complex with $1$-skeleton $X$.
\end{proof}

\begin{example}
It follows from the definitions and results in
Section~\ref{sec_cspaces} that the geometric realisation
  $|\EMr|$ of the nerve of the right Cayley graph category of $M$ is a
  left equivariant classifying space for $M$.
\end{example}

\begin{Cor}\label{c:uniqueness.BM}
If $X,Y$ are equivariant classifying spaces for $M$, then $M\backslash X$ and $M\backslash Y$ are homotopy equivalent.  In particular, $M\backslash X\simeq |BM|$.  Therefore, if $G$ is a group and $X$ is an equivariant classifying space for $G$, then $G\backslash X$ is an Eilenberg-Mac Lane space for $G$.  Conversely, the universal cover of any Eilenberg-Mac Lane space for $G$ is an equivariant classifying space for $G$.
\end{Cor}
\begin{proof}
The first statement follows from Theorem~\ref{t:uniqueness.left} and Proposition~\ref{p:preservation.hom.equiv}.  The second statement follows from the first as $|\EMr|$ is an equivariant classifying space for $M$.  The group statements then follow from the previous statements and classical covering space theory.
\end{proof}

If $M$ and $N$ are monoids, then $E(M\times N)= E(M)\times E(N)$ and $(M\times N)(e,f)= Me\times Nf$.  It follows that if $P$ is a (finitely generated) projective $M$-set and $Q$ a (finitely generated) projective $N$-set, then $P\times Q$ is a (finitely generated projective) $M\times N$-set.

\begin{Prop}\label{p:direct.prod}
Let $M,N$ be monoids and let $X,Y$ be equivariant classifying spaces for $M,N$, respectively.  Then $X\times Y$ is an $M\times N$-equivariant classifying space, which is of $M\times N$-finite type whenever $X$ is of $M$-finite type and $Y$ is of $N$-finite type.
\end{Prop}
\begin{proof}
Assume that $X$ is obtained via attaching projective $M$-cells $P_n\times B^n$ and that $Y$ is obtained by attaching projective $N$-cells $Q_n\times B^n$.  Then the $n$-cells of $X\times Y$ are obtained by adjoining $\coprod_{i=0}^n P_i\times Q_{n-i}\times B^n$ and hence $X\times Y$ is an $M\times N$-projective CW complex which is of $M\times N$-finite type whenever $X$ is of $M$-finite type and $Y$ is of $N$-finite type.  As $X\times Y$ is contractible, we deduce that it is an $M\times N$-equivariant classifying space. 
\end{proof}

\subsection{Monoids of type left-$\F_n$}

A key definition for this paper is the following.  A monoid $M$ is of type \emph{left-$\F_n$} if there is an equivariant classifying space $X$ for $M$ such that $X_n$ is $M$-finite, i.e., such that $M\backslash X$ has finite $n$-skeleton.   We say that $M$ is of type \emph{left-$\F_{\infty}$} if $M$ has an equivariant classifying space $X$ that is of $M$-finite type, i.e., $M\backslash X$ is of finite type. The monoid $M$ is defined to have type \emph{right-$\F_n$} if $M^{op}$ is of type left-$\F_n$ for $0\leq n\leq \infty$. The following proposition contains some basic facts.

\begin{Prop}\label{p:basic.props.fn}
Let $M$ be a monoid.
\begin{enumerate}
\item A group is of type left-$\F_n$ if and only if it is of type $\F_n$ in the usual sense for any $0\leq n\leq \infty$.
\item For $0\leq n\leq \infty$, if $M$ is of type left-$\F_n$, then it is of type left-$\FP_n$.
\item If $M$ is of type left-$\F_{\infty}$, then it is of type left-$\F_n$ for all $n\geq 0$.
\end{enumerate}
\end{Prop}
\begin{proof}
The first item follows from Corollary~\ref{c:uniqueness.BM} and Corollary~\ref{c:quotient}.  The second is immediate using that the augmented cellular chain complex of an equivariant classifying space $X$ gives a projective $\mathbb ZM$-resolution of the trivial $\mathbb ZM$-module since if $X$ is built up from pushouts as per \eqref{eq:pushout}, then the $n^{th}$-chain module is isomorphic to $\mathbb ZP_n$.  The final item is trivial.
\end{proof}

Note that, trivially, if $M$ is a finite monoid then $|\EMr|$ has finitely many
cells in each dimension and thus $M$ is of type left-$\Finfty$.

Sometimes it will be convenient to use the following reformulation of the property left-$\F_n$.

\begin{Prop}\label{p:reform.fn}
Let $M$ be a monoid.  The following are equivalent for $0\leq n<\infty$.
\begin{enumerate}
  \item $M$ is of type left-$\F_n$
  \item There is a connected $M$-finite  projective $M$-CW complex $X$ of dimension at most $n$ with $\pi_q(X)=0$ for $0\leq q<n$.
\end{enumerate}
\end{Prop}
\begin{proof}
If $Y$ is an equivariant classifying space for $M$ such that $Y_n$ is $M$-finite, then since $Y_n\to Y$ is an $n$-equivalence, we deduce that $X=Y_n$ is as required for the second item.  Conversely, if $X$ is as in the second item, we can construct by Lemma~\ref{l:improve} an equivariant classifying space $Y$ for $M$ with $Y_n=X$.  Thus $M$ is of type left-$\F_n$.
\end{proof}

%
%

Recall that the fundamental group of $|BM|$ is isomorphic to the universal group (or maximal group image, or group completion) $U(M)$ of $M$ i.e., the group with generators $M$ and relations the multiplication table of $M$ (cf.~\cite{GabrielZisman}).

\begin{Cor}\label{c:univ.group}
Let $M$ be a monoid.  If $M$ is of type left-$\F_1$, then $U(M)$ is finitely generated.  If $M$ is of type left-$\F_2$, then $U(M)$ is finitely presented.
\end{Cor}
\begin{proof}
By Corollary~\ref{c:uniqueness.BM}, $|BM|$ in the first case is homotopy equivalent to a CW complex with finite $1$-skeleton and in the second case to a CW complex with finite $2$-skeleton by Corollary~\ref{c:quotient}.  Thus $U(M)\cong \pi_1(BM)$ has the desired properties in both cases.
\end{proof}

Recall that an inverse monoid is a monoid $M$ with the property that for every $m \in M$ there is a unique element $m' \in M$ such that $mm' m = m$ and $m' m m' = m'$. For more on inverse monoids, and other basic concepts from semigroup theory we refer the reader to \cite{Howie}.

\begin{Cor}\label{c:BM.eilen}
Let $M$ be a monoid such that $|BM|$ is an Eilenberg-Mac Lane space (e.g., if $M$ is cancellative with a left or right Ore condition or if $M$ is an inverse monoid) and suppose that $0\leq n\leq \infty$.  If $M$ is of type left-$\F_n$, then $U(M)$ is of type $\F_n$.
\end{Cor}
\begin{proof}
If $X$ is an equivariant classifying space for $M$, then $M\backslash X$ is homotopy equivalent to $|BM|$ by Corollary~\ref{c:uniqueness.BM}
 and hence is an Eilenberg-Mac Lane space for $U(M)$.  The result now follows from Corollary~\ref{c:quotient}.
\end{proof}

Since, as already mentioned above, D. McDuff \cite{McDuff1979} has shown that every path-connected space has the weak homotopy type of the classifying space of some monoid, not every monoid $|BM|$ is an Eilenberg-Mac Lane space. So not every monoid satisfies the hypotheses of Corollary~\ref{c:BM.eilen}. The fact that if $M$ is cancellative with a left or right Ore condition then $|BM|$ is an Eilenberg-Mac Lane space is well known. If $M$ is an inverse monoid then $|BM|$ may also be shown to be an Eilenberg-Mac Lane space. Both of these results can easily be proved by appealing to Quillen's theorem A, see~\cite[Chapter~4]{WeibelKBook}, and should be considered folklore.

The converse of Corollary~\ref{c:BM.eilen} does not hold. For example,
the free inverse monoid on one generator is not of type
left-$\F_2$ while its maximal group image  $\mathbb Z$
is $\Finfty$ (this proof of the fact that the free inverse monoid on one generator
is not left-$\F_2$ will appear in \cite{GraySteinbergInPrep}).

For groups, being of type $\F_1$ is equivalent to finite generation.  For monoids, the condition of being left-$\F_1$ is considerably weaker.  Recall that if $M$ is a monoid and $A\subseteq M$, then the (right) \emph{Cayley digraph} $\Gamma(M,A)$ of $M$ with respect to $A$ is the graph with vertex set $M$ and with edges in bijection with $M\times A$ where the directed edge (arc) corresponding to $(m,a)$ starts at $m$ and ends at $ma$.  Notice that $\Gamma(M,A)$ is a free $M$-CW graph and is $M$-finite if and only if $A$ is finite.

\begin{Thm}\label{t:f1}
Let $M$ be a monoid.  The following are equivalent.
\begin{enumerate}
  \item $M$ is of type left-$\F_1$.
  \item $M$ is of type left-$\FP_1$.
  \item There is a finite subset $A\subseteq M$ such that $\Gamma(M,A)$ is connected as an undirected graph.
\end{enumerate}
In particular, any finitely generated monoid is of type left-$\F_1$.
\end{Thm}
\begin{proof}
Item (1) implies (2) by Proposition~\ref{p:basic.props.fn}, whereas (2) implies (3) by a result due to Kobayashi~\cite{Kobayashi2007}. For completeness, let us sketch the proof that (2) implies (3).
Let $\varepsilon\colon \mathbb ZM\to \mathbb Z$ be the augmentation map; the ideal $I=\ker \varepsilon$ is called the augmentation ideal.  If $M$ is of type left-$\FP_1$, then $I$ must be finitely generated because the augmentation map gives a partial free resolution.  But $I$ is generated by all elements of the form $m-1$ with $m\in M$. Hence there is a finite subset $A\subseteq M$ such that the elements $a-1$ with $a\in A$ generate $I$.  Consider the Cayley digraph $\Gamma(M,A)$.  Then $M$ acts cellularly on $\Gamma(M,A)$ and hence acts on $\pi_0(\Gamma(M,A))$.  There is a surjective $\mathbb ZM$-module homomorphism $\eta\colon \mathbb ZM\to \mathbb Z\pi_0(\Gamma(M,A))$ mapping $m\in M$ to the connected component of the vertex $m$ of $\Gamma(M,A)$. Moreover, the augmentation $\varepsilon$ factors through $\eta$.  Thus to show that $\Gamma(M,A)$ is connected, it suffices to show that $I=\ker \eta$.  By construction $\ker \eta\subseteq I$.  But if $a\in A$, then $a$ and $1$ are in the same connected component of $\Gamma(M,A)$ and thus $a-1\in \ker \eta$.  Since the $a-1$ with $a\in A$ generate $I$, we deduce that $I\subseteq\ker \eta$ and hence $\Gamma(M,A)$ is connected.

 Finally, (3) implies (1) by Proposition~\ref{p:reform.fn} as $\Gamma(M,A)$ is an $M$-finite connected free $M$-CW complex of dimension at most $1$.
\end{proof}

We next show that a finitely presented monoid is of type left-$\F_2$.  In fact, we shall see later that finitely presented monoids are of type bi-$\F_2$, which implies left-$\F_2$, but the proof of this case is instructive.

\begin{Thm}\label{t:f2}
Let $M$ be a finitely presented monoid.  Then $M$ is of type left-$\F_2$
\end{Thm}
\begin{proof}
Suppose that $M$ is generated by a finite set $A$ with defining relations $u_1=v_1,\ldots, u_n=v_n$.  Let us construct a $2$-dimensional, $M$-finite, free $M$-CW complex $X$ with $1$-skeleton the Cayley graph $\Gamma(M,A)$ by attaching a free $M$-cell $M\times B^2$ for each relation.  Let $p_i,q_i$ be the paths from $1$ to $m_i$ labelled by $u_i$ and $v_i$, respectively, where $m_i$ is the image of $u_i$ (and $v_i$) in $M$.  Then we glue in a disk $d_i$ with boundary path $p_iq_i^{-1}$ and glue in $M\times B^2$ using Proposition~\ref{p:cell-maps} (so $\{m\}\times B^2$ is sent to $md_i$).  Then $X$ is an $M$-finite connected free $M$-CW complex of dimension at most $2$.  See Figure~\ref{fi:equiv} for this construction for the bicyclic monoid.  By Proposition~\ref{p:reform.fn}, it suffices to prove that $X$ is simply connected.

A digraph is said to be \emph{rooted} if there is a vertex $v$ so that there is a directed path from $v$ to any other vertex.  For instance, $\Gamma(M,A)$ is rooted at $1$.  It is well known that a rooted digraph admits a spanning tree, called a \emph{directed spanning tree}, such that the geodesic from the root to any vertex is directed.  Let $T$ be a directed spanning tree for $\Gamma(M,A)$ rooted at $1$ (this is the same thing as a prefix-closed set of normal forms for $M$ so, for instance, shortlex normal forms would do).  Let $e=m\xrightarrow{\,\,a\,\,} ma$ be a directed edge not belonging to $T$.  Then the corresponding generator of $\pi_1(X,1)$ is of the form $peq^{-1}$ where $p$ and $q$ are directed paths from $1$ to $m$ and $ma$, respectively.  Let $u$ be the label of $p$ and $v$ be the label of $q$.  Then $ua=v$ in $M$.  Thus it suffices to prove that if $x,y\in A^*$ are words which are equal in $M$ to an element $m'$, then the loop $\ell$ labelled $xy^{-1}$ at $1$, corresponding to the pair of parallel paths $1$ to $m'$ labelled by $x$ and $y$, is null homotopic.  By induction on the length of a derivation from $x$ to $y$, we may assume that $x=wu_iw'$ and $y=wv_iw'$ for some $i=1,\ldots, n$.  Let $m_0$ be the image of $w$ in $M$.  Then $m_0d_i$ is a $2$-cell with boundary path the loop at $m_0$ labeled by $u_iv_i^{-1}$.  It follows that $\ell$ is null homotopic.  This completes the proof.
\end{proof}

The converse of Theorem~\ref{t:f2} is not true, e.g., the monoid
$(\mathbb{R},\cdot)$ is of type left-$\F_2$ (by Corollary~\ref{c:right.zero}) but is not even finitely
generated.
It is natural to ask whether there is a nice characterisation,
analogous to Theorem~\ref{t:f1}(3), for left-$\F_2$ in terms of the
right Cayley graph together with the left action of $M$. We would
guess that $M$ is of type left-$\F_2$ if and only if it has a finite
subset $A \subseteq M$ such that $\Gamma(M,A)$ is connected, and
finitely many free $M$-2-cells can be adjoined to make a simply
connected $2$-complex.

It is well known that, for finitely presented groups, the properties
$\F_n$ and $\FP_n$ are equivalent for $3\leq n\leq \infty$.  We now
provide the analogue in our context.  Here we replace finitely
presented by left-$\F_2$.

\begin{Thm}\label{t:hom=top}
Let $M$ be a monoid of type left-$\F_2$.  Then $M$ is of type left-$\F_n$ if and only if $M$ is of type left-$\FP_n$ for $0\leq n\leq \infty$.
\end{Thm}
\begin{proof}
We prove that if there is a connected $M$-finite  projective $M$-CW complex $X$ of dimension at most $n$ with $\pi_q(X)=0$ for $0\leq q<n$ with $n\geq 2$ and $M$ is of type left-$\FP_{n+1}$, then there is a  connected  $M$-finite projective $M$-CW complex $Y$ of dimension at most $n+1$ with $Y_n=X$ and $\pi_q(Y)=0$ for all $0\leq q<n+1$.    This will imply the theorem by Proposition~\ref{p:reform.fn}, Proposition~\ref{p:basic.props.fn} and induction.

Since $X$ is simply connected, $H_q(X)=0$ for $1\leq q<n$ and $H_n(X)\cong \pi_n(X)$ by the Hurewicz theorem.  Therefore, the augmented cellular chain complex of $X$ gives a partial projective resolution of $\mathbb Z$ of length at most $n$, which is finitely generated in each degree.  Therefore, since $M$ is of type left-$\FP_{n+1}$, it follows that $H_n(X)=\ker d_n\colon C_n(X)\to C_{n-1}(X)$ is finitely generated as a left $\mathbb ZM$-module.  Choose representatives of $f_a\colon S^n\to X$, with $a\in A$, of a finite set of elements of $\pi_n(X)$ that map to a  finite $\mathbb ZM$-module generating set of $H_n(X)$ under the Hurewicz isomorphism.  Then form $Y$ by adjoining $M\times A\times B^{n+1}$ to $X$ via the attaching map $M\times A\times S^n\to X_n$ given by $(m,a,x)\mapsto mf_a(x)$.  Then $Y$ is an $M$-finite projective $M$-CW complex of dimension $n+1$ with $Y_n=X$.  Since the inclusion of $X=Y_n$ into $Y$ is an $n$-equivalence, we deduce that $\pi_q(Y)=0$ for $1\leq q<n$ and that the inclusion $X\to Y$ is surjective on $\pi_n$.  But since the Hurewicz map in degree $n$ is natural and is an isomorphism for both $X$ and $Y$, we deduce that the inclusion $X\to Y$ induces a surjection $H_n(X)\to H_n(Y)$.  But, by construction, the images of the $\mathbb ZM$-module generators of $H_n(X)$ are trivial in $H_n(Y)$ (since they represent trivial elements of $\pi_n(Y)$).  We deduce that $H_n(Y)=0$ and hence, by the Hurewicz theorem, $\pi_n(Y)=0$.  This completes the induction.
\end{proof}

Notice that Theorem~\ref{t:hom=top} implies that $M$ is of type left-$\F_{\infty}$ if and only if $M$ is of type left-$\F_n$ for all $n\geq 0$.

\begin{Prop}\label{p:proj=free}
If $M$ is of type left-$\F_n$ with $n\geq 1$, then $M$ has a free contractible $M$-CW complex $X$ such that $X_n$  is $M$-finite.
\end{Prop}
\begin{proof}
This is clearly true for $n=1$ by Theorem~\ref{t:f1}. Note that Lemma~\ref{l:improve} and the  construction in the proof of Theorem~\ref{t:hom=top} show that if $Y$ is a simply connected $M$-finite free $M$-CW complex of dimension at most $2$ and $M$ is of type left-$\FP_n$, then one can build a contractible free $M$-CW complex $X$ with $X_2=Y$ such that $X_n$ is $M$-finite.  Thus it remains to prove that if there is a simply connected $M$-finite projective $M$-CW complex $Y$ of dimension $2$, then there is a simply connected $M$-finite free $M$-CW complex $X$ of dimension $2$.

Note that $Y_0=\coprod_{a\in A}Me_a$ with $A$ a finite set.  Define $X_0 = M\times A$.  Identifying $Me_a$ with $Me_a\times \{a\}$, we may view $Y_0$ as an $M$-subset of $X_0$.  Using this identification, we can define $X_1$ to consists of $Y_1$ (the edges of $Y$) along with some new edges.  We glue in an edge  from $(m,a)$ to $(me_a,a)$ for each $m\in M$ and $a\in A$; that is we glue in a free $M$-cell $M\times A\times B^1$ where the attaching map takes $(m,a,0)$ to $(m,a)$ and $(m,a,1)$ to $(me_a,a)$.   Notice that all vertices of $X_0\setminus Y_0$ are connected to a vertex of $Y_0$ in $X_1$ and so $X_1$ is connected as $Y_1$ was connected.

To define $X_2$, first we keep all the two-cells from $Y_2$.  Notice that if $T$ is a spanning tree for $Y$, then a spanning tree $T'$ for $X$ can be obtained by adding to $T$ all the edges $(m,a)\longrightarrow (me_a,a)$ with $m\notin Me_a$ (all vertices of $X_0\setminus Y_0$ have degree one). Thus  the only edges in $X_1\setminus Y_1$ that do not belong to $T'$ are the loop edges $(m,a)\longrightarrow (me_a,a)$ for $m\in Me_a$ that we have added.  So if we attach $M\times A\times B^2$ to $X_1$ by the attaching map $M\times A\times S^1\to X_1$ mapping $\{m\}\times \{a\}\times S^1$ to the loop edge $(me_a,a)\longrightarrow (me_a,a)$ from $X_1\setminus Y_1$, then we obtain a simply connected free $M$-CW complex $X$ which is $M$-finite.  This completes the proof.
\end{proof}

In light of Proposition~\ref{p:proj=free}, one might wonder why we bother allowing projective $M$-CW complexes rather than just free ones.  The reason is because projective $M$-CW complexes are often easier to construct and, as we are about to see, sometimes it is possible to find an $M$-finite equivariant classifying space for $M$ which is projective when no $M$-finite free equivariant classifying space exists.  This will be relevant when considering geometric dimension.

Let $M$ be a non-trivial monoid with a right zero element $z$.  Then $Mz=\{z\}$ is a one-element set with the trivial action.  Since $z$ is idempotent, it follows that the trivial $M$-set is projective but not free.  Therefore, the one-point space with the trivial $M$-action is an $M$-finite equivariant classifying space for $M$, which is not free.    We will show that if $M$ has a finite number of right zeroes (e.g., if $M$ has a zero element), then there is no finite free resolution of the trivial module which is finitely generated in each degree.  In this case, every free equivariant classifying space for $M$ of $M$-finite type will be infinite dimensional.

A finitely generated projective module $P$ over a ring $R$ is said to be \emph{stably free} if there are finite rank free modules $F,F'$ such that $P\oplus F'\cong F$.  The following lemma is well known, but we include a proof for completeness.

\begin{Lemma}\label{l:stably.free}
Let $P$ be a finitely generated projective (left) module over a ring $R$.  Then $P$ has a finite free resolution, finitely generated in each degree, if and only if $P$ is stably free.
\end{Lemma}
\begin{proof}
Suppose that $P$ is stably free, say $P\oplus F'\cong F$ with $F,F'$ finite rank free $R$-modules.  Then the exact sequence
\[0\longrightarrow F'\longrightarrow F\longrightarrow P\longrightarrow 0\] provides a finite free resolution of $P$ that is finitely generated in each degree.

Conversely, suppose that \[0\longrightarrow F_n\longrightarrow F_{n-1}\longrightarrow \cdots \longrightarrow F_0\longrightarrow P\longrightarrow 0\] is a free resolution with $\F_i$ finitely generated for all $0\leq i\leq n$. We also have a projective resolution
\[0\longrightarrow P_n\longrightarrow P_{n-1}\longrightarrow \cdots \longrightarrow P_0\longrightarrow P\longrightarrow 0\] with $P_0=P$ and $P_i=0$ for $1\leq i\leq n$ because $P$ is projective.  By the generalized Schanuel's lemma, we have that
\[P_0\oplus F_1\oplus P_2\oplus F_3\oplus \cdots \cong F_0\oplus P_1\oplus F_2\oplus P_3\oplus \cdots\] and hence
\begin{equation*}
P\oplus F_1\oplus F_3\oplus\cdots \cong F_0\oplus F_2\oplus \cdots
\end{equation*}
and so $P$ is stably free.
\end{proof}

So we are interested in showing that the trivial module for $\mathbb ZM$ is not stably free if $M$ is a non-trivial monoid with finitely many right zeroes (and at least one).

Recall that a ring $R$ is said to have the \emph{Invariant Basis Number property} (IBN) if whenever $R^n\cong R^m$ as $R$-modules, one has $m=n$ (where $m,n$ are integers); in this definition it does not matter if one uses left or right modules~\cite{LamBook}.

Our first goal is to show that if $M$ is a monoid with zero $z$, then the contracted monoid ring $\mathbb ZM/\mathbb Zz$ has IBN.  This result is due to Pace Nielsen, whom we thank for allowing us to reproduce it. It is equivalent to show that if $M$ is a monoid and $I$ is a proper ideal of $M$, then $\mathbb ZM/\mathbb ZI$ has IBN.  The proof makes use of the Hattori-Stallings trace (see \cite[Chapter~2]{WeibelKBook}).

 Let $\sim$ be the least equivalence relation on a monoid $M$ such that $mn\sim nm$ for all $m,n\in M$; this relation, often called \emph{conjugacy}, has been studied by a number of authors.

\begin{Lemma}\label{l:conjugate.to.idem}
Let $M$ be a monoid and $e\in M$ an idempotent.  Suppose that $e$ is conjugate to an element of an ideal $I$.  Then $e\in I$.
\end{Lemma}
\begin{proof}
Suppose that $e$ is conjugate to $m\in I$.  Then we can find elements $x_1,\ldots, x_n,y_1,\ldots, y_n\in M$ with $e=x_1y_1$, $y_ix_i=x_{i+1}y_{i+1}$ and $y_nx_n=m$.  Then $e=e^{n+1}=(x_1y_1)^{n+1}=x_1(y_1x_1)^ny_1=x_1(x_2y_2)^ny_1=x_1x_2(y_2x_2)^{n-1}y_2y_1=\cdots= x_1\cdots x_n(y_nx_n)y_ny_{n-1}\cdots y_1\in I$ as $y_nx_n=m\in I$.
\end{proof}

If $R$ is a ring, then $[R,R]$ denotes the additive subgroup generated
by all commutators $ab-ba$ with $a,b\in R$.  The abelian group
$R/[R,R]$ is called the \emph{Hattori-Stallings trace} of $R$; this is
also the $0$-Hochschild homology group of $R$.  Cohn proved that if
$1+[R,R]$ has infinite order in $R/[R,R]$, then $R$ has IBN (see \cite[Exercise~1.5]{LamBook}). The point
is that if $A$ is an $m\times n$ matrix over $R$ and $B$ is an
$n\times m$-matrix over $R$ with $AB=I_m$ and $BA=I_n$, then
$m+[R,R]=T(AB)=T(BA)=n+[R,R]$ where, for a square matrix $C$ over $R$,
we define $T(C)$ to be the class of the sum of the diagonal entries of
$C$ in $R/[R,R]$.   The following proposition is an elaboration of
Pace Nielsen's argument (see \cite{248371}).

\begin{Prop}\label{p:nielsen}
Let $M$ be a monoid and $I$ a (possibly empty) ideal. Let $R=\mathbb ZM/\mathbb ZI$.  Then $R/[R,R]$ is a free abelian group on the conjugacy classes of $M$ that do not intersect $I$.  More precisely, if $T$ is a transversal to the conjugacy classes of $M$ not intersecting $I$, then the elements of the form $m+[R,R]$ with $m\in T$ form a basis for $R/[R,R]$.
\end{Prop}
\begin{proof}
We view $R$ as having basis $M\setminus I$ subject to the relations of the multiplication table of $M$ and that the elements of $I$ are $0$.  Let $A$ be the free abelian group on the conjugacy classes of $M$ that do not intersect $I$.  Write $[m]$ for the conjugacy class of $m\in M$.  Define an abelian group homomorphism $f\colon A\to R/[R,R]$ by $f([m]) = m+[R,R]$.  This is well defined because $xy+[R,R]=yx+[R,R]$ for $x,y\in M$ with $xy,yx\notin I$.  To see that $f$ is surjective, note that if $m\in M\setminus I$ with $[m]\cap I\neq \emptyset$, then $m+[R,R]=[R,R]$.  This follows because if $m=x_1y_1$, $y_ix_i=x_{i+1}y_{i+1}$, for $i=1,\ldots, n-1$,  and $y_nx_n\in I$, then $[R,R]=y_nx_n+[R,R]=x_ny_n+[R,R]=\cdots =y_1x_1+[R,R]=m+[R,R]$.

Let us define $g\colon R\to A$ on $m\in M\setminus I$ by
\[g(m) = \begin{cases} [m], & \text{if}\ [m]\cap I=\emptyset\\ 0, & \text{else.}\end{cases}\]
Then if $a,b\in R$ with, say,
\[a=\sum_{m\in M\setminus I} c_mm,\ b=\sum_{n\in M\setminus I} d_nn\]
then we have that
\[ab-ba =\sum_{m,n\in M\setminus I} c_md_n(mn-nm).\]  Since $mn\sim nm$, either both map to $0$ under $g$ or both map to $[mn]=[nm]$.  Therefore, $ab-ba\in \ker g$ and so $g$ induces a homomorphism $g'\colon R/[R,R]\to A$.  Clearly, if $[m]\cap I=\emptyset$, then $gf([m]) = g'(m+[R,R])= g(m)=[m]$.  It follows that $f$ is injective and hence an isomorphism. The result follows.
\end{proof}

As a consequence we deduce the result of Nielsen.

\begin{Cor}\label{c:ibn.contracted}
Let $M$ be a monoid and $I$ a proper ideal (possibly empty).  Then $\mathbb ZM/\mathbb ZI$ has IBN.  In particular, contracted monoid rings have IBN.
\end{Cor}
\begin{proof}
Put $R=\mathbb ZM/\mathbb ZI$.  If $I$ is a proper ideal, then $1$ is not conjugate to any element of $I$ by Lemma~\ref{l:conjugate.to.idem}.  It follows from Proposition~\ref{p:nielsen} that $1+[R,R]$ has infinite order in $R/[R,R]$ and hence $R$ has IBN.
\end{proof}

\begin{Thm}\label{t:no.finite.free}
Let $M$ be a non-trivial monoid with a finitely many right zeroes (and at least one).  Then the trivial left $\mathbb ZM$-module $\mathbb Z$ is projective but not stably free and hence does not have a finite free resolution that is finitely generated in each degree.
\end{Thm}
\begin{proof}
Let $I$ be the set of right zero elements of $M$ and fix $z\in I$. Observe that $I$ is a proper two-sided ideal.    Note that $z,1-z$ form a complete set of orthogonal idempotents of $\mathbb ZM$ and so $\mathbb ZM\cong \mathbb ZMz\oplus \mathbb ZM(1-z)$ and hence $\mathbb ZMz\cong \mathbb Z$ is projective.  Suppose that $\mathbb Z$ is stably free, that is, $\mathbb Z\oplus F'\cong F$ with $F,F'$ free $\mathbb ZM$-modules of rank $r,r'$, respectively.

There is a exact functor from $\mathbb ZM$-modules to $z\mathbb ZMz$-modules given by $V\mapsto zV$.  Note that $z\mathbb ZMz=\mathbb Zz\cong \mathbb Z$ as a ring.  Also, $z\mathbb ZM= \mathbb ZI$ is a free abelian group (equals $z\mathbb ZMz$-module) of rank $|I|<\infty$.
Therefore, \[\mathbb Z^{|I|r}\cong (\mathbb ZI)^r\cong zF\cong \mathbb
Z\oplus zF'\cong \mathbb Z\oplus \mathbb (ZI)^{r'}\cong \mathbb
Z^{1+|I|r'}\] as $\mathbb Z$-modules and hence $r|I|=1+r'|I|$ as
$\mathbb Z$ has IBN.  But putting $R=\mathbb ZM/\mathbb ZI$ and observing that $\mathbb Z/\mathbb ZI\cdot \mathbb Z=0$, we have
that
\[
R^r\cong R\otimes_{\mathbb ZM} F\cong (R\otimes_{\mathbb ZM}\mathbb
Z) \oplus (R\otimes_{\mathbb ZM}F') \cong R^{r'}
\]
and hence $r=r'$ as $R$ has IBN by Corollary~\ref{c:ibn.contracted}.  This contradiction completes the proof.
\end{proof}

There is, of course, a dual result for left zeroes.  In particular, if $M$ is  non-trivial monoid with a zero element, then $\mathbb Z$ is not stably free as either a left or right $\mathbb ZM$-module.  Thus, if $M$ is a non-trivial monoid with zero, it has no $M$-finite free left or right equivariant classifying space but it has an $M$-finite  projective one.  This justifies considering projective $M$-CW complexes.

If $L$ is a left ideal of $M$ containing an identity $e$, then $L=Me=eMe$.  Note that $\p\colon M\to Me$ given by $\p(m)=me$ is a surjective monoid homomorphism in this case since $\p(1)=e$ and $\p(mn) = mne=mene=\p(m)\p(n)$ as $ne\in Me=eMe$. Also note that the left $M$-set structure on $Me$ given by inflation along $\p$ corresponds to the left $M$-set structure on $Me$ induced by left multiplication because if $m\in M$ and $n\in Me$, then $n=en$ and so $mn=men=\p(m)n$.  Notice that if $f\in E(Me)$, then $f\in E(M)$ and $Mef=Mf$ is projective as both an $Me$-set and an $M$-set.  Thus each (finitely generated) projective $Me$-set is a (finitely generated) projective $M$-set via inflation along $\p$. Note that if $A$ is a left $M$-set, then $eM\otimes_M A\cong eA$ via $em\otimes a\mapsto ema$.

\begin{Prop}\label{p:left.ideal.fn}
Suppose that $M$ and $e\in E(M)$ with $Me=eMe$ and $0\leq n\leq\infty$.  If $Me$ is of type left-$\F_n$, then so is $M$.  The converse holds if $eM$ is a finitely generated projective left $eMe$-set.
\end{Prop}
\begin{proof}
If $X$ is a projective $Me$-CW complex constructed via push\-outs as in \eqref{eq:pushout} (but with $M$ replaced by $Me$), then each $P_n$ is a projective $M$-set via inflation along $\p$ and so $X$ is a  projective $M$-CW complex. Moreover, if $X$ is of $Me$-finite type (respectively, $Me$-finite), then it is of $M$-finite type (respectively, $M$-finite).  Thus if $Me$ is of type left-$\F_n$, then so is $M$.

Suppose that $X$ is an equivariant classifying space for $M$ and $eM$ is a finitely generated projective left $eMe$-set. Then  $eM\otimes_M X\cong eX$.  Now $eX$ is a projective $eMe$-CW complex and if $X_n$ is $M$-finite, then $(eX)_n=eX_n$ is $eMe$-finite by Corollary~\ref{c:base.change.cw}.  Moreover, since $eX$ is a retract of $X$ as a CW complex and $X$ is contractible, it follows that $eX$ is contractible.  Thus $eX$ is an equivariant classifying space for $eMe$.  The result follows.
\end{proof}

Our first corollary is that having a right zero guarantees the property left-$\F_{\infty}$ (which can be viewed as a defect of the one-sided theory).

\begin{Cor}\label{c:right.zero}
If $M$ contains a right zero, then $M$ is of type left-$\F_{\infty}$. Hence any monoid with a zero is both of type left-  and right-$\F_{\infty}$.
\end{Cor}
\begin{proof}
If $e$ is a right zero, then $Me=\{e\}=eMe$ and $\{e\}$ is of type left-$\F_{\infty}$.  Thus $M$ is of type left-$\F_{\infty}$ by Proposition~\ref{p:left.ideal.fn}.
\end{proof}

Recall that two elements $m$ and $n$  a monoid $M$ are said to be
$\mathscr{L}$-related if and only if they generate the same principal
left ideal, i.e., if $Mm = Mn$. Clearly $\mathscr{L}$ is an equivalence
relation on $M$.

\begin{Cor}\label{c:min.left.ideal}
Suppose that $M$ is a monoid and $e\in E(M)$ with $eM$ a two-sided minimal ideal of $M$ and $0\leq n\leq \infty$. Note that $G_e=eMe$ is the maximal subgroup at $e$.  If $G_e$ is of type-$\F_n$, then $M$ is of type left-$\F_n$.  If $eM$ contains finitely many $\mathscr L$-classes, then the converse holds.
\end{Cor}
\begin{proof}
Note that $Me=eMe=G_e$ and so the first statement is immediate from Proposition~\ref{p:left.ideal.fn}.  For the converse, it follows from Green's lemma \cite[Lemma~2.2.1]{Howie} that $eM$ is free left $G_e$-set and that the orbits are the $\mathscr L$-classes of $eM$.  Thus the second statement follows from Proposition~\ref{p:left.ideal.fn}.
\end{proof}

Corollary~\ref{c:min.left.ideal} implies that if $M$ is a monoid with a minimal ideal that is a group $G$, then $M$ is of type left-$\F_n$ if and only if $G$ is of type $\F_n$ and dually for right-$\F_n$.

The following is a slight extension of the fact that a finite index subgroup of a group of type $\F_n$ is also of type $\F_n$; see \cite[Chapter~VIII, Proposition~5.1]{BrownCohomologyBook}.
\begin{Prop}\label{p:finite.index}
Let $M$ be a monoid and $N$ a submonoid such that $M$ is a finitely generated projective left $N$-set.  If $M$ is of type left-$\F_n$, then $N$ is of type left-$\F_n$, as well.
\end{Prop}
\begin{proof}
Observe that each finitely generated free left $M$-set is a finitely generated projective $N$-set.  Hence each finitely generated projective $M$-set, being a retract of a finitely generated free left $M$-set, is a finitely generated projective $N$-set.  Thus any equivariant classifying space for $M$ is also an equivariant classifying space for $N$.
\end{proof}

An immediate consequence of Proposition~\ref{p:direct.prod} is the following.

\begin{Prop}\label{p:finiteness.dp}
Let $M,N$ be monoids of type left-$\F_n$.  Then $M\times N$ is of type left-$\F_n$.
\end{Prop}

\subsection{Left geometric dimension}

Let us define the \emph{left geometric dimension} of $M$ to be the minimum dimension of a left equivariant classifying space for $M$.  The right geometric dimension is, of course, defined dually.  Clearly, the geometric dimension is an upper bound on the cohomological dimension $\mathop{\mathrm{cd}} M$ of $M$. Recall that the (left) \emph{cohomological dimension} of $M$ is the projective dimension of the trivial module $\mathbb Z$, that is, the shortest length of a projective resolution of $\mathbb Z$. As mentioned in the introduction, for groups of cohomological dimension different than $2$, it is known that geometric dimension coincides with cohomological dimension, but the general case is open.

\begin{Thm}\label{t:geom.dim}
Let $M$ be a monoid.  Then $M$ has an equivariant classifying space of dimension $\max\{\mathop{\mathrm{cd}} M,3\}$.
\end{Thm}
\begin{proof}
If $M$ has infinite cohomological dimension, then this is just the assertion that $M$ has an equivariant classiyfing space.  So assume that $\mathop{\mathrm{cd}} M<\infty$.  Put $n=\max\{\mathop{\mathrm{cd}} M,3\}$.  Let $Y$ be an equivariant classifying space for $M$.  As the inclusion $Y_{n-1}\to Y$ is an $(n-1)$-equivalence, we deduce that $\pi_q(Y_{n-1})$ is trivial for $0\leq q<n-1$.  Also, as the augmented cellular chain complex of $Y_{n-1}$ provides a partial projective resolution of the trivial module of length $n-1$ and $\mathop{\mathrm{cd}} M\leq n$, it follows that $\ker d_{n-1}=H_{n-1}(Y_{n-1})$ is a projective $\mathbb ZM$-module.  By the Eilenberg swindle, there is a free $\mathbb ZM$-module $F$ such that $H_{n-1}(Y_{n-1})\oplus F\cong F$.  Suppose that $F$ is free on a set $A$.  Fix a basepoint $y_0\in Y_{n-1}$.  We glue a wedge of $(n-1)$-spheres, in bijection with $A$, into $Y_{n-1}$ at $y_0$ as well as freely gluing in its translates.   That is we form a new projective $M$-CW complex $Z$ with $Z_{n-2}=Y_{n-2}$ and where $Z=Z_{n-1}$ consists of the $(n-1)$-cells from $Y_{n-1}$ and $M\times A\times B^{n-1}$ where the attaching map $M\times A\times S^{n-2}$ is given by $(m,a,x)\mapsto my_0$.

Notice that $C_{n-1}(Z)\cong  C_{n-1}(Y_{n-1})\oplus F$ as a $\mathbb ZM$-module and that the boundary map is zero on the $F$-summand since the boundary of each of the new $(n-1)$-cells that we have glued in is a point and $n\geq 3$.  Therefore, $H_{n-1}(Z)=\ker d_{n-1}=H_{n-1}(Y_{n-1})\oplus F\cong F$.  As the inclusion $Y_{n-2}=Z_{n-2}\to Z$ is an $(n-2)$-equivalence, we deduce that $\pi_q(Z)$ is trivial for $0\leq q\leq n-2$.  In particular, $Z$ is simply connected as $n\geq 3$. By the Hurewicz theorem, $\pi_{n-1}(Z,y_0)\cong H_{n-1}(Z)$. Choose mappings $f_a\colon S^{n-1}\to Z$, for $a\in A$, whose images under the Hurewicz mapping from a $\mathbb ZM$-module basis for $H_{n-1}(Z)\cong F$.  Then form $X$ by attaching $M\times A\times B^n$  to $Z=Z_{n-1}$ via the mapping \[M\times A\times S^{n-1}\to Z\] sending $(m,a,x)$ to $mf_a(x)$.

Note that $X$ is an $n$-dimensional CW complex with $X_{n-1}=Z_{n-1}$ and hence the inclusion $Z=X_{n-1}\to X$ is an $(n-1)$-equivalence.  Therefore, $\pi_q(X)=0=H_q(X)$ for $0\leq q\leq n-2$.  Also $\pi_{n-1}(X,y_0)\cong H_{n-1}(X)$ via the Hurewicz isomorphism.  Moreover, as the inclusion $Z=X_{n-1}\to X$ is an $(n-1)$-equivalence, we deduce that the inclusion induces a surjective homomorphism $\pi_{n-1}(Z,y_0)\to \pi_{n-1}(X,y_0)$ and hence a surjective homomorphism $H_{n-1}(Z)\to H_{n-1}(X)$.  As the $\mathbb ZM$-module generators of $H_{n-1}(Z)$ have trivial images in $H_{n-1}(X)$ by construction and the Hurewicz map, we deduce that $H_{n-1}(X)=0$.

Recall that $C_n(X)=H_n(X_n,X_{n-1})$. By standard cellular homology $H_n(X_n,X_{n-1})$ is a free $\mathbb ZM$-module on the images of the generator  of the relative homology of $(B^n,S^{n-1})$ under the characteristic mappings \[h_a\colon (\{1\}\times \{a\}\times B^n,\{1\}\times \{a\}\times S^{n-1}) \to (X_n,X_{n-1})\] and the boundary map $\partial_n \colon H_n(X_n,X_{n-1})\to H_{n-1}(X_{n-1})$ sends  the class corresponding to $a\in A$ to the image of the generator of $S^{n-1}$ under the map on homology induced by the attaching map $f_a\colon S^{n-1}\to X_{n-1}$.  Hence a free basis of $H_n(X_n,X_{n-1})$ is sent by $\partial_n$ bijectively to a free basis for $H_{n-1}(X_{n-1})$ and so $\partial_n$ is an isomorphism.  The long exact sequence for reduced homology and the fact that an $(n-1)$-dimensional CW complex has trivial homology in degree $n$ provides an exact sequence
\[0=H_n(X_{n-1})\longrightarrow H_n(X_n)\longrightarrow H_n(X_n,X_{n-1})\xrightarrow{\partial_n} H_{n-1}(X_{n-1})\] and so $H_n(X)=H_n(X_n)\cong \ker \partial_n=0$.  As $X$ is a simply connected $n$-dimensional CW complex with $H_q(X)=0$ for $0\leq q\leq n$, we deduce that $X$ is contractible by the Hurewicz and Whitehead theorems.  Therefore, $X$ is an $n$-dimensional equivariant classifying space for $M$, completing the proof.
\end{proof}

We end this section by observing that monoids of left cohomological dimenison $0$ are precisely the monoids of left geometric dimension $0$.
The following result generalises \cite[Lemma~1 and Theorem~1]{Guba1998}.

\begin{Prop}\label{p:coh.dim.zero}
Let $M$ be a monoid.  Then the following are equivalent.
\begin{enumerate}
\item $M$ has a right zero element.
\item $M$ has  left cohomological dimension $0$.
\item $M$ has left geometric dimension $0$.
\end{enumerate}
\end{Prop}
\begin{proof}
If $M$ has a right zero $z$, then $Mz=\{z\}$ is a projective $M$-set and hence the one point space is an equivariant classifying space for $M$.  Thus $M$ has left geometric dimension zero.  If $M$ has left geometric dimension zero, then it has left cohomological dimension zero.  If $M$ has left cohomological dimension zero, then $\mathbb Z$ is a projective $\mathbb ZM$-module and so the augmentation mapping $\varepsilon\colon  \mathbb ZM\to \mathbb Z$ splits.  Let $P$ be the image of the splitting, so that $\mathbb ZM=P\oplus Q$.  As $P$ is a retract of $\mathbb ZM$ and each endomorphism of $\mathbb ZM$ is induced by a right multiplication, we have that $\mathbb Z\cong P=\mathbb ZMe$ for some idempotent $e\in \mathbb ZM$ with $\varepsilon(e)=1$.  Then since $me=e$ for all $m\in M$ and $e$ has finite support $X$, we must have that $M$ permutes $X$ under left multiplication.  Let $G$ be the quotient of $M$ that identifies two elements if they act the same on $X$.  Then $G$ is a finite group and $\mathbb Z$ must be a projective $\mathbb ZG$-module.  Therefore, $G$ is trivial.  But this means that if $x\in X$, then $mx=x$ for all $m\in M$ and so $M$ has a right zero.
\end{proof}

We do not know whether left geometric dimension equals left cohomological dimension for monoids of cohomological dimension one or two, although the former is true for groups by the Stallings-Swan theorem.

\section{Bi-equivariant classifying spaces}\label{subsec_bi-equi_CS}
Let $M$ be a monoid.  We now introduce the bilateral notion of a classifying space in order to introduce a stronger property, bi-$\F_n$.  It will turn out that bi-$\F_n$ implies both left-$\F_n$ and right-$\F_n$, but is strictly stronger.  Moreover, bi-$\F_n$ implies bi-$\FP_n$ which is of interest from the point of view of Hochschild cohomology, which is the standard notion of cohomology for rings. Many of the results are similar to the previous section, but the proofs are more complicated.

First recall that $M$ is an $M\times M^{op}$-set via the action $(m_L,m_R)m=m_Lmm_R$.  We say that a projective $M\times M^{op}$-CW complex $X$ is a \emph{bi-equivariant classifying space for $M$} if $\pi_0(X)\cong M$ as an $M\times M^{op}$-set and each component of $X$ is contractible; equivalently, $X$ has an $M\times M^{op}$-equivariant homotopy equivalence to the discrete $M\times M^{op}$-set $M$.

We can augment the cellular chain complex of $X$ via the canonical surjection $\varepsilon\colon  C_0(X)\to H_0(X)\cong \mathbb Z\pi_0(X)\cong \mathbb ZM$.  The fact that each component of $X$ is contractible guarantees that this is a resolution, which will be a projective bimodule resolution of $\mathbb ZM$ and hence suitable for computing Hochschild cohomology.  We begin by establishing the uniqueness up to $M\times M^{op}$-homotopy equivalence of bi-equivariant classifying spaces.

\begin{Lemma}\label{l:construct.map.bi}
Let $X$ be a bi-equivariant classifying space for $M$ and let $Y$ be a locally path connected $M\times M^{op}$-space with contractible connected components.  Suppose that $g\colon \pi_0(X)\to \pi_0(Y)$ is an $M\times M^{op}$-equivariant mapping.  Then there exists a continuous $M\times M^{op}$-equivariant mapping  $f\colon X\to Y$ such that the mapping $f_\ast\colon \pi_0(X)\to \pi_0(Y)$ induced by $f$ is $g$.
\end{Lemma}
\begin{proof}
Let $r\colon X\to \pi_0(X)$ and $k\colon Y\to \pi_0(Y)$ be the projections to the set of connected components. Then $k$ and $r$ are continuous $M\times M^{op}$-equivariant maps where $\pi_0(X)$ and $\pi_0(Y)$ carry the discrete topology.  Our goal will be to construct an $M\times M^{op}$-equivariant continuous mapping $f\colon X\to Y$ such that the diagram
\[\begin{tikzcd}X\ar{r}{f}\ar{d}[swap]{r} &Y\ar{d}{k}\\ \pi_0(X)\ar{r}[swap]{g}  & \pi_0(Y)\end{tikzcd}\] commutes.
We construct, by induction, $M\times M^{op}$-equivariant continuous mappings $f_n\colon X_n\to Y$ such that
\begin{equation}\label{eq:must.comm}
\begin{tikzcd}X_n\ar{r}{f_n}\ar{d}[swap]{r} &  Y\ar{d}{k}\\  \pi_0(X)\ar{r}[swap]{g} &\pi_0(Y)\end{tikzcd}
\end{equation}
 commutes and $f_n$ extends $f_{n-1}$.

To define $f_0$, observe that $X_0 = \coprod_{a\in A}Me_a\times e'_aM$.  Choose $y_a\in Y$ with $k(y_a)=g(r(e_a,e'_a))$.  Then $k(e_aye'_a)=e_ak(y_a)e'_a=e_ag(r(e_a,e'_a))e'_a=g(r(e_a,e'_a))$ and so replacing $y_a$ by $e_aye'_a$, we may assume without loss of generality that $y_a\in e_aYe'_a$.  Then by  Proposition~\ref{p:cell-maps} there is an $M\times M^{op}$-equivariant mapping $X_0\to Y$ given by $(me_a,e'_am')\mapsto me_ay_ae'_am'$ for $a\in A$ and $m\in M$.  By construction, the diagram~\eqref{eq:must.comm} commutes.

Assume now that $f_n$ has been defined. The map $k\colon Y\to \pi_0(Y)$ is $M\times M^{op}$-equivariant and a weak equivalence (where $\pi_0(Y)$ has the discrete topology) because $Y$ has contractible connected components.  So by Theorem~\ref{t:HELP} we can construct a commutative diagram
\[\begin{tikzcd}X_n\ar{rr}{i_0}\ar[hook]{dd}[swap]{i} & & X_n\times I\ar{dl}[swap]{gr\pi_{X_n}}\ar[hook]{dd} & & X_n\ar{ll}[swap]{i_1}\ar[hook]{dd}{i}\ar{dl}{f_n}\\
                              &\pi_0(Y)&           &\ar[crossing over]{ll}[swap]{k}Y&   \\
                X_{n+1}\ar{rr}{i_0}\ar{ur}{gr}             & & X_{n+1}\times I\ar[dashrightarrow]{ul}[swap]{} & & X_{n+1}\ar{ll}[swap]{i_1}\ar[dashrightarrow]{ul}[swap]{f_{n+1}}\end{tikzcd}\]
where $f_{n+1}$ is $M\times M^{op}$-equivariant and $\pi_{X_n}$ is the projection.  Note that $kf_{n+1}\simeq gr$ and hence, since $\pi_0(Y)$ is a discrete space, we conclude that $kf_{n+1}=gr$.  Now take $f$ to be the colimit of the $f_n$.  This completes the proof.
\end{proof}

\begin{Thm}\label{t:uniqueness.bi}
Let $X,Y$ be bi-equivariant classifying spaces for $M$.  Then $X$ and $Y$ are $M\times M^{op}$-homotopy equivalent by a cellular $M\times M^{op}$-homotopy equivalence.
\end{Thm}
\begin{proof}
As $\pi_0(X)\cong M\cong \pi_0(Y)$ as $M\times M^{op}$-sets, there is an $M\times M^{op}$-equivariant isomorphism $g\colon \pi_0(X)\to \pi_0(Y)$.  Then by Lemma~\ref{l:construct.map.bi}, there is an $M\times M^{op}$-equivariant continuous mapping $f\colon X\to Y$ inducing $g$ on connected components.  It follows that $f$ is a weak equivalence as $X$ and $Y$ both have contractible connected components.
The result now follows  from Corollary~\ref{c:whitehead} and Theorem~\ref{t:cell.approx}.
\end{proof}

Next we prove in an elementary fashion that bi-equivariant classifying spaces for $M$ exist.  A more canonical construction, using simplicial sets, was described earlier.

\begin{Lemma}\label{l:improve.bi}
Let $M$ be a monoid.
\begin{enumerate}
  \item If $X$ is a projective (free) $M\times M^{op}$-CW complex such that $\pi_0(X)\cong M$ and $\pi_q(X,x)=0$ for all $1\leq q<n$ and $x\in X$, then there exists a projective $M\times M^{op}$-CW complex $Y$ containing $X$ as a projective $M\times M^{op}$-CW subcomplex and such that $Y_n=X_n$ and $\pi_q(Y,y)=0$ for all $y\in Y$ and $1\leq q\leq n$.
  \item If $X$ is a  projective (free) $M\times M^{op}$-CW complex such that $\pi_0(X)\cong M$ and $\pi_q(X,x)=0$ for all $1\leq q<n$ and $x\in X$, then there exists a projective (free) $M\times M^{op}$-CW complex $Y$ with contractible connected components containing $X$ as a projective $M\times M^{op}$-CW subcomplex and such that $Y_n=X_n$.
\end{enumerate}
\end{Lemma}
\begin{proof}
  This is a minor adaptation of the proof of Lemma~\ref{l:improve}
  that we leave to the reader.
\end{proof}

\begin{Cor}\label{c:exist.fbi.equi}
Let $M$ be a monoid.  Then there exists a free $M\times M^{op}$-CW complex $X$ with $\pi_0(X)\cong M$ and each connected component of $X$ contractible.
\end{Cor}
\begin{proof}
By Lemma~\ref{l:improve.bi} it suffices to construct a free $M\times M^{op}$-graph $X$ with $\pi_0(X)\cong M$.  We take $X_0=M\times M$ and we take an edge set in bijection with $M\times M\times M$.  The edge $(m_L,m_R,m)$ will connect $(m_L,mm_R)$ to $(m_Lm,m_R)$.  Then $X$ is a free $M\times M^{op}$-graph.  Notice that if $(m_1,m_2)$ is connected by an edge to $(m_1',m_2')$, then $m_1m_2=m_1'm_2'$.  On the other hand, $(1,m_2,m_1)$ is an edge from $(1,m_1m_2)$ to $(m_1,m_2)$ and hence there is a bijection $\pi_0(X)\to M$ sending the component of $(m_1,m_2)$ to $m_1m_2$ and this mapping is an $M\times M^{op}$-equivariant bijection.
\end{proof}

\begin{example}
  It follows from the definitions and results in
  Section~\ref{sec_cspaces} that the geometric realisation $|\EMb|$ of
    the nerve of the two-sided Cayley graph category of $M$ is a
    bi-equivariant classifying space for $M$.
\end{example}

\begin{Cor}\label{c:class.quot.bi}
If $X$ is a bi-equivariant classifying space for $M$, then $M\backslash X/M\simeq |BM|$.
\end{Cor}
\begin{proof}
We have $M\backslash |\overleftrightarrow{EM}|/M\cong |BM|$. The result now follows from Theorem~\ref{t:uniqueness.bi} and Proposition~\ref{p:preservation.hom.equiv}.
\end{proof}

Another important definition for this paper is the following.  A
monoid $M$ is of type \emph{bi-$\F_n$} if there is a bi-equivariant
classifying space $X$ for $M$ such that $X_n$ is
$M\times M^{op}$-finite, i.e., $M\backslash X/M$ has finite
$n$-skeleton.  We say that $M$ is of type \emph{bi-$\F_{\infty}$} if
$M$ has a bi-equivariant classifying space $X$ that is of
$M\times M^{op}$-finite type, i.e., $M\backslash X/M$ is of finite
type.
Clearly by making use of the canonical two-sided classifying space
$|\EMb|$ we can immediately conclude that any finite monoid is of type
bi-$\Finfty$.

%
%

Recall that a monoid $M$ is said to be of type \emph{bi-$\FPn$} if
there is a partial resolution of the $(\mathbb ZM,\mathbb ZM)$-bimodule $\mathbb ZM$
\[
A_n \rightarrow A_{n-1} \rightarrow \cdots \rightarrow A_1 \rightarrow
A_0 \rightarrow \mathbb ZM \rightarrow 0
\]
where $A_0, A_1, \ldots, A_n$ are finitely generated projective
$(\mathbb{Z}M, \mathbb{Z}M)$-bimodules. Monoids of type bi-$\FPn$ were
studied by Kobayashi and Otto in \cite{KobayashiOtto2001}. We note
that this differs from the definition of bi-$\FPn$ considered in
\cite{AlonsoHermiller2003}, which is called \emph{weak bi-$\FPn$} by
Pride in \cite{Pride2006} where it is shown to be equivalent to being
simultaneously of type left- and right-$\FPn$. In this paper by
bi-$\FPn$ we shall always mean bi-$\FPn$ in the above sense of Kobayashi and
Otto. The property bi-$\FPn$ is of interest because of its connections
with the study of Hochschild cohomology
\cite[Chapter~9]{Weibel1994}. Kobayashi investigated bi-$\FPn$ in
\cite{Kobayashi2005, Kobayashi2007, Kobayashi2010} proving, in
particular, that any monoid which admits a presentation by a finite
complete rewriting system is of type bi-$\FPn$. This has applications
for the computation of Hochschild cohomology. We shall recover this
theorem of Kobayashi below 
in Section~\ref{sec_admitting}
as an application of our results on
equivariant discrete Morse theory and collapsing schemes. See also
\cite{Pasku2008} for further related results on bi-$\FPn$.

The following result relates bi-$\F_n$ with bi-$\FP_n$.

\begin{Prop}\label{p:basic.props.bi.fn}
Let $M$ be a monoid.
\begin{enumerate}
\item For $0\leq n\leq \infty$, if $M$ is of type bi-$\F_n$, then it is of type bi-$\FP_n$.
\item If $M$ is of type bi-$\F_{\infty}$, then it is of type bi-$\F_n$ for all $n\geq 0$.
\item If $M$ is of type bi-$\F_n$ for $0\leq n\leq\infty$, then $M$ is of type left-$\F_n$ and type right-$\F_n$.
\item For $0\leq n\leq\infty$, a group is of type bi-$\F_n$ if and only if it is of type $\F_n$.
\end{enumerate}
\end{Prop}
\begin{proof}
The first item follows using that the  cellular chain complex of a bi-equivariant classying space $X$ can be augmented, as discussed earlier, to give a bimodule resolution of $\mathbb ZM$ and that if $X$ is built up from pushouts as per \eqref{eq:pushout} (with $M\times M^{op}$ in place of $M$), then the $n^{th}$-chain module is isomorphic to $\mathbb ZP_n$ as a bimodule and hence is projective.  The second item is trivial.

For the third item, one verifies that if $\overleftrightarrow{EM}$ is the two-sided bar construction, then $\overleftarrow{EM}\cong \overleftrightarrow{EM}/M$ where $\overleftarrow{EM}$ is the left bar construction.  Suppose now that $X$ is a bi-equivariant classifying space for $M$ such that $X_n$ is $M\times M^{op}$-finite.  Then $X\simeq_{M\times M^{op}} \overleftrightarrow{EM}$ by Theorem~\ref{t:uniqueness.bi}.  Therefore, $X/M\simeq_M |\overleftrightarrow{EM}|/M=|\overleftarrow{EM}|$ and $X/M$ is a projective $M$-CW complex with $(X/M)_n=X_n/M$ being $M$-finite by Corollary~\ref{c:one-sided}.  Thus if $M$ is of type bi-$\F_n$ for $0\leq n\leq \infty$, then $M$ is of type left-$\F_n$ and dually right-$\F_n$.

If $G$ is a group of type bi-$\F_n$, then it is of type $\F_n$ by the previous item and Proposition~\ref{p:basic.props.fn}. Conversely, suppose that $X$ is a free $G$-CW complex with $G$-finite $n$-skeleton.  Then using the right $G$-set structure on $G\times G$ from Proposition~\ref{p:easy.tensor} we have that $Y=(G\times G)\otimes_G X$ is a projective $G\times G^{op}$-CW complex by Proposition~\ref{c:base.change.cw} such that $Y_n$ is $G\times G^{op}$-finite.  Moreover, $\pi_0(Y)\cong (G\times G)\otimes_G \pi_0(X)$ by Proposition~\ref{p:components}.  But $\pi_0(X)$ is the trivial $G$-set and $(G\times G)\otimes_G 1\cong G$ as a $G\times G^{op}$-set via $(g,h)\otimes 1\mapsto gh$.  Finally, since $G$ is a free right $G$-set of on $G$-generators by Proposition~\ref{p:easy.tensor}, it follows that as a topological space $Y=\coprod_G X$ and hence each component of $Y$ is contractible.  This completes the proof.
\end{proof}


The proof of Proposition~\ref{p:basic.props.bi.fn} establishes the following proposition.

\begin{Prop}\label{p:two.to.one}
If $X$ is a bi-equivariant classifying space for $M$, then $X/M$ is an equivariant classifying space for $M$.
\end{Prop}

Sometimes it will be convenient to use the following reformulation of the property bi-$\F_n$.

\begin{Prop}\label{p:reform.bi.fn}
Let $M$ be a monoid.  The following are equivalent for $0\leq n<\infty$.
\begin{enumerate}
  \item $M$ is of type bi-$\F_n$
  \item There is an $M\times M^{op}$-finite  projective $M\times M^{op}$-CW complex $X$ of dimension at most $n$ with $\pi_0(X)\cong M$ and  $\pi_q(X,x)=0$ for $1\leq q<n$ and $x\in X$.
\end{enumerate}
\end{Prop}
\begin{proof}
This is entirely analogous to the proof of Proposition~\ref{p:reform.fn}.
\end{proof}

If $M$ is a monoid and $A\subseteq M$, then the \emph{two-sided Cayley digraph} $\overleftrightarrow{\Gamma(M,A)}$ is the digraph with vertex set $M\times M$ and with edges in  bijection with elements of $M\times M\times A$.  The directed edge $(m_L,m_R,a)$ goes from $(m_L,am_R)$ to $(m_La,m_R)$ and we draw it as
\[(m_L,am_R)\xrightarrow{\,\,a\,\,} (m_La,m_R).\]  Note that $\overleftrightarrow{\Gamma(M,A)}$ is a free $M\times M^{op}$-graph and is $M\times M^{op}$-finite if and only if $A$ is finite.  Also note that if $(m_1,m_2)$ is connected to $(m_1',m_2')$ by an edge, then $m_1m_2=m_1'm_2'$.  Hence multiplication of the coordinates of a vertex induces a surjective $M\times M^{op}$-equivariant mapping $\pi_0(\overleftrightarrow{\Gamma(M,A)})\to M$.  If $A$ is a generating set for $M$, then the mapping is an isomorphism because if $m_1,m_2\in M$ and $u\in A^*$ is a word representing $m_1$, then there is a directed path labelled by $u$ from $(1,m)$ to $(m_1,m_2)$.  Namely, if $u=a_1\cdots a_k$ with $a_i\in A$, then the path labelled by $u$ from $(1,m)$ to $(m_1,m_2)$ is
\begin{equation}\label{eq:two.sided.path}
(1,m)\xrightarrow{\,\,a_1\,\,} (a_1,a_2\cdots a_km_2)\xrightarrow{\,\,a_2\,\,}(a_1a_2,a_3\cdots a_km_2)\xrightarrow{\,\,a_3\,\,}\cdots\xrightarrow{\,\,a_k\,\,}(m_1,m_2).
\end{equation}

A monoid $M$ is said to be \emph{dominated} by a subset $A$ if whenever $f,g\colon M\to N$ are monoid homomorphisms with $f|_A=g|_A$, one has $f=g$.  In other words, the inclusion $\langle A\rangle\hookrightarrow M$ is an epimorphism (in the category theory sense). Of course, a generating set of $M$ dominates $M$.  Note that if $A$ is a subset of an inverse monoid $M$ (e.g., a group), then $A$ dominates $M$ if and only if $A$ generates $M$ as an inverse monoid. Hence $M$ is finitely generated if and only if $M$ is finitely dominated.  Kobayashi gives an example of an infinitely generated monoid that is finitely dominated. See~\cite{Kobayashi2007} for details.

\begin{Thm}\label{t:bi.f1}
The following are equivalent for a monoid $M$.
\begin{enumerate}
  \item $M$ is of type bi-$\F_1$.
  \item $M$ is of type bi-$\FP_1$.
  \item There is a finite subset $A\subseteq M$ such that the natural mapping $\pi_0(\overleftrightarrow{\Gamma(M,A)})\to M$ is an isomorphism.
  \item There is a finite subset $A\subseteq M$ that dominates $M$.
\end{enumerate}
In particular, any finitely generated monoid is of type bi-$\F_1$.
\end{Thm}
\begin{proof}
The equivalence of (2) and (4) was established by Kobayashi~\cite{Kobayashi2007} using Isbel's zig-zag lemma (actually, the equivalence of (3) and (4) is also direct from Isbel's zig-zag lemma).

Assume that (3) holds.  Then $\overleftrightarrow{\Gamma(M,A)}$ is $M\times M^{op}$-finite and so $M$ is of type bi-$\F_1$ by Proposition~\ref{p:reform.bi.fn}.  Proposition~\ref{p:basic.props.bi.fn} shows that (1) implies (2).

Assume that $M$ is of type bi-$\FP_1$.  Then we have a partial free resolution
\[\mathbb ZM\otimes \mathbb ZM\xrightarrow{\,\,\mu\,\,} \mathbb ZM\longrightarrow 0\] of finite type where $\mu$ is induced by the multiplication in $\mathbb ZM$.  Since $M$ is of type bi-$\FP_1$, $\ker \mu$ is finitely generated.  It is well known that $\ker \mu$ is generated as a $\mathbb ZM\otimes \mathbb ZM^{op}$-module by the elements $m\otimes 1-1\otimes m$ with $m\in M$.  Indeed, if $\sum c_im_i\otimes n_i$ is in $\ker \mu$, then $\sum c_im_in_i=0$ and so
\begin{align*}
\sum c_im_i\otimes n_i &= \sum c_im_i\otimes n_i-\sum c_i(1\otimes m_in_i)\\ &= \sum c_i(m_i\otimes n_i-1\otimes m_in_i)\\ & = \sum c_i(m_i\otimes 1-1\otimes m_i)n_i.
\end{align*}
 Hence there is a finite subset $A\subseteq M$ such that $\ker\mu$ is generated by the elements $a\otimes 1-1\otimes a$ with $a\in A$.  We claim that the natural surjective mapping $\pi_0(\overleftrightarrow{\Gamma(M,A)})\to M$ is an isomorphism.

Identifying $\mathbb Z[M\times M^{op}]$ with $\mathbb ZM\otimes \mathbb ZM^{op}$ as rings and  $\mathbb Z[M\times M]$ with $\mathbb ZM\otimes \mathbb ZM$ as bimodules, we have a bimodule homomorphism \[\lambda\colon \mathbb Z[M\times M]\to \mathbb Z\pi_0(\overleftrightarrow{\Gamma(M,A)})\] sending $(m_L,m_R)$ to its connected component in  $\overleftrightarrow{\Gamma(M,A)}$ and $\mu$ factors as $\lambda$ followed by the natural mapping \[\mathbb Z\pi_0(\overleftrightarrow{\Gamma(M,A)})\to \mathbb ZM.\] Clearly, $\ker \lambda\subseteq \ker \mu$ and so to prove the result it suffices to show that $\ker \mu\subseteq \ker \lambda$. Now $\ker \mu$ is generated by the elements $(1,a)-(a,1)$ with $a\in A$ under our identifications.  But $(1,1,a)$ is an edge from $(1,a)$ to $(a,1)$.  Thus $(1,a)-(a,1)\in \ker \lambda$ for all $a\in A$.  This establishes that (2) implies (3), thereby completing the proof.
\end{proof}

If $G$ is a group, then it follows from Theorem~\ref{t:bi.f1} that $G\cup \{0\}$ is of type bi-$\F_1$ if and only if $G$ is finitely generated.  Indeed, $G\cup \{0\}$ is an inverse monoid and hence finitely dominated if and only if finitely generated.  But $G\cup \{0\}$ is finitely generated if and only if $G$ is finitely generated.
On the other hand, $G\cup \{0\}$ is both of type left- and  right-$\F_{\infty}$ for any group $G$ by Corollary~\ref{c:right.zero}.  Thus bi-$\F_n$ is a much stronger notion.

\begin{Rmk}
It can be shown that if $M$ is a monoid and $M^0$ is the result of adjoining a $0$ to $M$, then if $M^0$ is of type bi-$\F_n$, then $M$ is of type bi-$\F_n$.  The idea is that if $X$ is a bi-equivariant classifying space for $M^0$, then the union $Y$ of components of $X$ corresponding to elements of $M$ is a bi-equivariant classifying space for $M$ and $Y_n$ will be $M\times M^{op}$-finite if $X_n$ is $M^0\times (M^0)^{op}$-finite. More generally, if $T$ is a submonoid of $M$ such that $M \setminus T$ is an ideal, then $M$ being of type bi-$\Fn$ implies $T$ is also of type bi-$\Fn$.
\end{Rmk}

Next we show that finitely presented monoids are of type bi-$\F_2$.  The proof is similar to the proof of Theorem~\ref{t:f2}, which is in fact a consequence.

\begin{Thm}\label{t:bi.f2}
Let $M$ be a finitely presented monoid.  Then $M$ is of type bi-$\F_2$.
\end{Thm}
\begin{proof}
Suppose that $M$ is generated by a finite set $A$ with defining relations $u_1=v_1,\ldots, u_n=v_n$.  We construct an $M\times M^{op}$-finite $2$-dimensional free $M\times M^{op}$-CW complex $X$ with $1$-skeleton the two-sided Cayley graph $\overleftrightarrow{\Gamma(M,A)}$ by attaching an $M\times M^{op}$-cell $M\times M\times B^2$ for each relation.  Suppose that $u_i$ and $v_i$ map to $m_i$ in $M$.  Let $p_i,q_i$ be the paths from $(1,m_i)$ to $(m_i,1)$ labelled by $u_i$ and $v_i$, respectively, cf.~\eqref{eq:two.sided.path}.  Then we glue in a disk $d_i$ with boundary path $p_iq_i^{-1}$ and glue in $M\times M\times B^2$ using Proposition~\ref{p:cell-maps} (so $\{(m_L,m_R)\}\times B^2$ is sent to $m_Ld_im_R$).  Then $X$ is a free $M\times M^{op}$-CW complex of dimension at most $2$ that is $M\times M^{op}$-finite and $\pi_0(X)\cong M$.  By Proposition~\ref{p:reform.fn}, it suffices to prove that each connected component of $X$ is simply connected.

The connected component $X(m)$ of $X$ corresponding to $m\in M$ is a digraph rooted at $(1,m)$ by \eqref{eq:two.sided.path}. Let $T_m$ be a directed spanning tree for $X(m)$ rooted at $(1,m)$.  Let $e=(n_1,an_2)\xrightarrow{\,\,a\,\,} (n_1a,n_2)$ be a directed edge of $X(m)$ not belonging to $T_m$.  Then the corresponding generator of $\pi_1(X(m),(1,m))$ is of the form $peq^{-1}$ where $p$ and $q$ are directed paths from $(1,m)$ to $(n_1,an_2)$ and $(n_1a,n_2)$, respectively.  Let $u$ be the label of $p$ and $v$ be the label of $q$.  Then $ua=v$ in $M$.  Thus it suffices to prove that if $x,y\in A^*$ are words which are equal in $M$ to $m'$ labelling respective paths from $(1,m)$ to $(m',m'')$ with $m'm''=m$, then the corresponding loop $\ell$ labelled $xy^{-1}$ at $(1,m)$ is null homotopic.

 By induction on the length of a derivation from $x$ to $y$, we may assume that $x=wu_iw'$ and $y=wv_iw'$ for some $i=1,\ldots, n$.  Then the  path labelled by $w$ starting at $(1,m)$ ends at $(w,m_iw'm'')$ where we recall that $m_i$ is the image of $u_i,v_i$ in $M$.  Then $wd_iw'm''$ is a $2$-cell bounded by parallel paths from $(w,m_iw'm'')$ to $(wm_i,w'm'')$ labeled by $u_i$ and $v_i$, respectively.  It follows that the paths labelled by $x$ and $y$ from $(1,m)$ to $(m',m'')$ are homotopic relative to endpoints and hence  $\ell$ is null homotopic.  This completes the proof that $X(m)$ is simply connected.
\end{proof}

\begin{remark}
We currently do not know the precise relationship between bi-$\F_2$ and finitely presentability for monoids.
Specifically we have the question:  Is there a finitely generated bi-$\F_2$ monoid that is not finitely presented? Even for inverse monoids this question remains open.
\end{remark}

We next observe that finitely generated free monoids are bi-$\F_{\infty}$.

\begin{Prop}\label{p:free}
Let $A$ be a finite set.  Then the free monoid $A^*$ is of type bi-$\F_{\infty}$.
\end{Prop}
\begin{proof}
Each connected component of $\overleftrightarrow{\Gamma(M,A)}$ is a tree and hence contractible.  Thus $\overleftrightarrow{\Gamma(M,A)}$ is an $A^*\times (A^*)^{op}$-finite bi-equivariant classifying space for $A^*$.
\end{proof}

Theorem~\ref{t:hom=top} has an analogue for bi-$\F_n$ and bi-$\FP_n$ with essentially the same proof, which we omit.

\begin{Thm}\label{t:hom=top.bi}
Let $M$ be a monoid of type bi-$\F_2$.  Then $M$ is of type bi-$\F_n$ if and only if $M$ is of type bi-$\FP_n$ for $0\leq n\leq \infty$.
\end{Thm}

Observe that Theorem~\ref{t:hom=top.bi} implies that $M$ is of type bi-$\F_{\infty}$ if and only if $M$ is of type bi-$\F_n$ for all $n\geq 0$.  The analogue of Proposition~\ref{p:proj=free} in our setting again admits a very similar proof that we omit.

\begin{Prop}\label{t:proj=free.bi}
If $M$ is of type bi-$\F_n$ with $n\geq 1$, then $M$ has a free $M\times M^{op}$-CW complex $X$ that is a bi-equivariant classifying space for $M$ where $X_n$  is $M\times M^{op}$-finite.
\end{Prop}

Proposition~\ref{p:finiteness.dp} also has a two-sided analogue.

\begin{Prop}
If $M,N$ are of type bi-$\F_n$, then $M\times N$ is of type bi-$\F_n$.
\end{Prop}

Let us turn to some inheritance properties for bi-$\F_n$. If $I$ is an ideal of $M$ containing an identity $e$, then $e$ is a central idempotent and $MeM=Me=eM=eMe$.  Indeed, $em=(em)e=e(me)=me$ as $em,me\in I$.
If $f,f'\in E(M)$, then $fe,f'e\in E(eMe)$ and $e(Mf\times f'M)e =eMefe\times f'eMe$ as an $eMe$-$eMe$-biset and hence is finitely generated projective. Thus if $P$ is a (finitely generated) projective $M\times M^{op}$-set, then $ePe$ is a (finitely generated) projective $eMe\times eMe^{op}$-set.

\begin{Prop}\label{p:central.idem}
Let $M$ be a monoid and $0\leq n\leq \infty$ and $e\in E(M)$ be a central idempotent.  If $M$ is of type bi-$\F_n$, then so is $eMe$.
\end{Prop}
\begin{proof}
Let $X$ be a bi-equivariant classifying space of $M$ such that $X_n$ is $M\times M^{op}$-finite.  Suppose that $X$ is obtained via pushouts as per \eqref{eq:pushout} (but with $M\times M^{op}$ in place of $M$).  Then each $eP_ke$ is a projective $eMe\times eMe^{op}$-set and is finitely generated whenever $P_k$ was finitely generated by the observation preceding the proposition.  Thus $eXe$ is a projective $eMe\times eMe^{op}$-CW complex and $(eXe)_n=eX_ne$ is $eMe\times eMe^{op}$-finite.  Also, since $eXe\cong (eM\times Me)\otimes_{M\times M^{op}} X$, we deduce that $\pi_0(eXe)\cong e\pi_0(X)e\cong eMe$ by Proposition~\ref{p:components}.  If $X(m)$ is the component of $X$ corresponding to $m\in eMe$, then $eX(m)e$ is the component of $eXe$ corresponding to $m$ in $eXe$ and is a retract of $X(m)$.  But $X(m)$ is contractible and hence $eX(m)e$ is contractible.  This shows that $eXe$ is a bi-equivariant classifying space for $eMe$, completing the proof.
\end{proof}

Two monoids $M$ and $N$ are \emph{Morita equivalent} if the categories of left $M$-sets and left $N$-sets are equivalent.  It is known that this is the case if and only if there is an idempotent $e\in E(M)$ such that $xy=1$ for some $x,y\in M$ with $ey=y$ and $eMe\cong N$~\cite{Knauer1971}.  It follows easily that if $M$ and $N$ are Morita equivalent, then so are $M^{op}$ and $N^{op}$. Note that if $e$ is as above, then the functor $A\mapsto eA\cong eM\otimes_M A$ from $M$-sets to $N$-sets (identifying $N$ with $eMe$) is an equivalence of categories with inverse $B\mapsto Me\otimes_N B$.  This uses that $Me\otimes_{eMe}eM\cong M$ as $M\times M^{op}$-sets via the multiplication map (the inverse bijection takes $m\in M$ to $mxe\otimes y$)  and $eM\otimes_M Me\cong eMe$  as $eMe\times (eMe)^{op}$-sets (via the multiplication with inverse $eme\mapsto em\otimes e$).  It follows that if $P$ is a (finitely generated) projective $M$-set, then $eP$ is a (finitely generated) projective $N$-set (as being projective is categorical and a projective is finitely generated if and only if it is a coproduct of finitely many indecomposable projectives, which is also categorical). In particular, $eM$ is a finitely generated projective $N$-set.

\begin{Prop}\label{p:Morita.equiv}
Let $M$ and $N$ be Morita equivalent monoids and $0\leq n\leq \infty$.
\begin{enumerate}
  \item $M$ is of type left-$\F_n$ if and only if $N$ is of type left-$\F_n$.
  \item $M$ is of type right-$\F_n$ if and only if $N$ is of type right-$\F_n$.
  \item $M$ is of type bi-$\F_n$ if and only if $N$ is of type bi-$\F_n$.
\end{enumerate}
\end{Prop}
\begin{proof}
By symmetry it suffice to prove the implications from left to right.   We may assume without loss of generality that $N=eMe$ where $1=xy$ with $ey=y$. Notice that $1=xy=xey$ and so replacing $x$ by $xe$, we may assume that $xe=x$.
To prove (1), suppose that $X$ is an equivariant classifying space for $M$ such that $X_n$ is $M$-finite. Then $eM\otimes_M X\cong eX$ is a projective $N$-CW complex by Proposition~\ref{c:base.change.cw} such that $(eX)_n=eX_n$ is $N$-finite.  But $eX$  is a retract of $X$ and hence contractible.  We deduce that $N$ is of type left-$\F_n$.

The proof of (2) is dual.  To prove (3), observe that $(e,e)(M\times M^{op})(e,e) = eMe\times (eMe)^{op}$ and that we have $(x,y)(y,x)=(1,1)$ and $(e,e)(y,x)=(y,x)$ in $M\times M^{op}$ because $xy=e$, $ey=y$ and $xe=x$.  Thus $M\times M^{op}$ is Morita equivalent to $N\times N^{op}$ and $eM\times Me$ is a finitely generated projective $N\times N^{op}$-set.  Suppose that $X$ is a bi-equivariant classifying space for $M$ such that $X_n$ is $M\times M^{op}$-finite.  Then $(eM\times Me)\otimes_{M\times M^{op}} X\cong eXe$ is a projective $N\times N^{op}$-CW complex such that $(eXe)_n=eX_ne$ is $N\times N^{op}$-finite by Corollary~\ref{c:base.change.cw}.  Also, $\pi_0(eXe)\cong e\pi(X)e\cong N$ by Proposition~\ref{p:components}.  Moreover, if $m\in eMe$ and $X(m)$ is the component of $X$ corresponding to $m$, then the component of $eXe$ corresponding to $m$ is $eX(m)e$, which is a retract of $X(m)$ and hence contractible.  Thus $eXe$ is a bi-equivariant classifying space of $N$.  The result follows.
\end{proof}

There are examples of Morita equivalent monoids that are not
isomorphic; see \cite{Knauer1971}.

We define the \emph{geometric dimension} of $M$ to be the minimum dimension of a bi-equivariant classifying space for $M$. The \emph{Hochschild cohomological dimension} of $M$, which we write $\dim M$, is the length of a shortest projective resolution of $\mathbb ZM$ as a $\mathbb Z[M\times M^{op}]$-module.  Of course, the Hochschild cohomological dimension bounds both the left and right cohomological dimension and the geometric dimension bounds the Hochschild cohomological dimension. Also the geometric dimension bounds both the left and right geometric dimensions because if $X$ is a bi-equivariant classifying space for $M$ of dimension $n$, then $X/M$ is a classifying space of dimension $n$.

The following theorem has an essentially identical proof to Theorem~\ref{t:geom.dim}.

\begin{Thm}\label{t:geom.dim.bi}
Let $M$ be a monoid.  Then $M$ has a bi-equivariant classifying space of dimension $\max\{\dim M,3\}$.
\end{Thm}

Free monoids have a forest for a bi-equivariant classifying space and
hence have geometric dimension $1$.  It is well known (see
e.g. \cite{Mitchell1972}) that they have Hochschild cohomological dimension $1$.

It is known that a monoid has Hochschild cohomological dimension $0$
if and only if it is a finite regular aperiodic monoid with sandwich
matrices invertible over $\mathbb Z$ (see \cite{Cheng1984}).  For
instance, any finite aperiodic inverse monoid has Hochschild
cohomogical dimension $0$.  A non-trivial monoid of Hochschild
cohomological dimension $0$ does not have geometric dimension $0$
because $M$ would have to be a projective $M$-biset.  So
$M\cong Me\times fM$, with $e,f\in E(M)$, via an equivariant map $\p$
sending $(e,f)$ to $1$ (as $M$ being finite aperiodic implies that $1$
is the unique generator of $M$ as a two-sided ideal).  But then
$f=\p(e,f)f=\p(e,f)=1$ and similarly $e=1$ and so $M$ is trivial.
Thus non-trivial monoids of Hochschild cohomological dimension $0$ do
not have geometric dimension $0$.
\section{Brown's theory of collapsing schemes}\label{sec_collapse}

The theory of collapsing schemes was introduced by Brown in
\cite{Brown1989}. Since then it has become an important and often-used
tool for proving that certain groups are of type $\Finfty$.  The first
place the idea appears is in a paper of Brown and Geoghegan
\cite{BrownGeoghegan1984} where they had a cell complex with one
vertex and infinitely many cells in each positive dimension, and they
showed how it could be collapsed to a quotient complex with only two
cells in each positive dimension. Brown went on to develop this idea
further in \cite{Brown1989} formalising it in his theory of collapsing
schemes, and applying it to give a topological proof that groups which
admit presentations by finite complete rewriting systems are of type
$\Finfty$ (see Section~\ref{sec_admitting} below for the definition of
complete rewriting system). Brown's theory of collapsing schemes was
later rediscovered under the name of discrete Morse theory
\cite{Forman1995, Forman2002}, an important area in algebraic
combinatorialists.  Chari \cite{Chari2000} formulated discrete Morse
theory combinatorially via Morse matchings, which turn out to be the
same thing as collapsing schemes.

The basic idea of collapsing schemes for groups is a follows. Suppose we are given a
finitely presented group $G$ and we would like to prove it is of type
$\Finfty$. Then we can first begin with the big $K(G,1)$ complex
$|BG|$ with infinitely many $n$-cells for each $n$. Then in certain
situations it is possible to show how one can collapse away all but
finitely many cells in each dimension resulting in a $K(G,1)$ much
smaller than the one we started with.  The collapse is carried out
using a so-called \emph{collapsing scheme} associated with the
simplicial set $BG$. It turns out that any group which is presentable
by a finite complete rewriting system admits a collapsing scheme that,
using this process, can be used to prove the group is of type
$\Finfty$; see \cite[page 147]{Brown1989}.

As mentioned in the introduction above, Brown in fact develops
this theory for monoids in general, and applies the theory of
collapsing schemes to show that if $M$ admits a presentation by a
finite complete rewriting system then its classifying space $|BM|$ has
the homotopy type of a CW complex with only finitely many cells in
each dimension. However, as discussed in detail in the introduction to
this article, this information about the space $|BM|$ is not enough on
its own to imply that the monoid $M$ is of type left-$\FPinfty$.

Motivated by this, in this section we shall develop the theory of
$M$-equivariant collapsing schemes. We shall prove that if an
$M$-simplicial set admits an $M$-equivariant collapsing scheme of
finite type then the monoid is of type left-$\Finfty$. We then prove
that if $M$ admits a finite complete rewriting system then $\EMr$
admits an $M$-equivariant collapsing scheme of finite type, thus
giving a topological proof that such monoids are of type
left-$\Finfty$. To do this, we shall identify conditions under which a
collapsing scheme for $BM$ can be lifted to give an $M$-equivariant
collapsing scheme for $\EMr$. These conditions will hold in a number
of different situations, including when $M$ admits a presentation by a
finite complete rewriting system and when $M$ is a, so-called,
factorable monoid \cite{HessArxiv}. We also develop the two-sided
theory. As a consequence we also obtain a topological proof of the
fact that such a monoid is of type bi-$\Finfty$, recovering a theorem
of Kobayashi \cite{Kobayashi2005}.

\subsection{Collapsing schemes}\label{sec_collapse2}

Let $K = \bigcup_{i \geq 0}K_i$ be a simplicial set and let $X = |K|$
be its geometric realisation. We identify the cells of $X$ with the
non-degenerate simplices of $K$. A collapsing scheme for $K$ consists
of the following data:

\begin{itemize}
\item A partition of the cells of $X$ into three classes, $\E$, $\C$, $\R$, called the \emph{essential}, \emph{collapsible} and \emph{redundant} cells, respectively, where the collapsible cells all have dimension at least one.
\item Mappings $c$ and $i$ which associate with each redundant $n$-cell $\tau$ a collapsible $(n+1)$-cell $c(\tau)$, and a number $i(\tau)$, such that $ \tau = d_{i(\tau)} (c(\tau)). $
\end{itemize}
Let $\sigma = c(\tau)$.
If $\tau'$ is a redundant $n$-cell such that $\tau' = d_j \sigma$ for some $j \neq i(\tau)$ then we call $\tau'$ an immediate predecessor of $\tau$ and write $\tau' \prec \tau$.
Furthermore, the conditions for a collapsing scheme are satisfied, which means:
\begin{enumerate}
\item[(C1)]  for all $n$, the mapping $c$ defines a bijection between $\R_n$ (the redundant $n$-cells) and $\C_{n+1}$ (the collapsible ($n+1$)-cells).
\item[(C2)] there is no infinite descending chain $\tau \succ \tau' \succ \tau'' \succ \cdots $ of redundant $n$-cells.
\end{enumerate}
Condition (C2) clearly implies that there is a unique integer $i$ such that $\tau = d_i (c(\tau))$ (otherwise we would have $\tau \succ \tau$, leading to an infinite descending chain). It also follows from (C2) that, by K\"{o}nigs lemma, there cannot be arbitrarily long descending chains $\tau_0 \succ \cdots \succ \tau_k$.
This is a key fact in the proof of \cite[Proposition~1]{Brown1989} since it gives rise to the notion of `height':
\begin{definition}[Height]\label{def_height}
The \emph{height} of a redundant cell $\tau$,
written $\mathrm{height}(\tau)$,
 is the maximum length of a descending chain $\tau = \tau_0 \succ \tau_1 \succ \cdots \succ \tau_k$.
\end{definition}

We say that a collapsing scheme is of \emph{finite type} if it has finitely many essential cells of each dimension.

In the construction of the `small' CW complex in the proof of
\cite[Proposition~1]{Brown1989} the redundant $n$-cells are adjoined
in order of their heights, guaranteeing in the proof that the
adjunction of $\tau$ and $c(\tau)$ is an elementary expansion. This is
the the key idea in the proof which is that each pair
$(\tau, c(\tau))$ of redundant and corresponding collapsible cells may
be adjoined without changing the homotopy type, and so in the end it
is only the essential cells that matter. More precisely, Brown proves
that if $K$ be a simplical set with a collapsing scheme, then its
geometric realisation $X = |K|$ admits a canonical quotient CW
complex $Y$, whose cells are in 1--1 correspondence with the essential
cells of $X$.
This notion of height in Brown's theory relates to the values taken by
the discrete Morse function in Forman's theory (see \cite[page
10]{Forman2002}). A discrete Morse function gives one a way to build
the simplicial complex by attaching the simplices in the order
prescribed by the function, i.e., adding first the simplices which are
assigned the smallest values. Brown's essential cells are called
`critical' in Forman's theory.

\section{$M$-equivariant collapsing schemes}
\label{sec_MeqCS}

In this section we develop the theory of $M$-equivariant collapsing schemes, or equivalently, of $M$-equivariant discrete Morse theory. Results on $G$-equivariant discrete Morse theory, for $G$  a group, may be found in \cite{Freij}.

Let $K = \bigcup_{i \geq 0} K_i$ be a simplicial set with degeneracy and face operators $d_i$, $s_i$, and equipped with a collapsing scheme $(E,R,C,c,i)$. Here $E$, $R$ and $C$ partition the cells (which are in bijective correspondence with the non-degenerate simplices) of $K$.

Let $M$ be a monoid acting on the simplicial set $K$ with the following conditions satisfied:
\begin{enumerate}
\item[(A1)] The action of $M$ maps $n$-simplicies to $n$-simplicies, and
commutes with $d_i$ and $s_i$, that is, $M$ is acting by simplicial morphisms.
\item[(A2)] For every $n$-simplex $\sigma$ and $m \in M$, $\sigma$ is a cell (i.e. is a non-degenerate simplex) if and only if $m\sigma$ is a cell, in which case $\sigma \in E$ (respectively $R$, $C$) if and only if $m\sigma \in E$ (respectively $R$, $C$).
\item[(A3)] If $(\sigma,\tau) \in R_n \times C_{n+1}$ is a matched redundant-collapsible pair (i.e. $\tau = c(\sigma)$) then so is the pair $m(\sigma,\tau) = (m\sigma,m\tau) \in R_n \times C_{n+1}$, i.e., $c(m\sigma)=mc(\sigma)$ for $\sigma\in R_n$.
\item[(A4)] There is a subset $B \subseteq E \cup R \cup C$ such that for all $n$ the set of $n$-cells is a free left $M$-set with basis $B_n$ (where $B_n$ is the set of $n$-cells in $B$).
Let $E^B = E \cap B$, $R^B = R \cap B$ and $C^B = C \cap B$.
Then $E_n$ is a free left $M$-set with basis $E_n^B$, and similarly for $R_n$ and $C_n$.
\item[(A5)] For every matched pair $(\sigma,\tau) \in R \times C$, $\sigma \in R^B$ if and only if $\tau \in C^B$. In particular, for every matched pair $(\sigma,\tau)$ there is a unique pair $(\sigma',\tau') \in R^B \times C^B$ and $m \in M$ such that $(\sigma, \tau) = m(\sigma', \tau')$; namely, if $\sigma=m\sigma'$ with $\sigma'\in R^B$ and $\tau'=c(\sigma')$, then $m\tau'=mc(\sigma')=c(m\sigma')=c(\sigma)=\tau$.
\item[(A6)] For every redundant cell $\tau$ and every $m\in M$
\[
\mathrm{height}(\tau) = \mathrm{height}(m\tau),
\]
with height defined as in Definition~\ref{def_height} above.
\end{enumerate}

These conditions imply that $K$ is a rigid free left $M$-simplicial set and hence by Lemma~\ref{lem_induced}
the action of $M$ on $K$ induces an action of $M$ on the geometric realisation $|K|$ by continuous maps making $|K|$ into a free left $M$-CW complex. When the above axioms hold, we call $(E,R,C,c,i)$ an \emph{$M$-equivariant collapsing scheme} for the rigid free left $M$-simplicial set $K$. Dually, given a rigid free right $M$-simplicial set $K$ with a collapsing scheme satisfying the above axioms for $K$ as an $M^{op}$-simplicial set we call  $(E,R,C,c,i)$ an $M$-equivariant collapsing scheme for $K$. If $K$ is a bi-$M$-simplicial set we say $(E,R,C,c,i)$ an $M$-equivariant collapsing scheme if the axioms are satisfied for $K$ as a left $M \times M^{op}$-simplicial set.

Our aim is to prove a result about the $M$-homotopy type of $|K|$ when $K$ has an $M$-equivariant collapsing scheme. Before doing this we first make some observations about mapping cylinders and the notion of elementary collapse.

\subsection{Mapping cylinders and elementary collapse}

If  $X$ is a subspace of a space $Y$ then $D\colon Y \rightarrow X$ is a \emph{strong deformation retraction} if there is a map $F\colon  Y \times I \rightarrow Y$ such that, with $\F_t\colon  Y \rightarrow Y$ defined by $\F_t(y) = F(y,t)$, we have
(i) $\F_0 = 1_Y$,
(ii) $\F_t(x) = x$ for all $(x,t) \in X \times I$, and
(iii) $\F_1(y) = D(y)$ for all $y \in Y$.
If $D\colon X \rightarrow Y$ is a strong deformation retraction then $D$ is a homotopy equivalence, a homotopy inverse of which is the inclusion $i\colon  X \hookrightarrow Y$.
%
%
%
%
%
\begin{definition}[Mapping cylinder]\label{def_MC}
Let $f\colon X \rightarrow Y$ be a cellular map between CW complexes. The \emph{mapping cylinder} $M_f$ is defined to be the adjunction complex $Y \coprod_{f_0} (X \times I)$ where $f_0\colon  X \times \{ 0 \}$ is the map $(x,0) \mapsto f(x)$. Let $i_1\colon  X \rightarrow X \times I, x \mapsto (x,1)$ and let $i_0\colon  X \rightarrow X \times I, x \mapsto (x,0)$.
Let $p$ be the projection $p\colon  X \times I \rightarrow X, (x,i) \mapsto
x$. Also set $i = k \circ i_1 $, with $k$ as below. Thus we have
\[
\begin{tikzcd}
X \ar{r}{f}\ar{d}{i_1}\ar[transform canvas={xshift=2ex}]{d}{i_0}  								& Y \ar{d}{j} \\
X \times I \ar{r}[swap]{k} 	\ar[transform canvas={xshift=-2ex}]{u}{p} 						& M_f
\end{tikzcd}
\]
The map from $X \times I$ to $Y$ taking $(x,t)$ to $f(x)$, and the identity map on $Y$, together induce a retraction $r\colon  M_f \rightarrow Y$.
\end{definition}

The next proposition is~\cite[Proposition~4.1.2]{GeogheganBook}.

\begin{proposition}\label{prop_Geoghegan412}
The map $r$ is a homotopy inverse for $j$, so $r$ is a homotopy equivalence. Indeed there is a strong deformation retraction $D\colon M_f \times I \rightarrow M_f$ of $M_f$ onto $Y$ such that $D_1 = r$.
\end{proposition}

The following result is the $M$-equivariant analogue of \cite[Proposition 4.1.2]{GeogheganBook}. Recall that if $X$ is a projective (resp. free) $M$-CW complex then $Y=M\times I$ is a projective (resp. free) $M$-CW complex, where $I$ is given the trivial action.

\begin{Lemma}\label{lem:MequiMappingCylinder}
Let $f\colon  X \rightarrow Y$ be a continuous $M$-equivariant cellular map between free (projective) $M$-CW complexes $X$ and $Y$. Let $M_f$ be the pushout of
\[
\begin{tikzcd}
X \ar{r}{f}\ar{d}[swap]{i_0} 								& Y \ar{d}{j} \\
X \times I \ar{r}[swap]{k} 							& M_f
\end{tikzcd}
\]
where $i_0\colon  X \rightarrow X \times I, x \mapsto (x,0)$. Then:
\begin{enumerate}
\item[(i)] $M_f$ is a free (projective) $M$-CW complex; and
\item[(ii)] there is an $M$-equivariant strong deformation retraction $r\colon  M_f \rightarrow Y$. In particular $M_f$ and $Y$ are $M$-homotopy equivalent.
\end{enumerate}
\end{Lemma}
\begin{proof}
It follows from Lemma~\ref{l:adjunction_free_pushout} that $M_f$ has the structure of a free (projective) $M$-CW complex, proving part (i). For part (ii) first note that the map from $X \times I$ to $Y$ taking $(x,t)$ to $f(x)$, and the identity map on $Y$, together induce a retraction $r\colon  M_f \rightarrow Y$. It follows from Proposition~\ref{prop_Geoghegan412} that $r$ is a homotopy equivalence with homotopy inverse $j$. By Corollary~\ref{c:whitehead} to show that $r$ is an $M$-homotopy equivalence it suffices to verify that $r$ is an $M$-equivariant map between the sets $M_f$ and $Y$.
But $M$-equivariance of $r$ follows from the definitions of $r$ and $M_f$, the definition of the action of $M$ on $M_f$ which in turn is determined by the actions of $M$ on $X \times I$ and on $Y$, together with the assumption that $f\colon  X \rightarrow Y$ is $M$-equivariant.
\end{proof}
 The fundamental idea of collapsing schemes, and
 discrete Morse theory, is that of a collapse.
 The following definition may be found in \cite[page~14]{CohenBook} and \cite[Section~2]{Farley}. We use the same notation as in \cite{CohenBook}. In particular $\approx$ denotes homeomorphism of spaces.

\begin{definition}[Elementary collapse]
If $(K,L)$ is a CW pair then $K$ collapses to $L$ by an elementary collapse, denoted $K \searrow^e L$, if and only if:
\begin{enumerate}
\item $K = L \cup e^{n-1} \cup e^n$ where $e^n$ and $e^{n-1}$ are open cells of dimension $n$ and $n-1$ respectively, which are not in $L$, and
\item there exists a ball pair $(B^n, B^{n-1}) \approx (I^n, I^{n-1})$ and a map $\varphi\colon B^n \rightarrow K$ such that
	\begin{enumerate}
	\item[(a)] $\varphi$ is a characteristic map for $e^n$
	\item[(b)] $\varphi|_{B^{n-1}}$ is a characteristic map for $e^{n-1}$
	\item[(c)] $\varphi(P^{n-1}) \subseteq L^{n-1}$ where $P^{n-1} = cl(\partial B^n - B^{n-1})$.
	\end{enumerate}
\end{enumerate}
\end{definition}
Note that in this statement $P^{n-1}$ is an $(n-1)$-ball (i.e. is homeomorphic to $I^{n-1}$). We say that $K$ collapses to $L$, writing $K \searrow L$, if $L$ may be obtained from $K$ by a sequence of elementary collapses. We also say that $K$ is an \emph{elementary expansion} of $L$. An elementary collapse gives a  way of modifying a CW complex $K$, by removing the pair $\{ e^{n-1}, e^n \}$, without changing the homotopy type of the space.  We can write down a homotopy which describes such an elementary collapse $K \searrow^e L$ as follows.
Let $(K,L)$ be a CW pair such that $K \searrow^e L$. Set $\varphi_0 = \varphi|_{p^{n-1}}$ in the above definition. Then
\[
\varphi_0\colon  (P^{n-1},\partial P^{n-1}) \rightarrow (L^{n-1},L^{n-2}),
\]
(using the identification $P^{n-1} \approx I^{n-1}$)
and
\[
(K,L) \approx (L \coprod_{\varphi_0} B^n, L).
\]

The following is~\cite[(4.1)]{CohenBook}.
\begin{Lemma}
If $K \searrow^e L$ then there is a cellular strong deformation retraction $D\colon K \rightarrow L$.
\end{Lemma}
Indeed, let $K = L \cup e^{n-1} \cup e^n$. There is a map $\varphi_0\colon  I^{n-1} \approx P^{n-1} \rightarrow L^{n-1}$ such that
\[
(K,L) \approx (L \coprod_{\varphi_0} B^n, L).
\]
But $L \coprod_{\varphi_0} B^n$ is the mapping cylinder $M_{\varphi_0}$ of $\varphi_0\colon  I^{n-1} \rightarrow  L^{n-1}$. Thus by Proposition ~\ref{prop_Geoghegan412} there is a strong deformation retraction
\[
D\colon  K \approx L \coprod_{\varphi_0} I^n \rightarrow L
\]
such that $D(\overline{e^n}) = \varphi_0(I^{n-1}) \subset L^{n-1}$.
The map $D$ is given by the map $r$ in Definition~\ref{def_MC}.

We may now state and prove the main result of this section.

\begin{theorem}\label{main_theorem}
Let $K$ be a rigid free left $M$-simplicial set with $M$-equivariant collapsing scheme $(E,R,C,c,i)$ (that is, the conditions
(A1)--(A6) are satisfied). Then, with the above notation, there is a free left $M$-CW complex $Y$
with $Y \simeq_M |K|$ and such that the cells of $Y$ are in bijective correspondence with $E$,
and under this bijective correspondence
$Y_n$ is a free left $M$-set with basis $E_n^B$ for all $n$.
\end{theorem}
\begin{proof}
Let $X$ be the geometric realisation $|K|$ of the simplicial set $K$. By axiom (A1) we have that $K$ is a left $M$-simplical set, and it follows that $X = |K|$ has the structure of a left $M$-CW complex where the $M$-action is given by Lemma~\ref{lem_induced}. In fact, by assumptions (A2)-(A6), $X$ is a free $M$-CW complex where, for each $n$, the set $E_n^B \cup R_n^B \cup C_n^B$ is a basis for the $n$-cells.

Write $X$ as an increasing sequence of subcomplexes
\[
X_0 \subseteq X_0^+ \subseteq X_1 \subseteq X_1^+ \subseteq \ldots
\]
where, $X_0$ consists of the essential vertices,
$X_n^+$ is obtained from $X_n$ by adjoining the redundant $n$-cells and collapsible $(n+1)$-cells, and
$X_{n+1}$ is obtained from $X_n^+$ by adjoining the essential $(n+1)$-cells.
We write $X_n^+$ as a countable union
\[
X_n = X_n^0 \subseteq X_n^1 \subseteq X_n^2 \subseteq \ldots
\]
with $X_n^+ = \coprod_{i \geq 0}X_n^i$ where $X_n^{j+1}$ is constructed from $X_n^j$ by adjoining $(\tau,c(\tau))$ for every redundant $n$-cell $\tau$ of height $j$. From assumptions (A1)-(A6), for every $n$ and $j$, each of $X_n^+$, $X_n$ and $X_n^j$ is a free $M$-CW subcomplex of $X$.

As argued in the proof of \cite[Proposition~1]{Brown1989}, for every redundant $n$-cell $\tau$ of height $j$ the adjunction of $(\tau, c(\tau))$ is an elementary expansion. In this way $X_n^{j+1}$ can be obtained from $X_n^{j}$ by a countable sequence of simultaneous elementary expansions. The same idea, together with Lemma~\ref{lem:MequiMappingCylinder}, can be used to obtain an $M$-homotopy equivalence between $X_n^{j}$ and $X_n^{j+1}$. The details are as follows.

Recall that $X_n^{j+1}$ is obtained from $X_n^j$ by adjoining $(\tau,c(\tau))$ for every redundant $n$-cell $\tau$ of height $j$. It follows from the axioms (A1)-(A6) that this set of pairs $(\tau,c(\tau))$ is a free $M$-set with basis $\{ (\tau, c(\tau)) \in R_n^B \times C_{n+1}^B: \mathrm{height}(\tau)=j \}$. Let $(\tau, c(\tau)) \in R_n^B \times C_{n+1}^B$ with $\mathrm{height}(\tau)=j$, and let $m \in M$. From the assumptions (A1)-(A6) it follows that
\[
m \cdot (\tau,c(\tau)) = (m\tau,c(m\tau)),
\quad
\mbox{and}
\quad
\mathrm{height}(m\tau) = \mathrm{height}(\tau) = j.
\]
The pair
\[
(X_n^{j+1}, X_n^{j+1} - \{ m\tau,c(m\tau) \})
\]
satisfies the conditions of an elementary collapse. Indeed (as argued in the proof of \cite[Proposition~1]{Brown1989}) every face of $c(m\tau)$ other than $m\tau$ is either (i) a redundant cell of height $<j$, (ii) is essential (so has height $0$), or (iii) is collapsible or degenerate. It follows that the adjunction of $m\tau$ and $c(m\tau)$ is an elementary expansion. This is true for every pair $(m \tau, c(m \tau))$ where $(\tau, c(\tau)) \in R_n^B \times C_{n+1}^B$ with $\mathrm{height}(\tau)=j$ and $m \in M$. Now $X_n^{j+1}$ is obtained from $X_n^j$ by adjoining all such pairs $(m \tau, c(m \tau))$.  Let $A = \{ (\tau, c(\tau)) \in R_n^B \times C_{n+1}^B: \mathrm{height}(\tau)=j\}$ and let $M \times A$ denote the free left $M$-set with basis $\{ (1,a): a \in A \}$ and action $m(n,a) = (mn,a)$. Then $(M \times A) \times I^{n+1}$ is a disjoint union of the free $M$-cells $(M \times \{a\}) \times I^{n+1}
$ with $a \in A$.
The characteristic maps for the collapsible $n+1$ cells of height $j$ combine to give an
$M$-equivariant map $\varphi\colon  (M \times A) \times I^{n+1} \rightarrow X_n^{j+1}$ such that
\begin{enumerate}
\item[(E1)] $\varphi$ restricted to the $(m,a) \times I^{n+1}$, $(m \in M, a \in A)$, gives characteristic maps for each of the  collapsible $n+1$ cells of height $j$;
\item[(E2)] $\varphi$ restricted to the $(m,a) \times I^{n}$, $(m \in M, a \in A)$, gives characteristic maps for each $\tau \in R_n$ such that $c(\tau)$ is a collapsible $n+1$ cell of height $j$;
\item[(E3)]  $\varphi((M \times A) \times P^n) \subseteq (X_n^{j+1})^{\leq n}$ (where $(X_n^{j+1})^{\leq n}$ is the subcomplex of $X_n^{j+1}$ of cells of dimension $\leq n$) where
$P^n = cl(\partial I^{n+1} - I^n)$.
\end{enumerate}
Set $\varphi_0 = \varphi|_{(M \times A) \times P^n}$. Then
\[
\varphi_0\colon  (M \times A) \times P^n \rightarrow (X_n^{j+1})^{\leq n}
\]
is a continuous $M$-equivariant cellular map between free $M$-CW complexes $(M \times A) \times P^n$ and $(X_n^{j+1})^{\leq n}$. It follows that $X_n^{j+1}$ is $M$-equivariantly isomorphic to
\begin{equation}\label{eq_MC}
(X_n^{j+1} - \{ (\tau, c(\tau)) \in R_n \times C_{n+1}\colon\mathrm{height}(\tau)=j\})
\coprod_{\varphi_0}
((M \times A) \times I^{n+1}).
\end{equation}
But since $P^n = cl(\partial I^{n+1} - I^n)  \approx I^n$ we conclude that \eqref{eq_MC} is just the mapping cylinder of $\varphi_0$. Thus we can apply Lemma~\ref{lem:MequiMappingCylinder} to obtain a strong deformation retraction $r_j\colon  X_n^{j+1} \rightarrow  X_n^j$ which is also an $M$-homotopy equivalence.  It follows that there is a retraction $r_n\colon  X_n^+ \rightarrow X_n$ which is an $M$-equivariant homotopy equivalence.

We build the space $Y$ such that $|K| \simeq_M Y$ inductively. First set $Y_0 = X_0$. Now suppose that we have an $M$-homotopy equivalence $\pi^{n-1}\colon  X_{n-1}^+ \rightarrow Y_{n-1}$ is given.  Define $Y_n$ to be the the $M$-CW complex $Y_{n-1} \cup (M \times E_n^B)$ where $(M \times E_n^B)$ is a collection of free $M$-cells indexed by $E_n^B$. These free $M$-cells are attached to $Y_n$ by composing with $\pi^{n-1}$ the attaching maps for the essential $n$-cells of $X$. This makes sense because $X_{n-1}^+$ contains the $(n-1)$-skeleton of $X$. Extend $\pi^{n-1}$ to an $M$-homotopy equivalence $\widehat{\pi^{n-1}}\colon X_n \rightarrow Y_n$ in the obvious way. This is possible since $X_{n}$ is obtained from $X_{n-1}^+$ by adjoining the essential $n$-cells. Composing $r_n$ with $\widehat{\pi^{n-1}}$ then gives an $M$-homotopy equivalence $X_n^+ \rightarrow Y_n$. Passing to the union gives the $M$-homotopy equivalence $X \simeq_M Y$ stated in the theorem.
\end{proof}

There is an obvious dual result for right simplicial sets which follows from the above result by replacing $M$ by $M^{op}$. We also have the two-sided version.

\begin{theorem}\label{main_theorem_bi}
Let $K$ be a rigid free bi-$M$-simplicial set with $M$-equivariant collapsing scheme $(E,R,C,c,i)$. Then, with the above notation, there is a free bi-$M$-CW complex $Y$
with $Y \simeq_M |K|$ and such that the cells of $Y$ are in bijective correspondence with $E$,
and under this bijective correspondence
$Y_n$ is a free bi-$M$-set with basis $E_n^B$ for all $n$.
\end{theorem}
\begin{proof}
This follows by applying Theorem~\ref{main_theorem} to the rigid free left $M \times M^{op}$-simplicial set $K$.
\end{proof}

\section{Guarded collapsing schemes}
\label{sec_guarded}

In this section we introduce the idea of a \emph{left \gu collapsing scheme}. We shall prove that whenever $BM$ admits a left \gu collapsing scheme then $\EMr$ will admit an $M$-equivariant collapsing scheme whose $M$-orbits of cells are in bijective correspondence with the essential cells of the given collapsing scheme for $BM$. Applying Theorem~\ref{main_theorem} it will then follow that when $BM$ admits a left guarded collapsing scheme of finite type then the monoid $M$ is of type left-$\Finfty$. Analogous results will hold for right \gu and right-$\Finfty$, and (two-sided) \gu and bi-$\Finfty$. In later sections we shall give some examples of monoids which admit \gu collapsing schemes of finite type, including monoids with finite complete presentations (rewriting systems), and factorable monoids in the sense of \cite{HessArxiv}.

\begin{definition}[Guarded collapsing schemes]\label{def_guarded}
Let $K = \bigcup_{i \geq 0} K_i$ be a simplicial set and let $X$ be its geometric realisation. We identify the cells of $X$ with the non-degenerate simplices of $K$, and suppose that $K$ admits a collapsing scheme $(E,C,R,c,i)$.
We say that this collapsing scheme is
\begin{itemize}
\item \emph{left \gu} if for every redundant $n$-cell $\tau$ we have $i(\tau) \neq 0$;
\item \emph{right \gu} if for every redundant $n$-cell $\tau$ we have $i(\tau) \neq n+1$;
\item \emph{\gu} if it is both left and right \gud.
\end{itemize}
\end{definition}

In other words, the collapsing scheme is \gu provided the function $i$ never takes either of its two possible extreme allowable values. The aim of this section is to prove the following result.

\begin{theorem}\label{thm_guarded}
Let $M$ be a monoid and suppose that $BM$ admits a collapsing scheme $(E,C,R,c,i)$.
\begin{enumerate}[(a)]
\item If $(E,C,R,c,i)$ is left \gu, then there is an $M$-equivariant collapsing scheme $(\E,\R,\C,\cc,\ii)$ for the free left $M$-simplicial set $\EMr$ such that, for each $n$, the set of essential $n$-cells $\E_n$ is a free left $M$-set of rank $|E_n|$.
\item If $(E,C,R,c,i)$ is right \gu, then there is an $M$-equivariant collapsing scheme $(\E,\R,\C,\cc,\ii)$ for the free right $M$-simplicial set $\EMl$ such that, for each $n$, the set of essential $n$-cells $\E_n$ is a free right $M$-set of rank $|E_n|$.
\item If $(E,C,R,c,i)$ is \gu, then there is an $M$-equivariant collapsing scheme $(\E,\R,\C,\cc,\ii)$ for the free bi-$M$-simplicial set $\EMb$ such that, for each $n$, the set of essential $n$-cells $\E_n$ is a free $M\times M^{op}$-set of rank $|E_n|$.
\end{enumerate}
\end{theorem}


\begin{corollary}
\label{cor_main}
Let $M$ be a monoid and let $|BM|$ be its classifying space.
\begin{enumerate}[(a)]
\item If $BM$ admits a left \gu collapsing scheme of finite type then $M$ is of type left $\Finfty$ (and therefore also is of type left-$\FPinfty$).
\item If $BM$ admits a right \gu collapsing scheme of finite type then $M$ is of type right-$\Finfty$ (and therefore also is of type right-$\FPinfty$).
\item If $BM$ admits a \gu collapsing scheme of finite type then $M$ is of type bi-$\Finfty$ (and therefore also is of type bi-$\FPinfty$).
\end{enumerate}
\end{corollary}
\begin{proof}
This follows directly from Theorem~\ref{main_theorem} and its dual, and Theorems~\ref{main_theorem_bi} and \ref{thm_guarded}.
\end{proof}

We shall give some examples of monoids to which this corollary applies in the next section.

The rest of this section will be devoted to the proof of Theorem~\ref{thm_guarded}. Clearly part (b) of the theorem is dual to part (a). The proofs of parts (a) and (c) are very similar, the only difference being that the stronger \gu condiion is needed for (c), while only left \gu is needed for (a). We will begin by giving full details of the proof of Theorem~\ref{thm_guarded}(a). Then afterwards we will explain the few modifications in the proof necessary to obtain the two-sided proof for (c), in particular highlighting the place where the \gu condition is needed.

\subsection{Proof of Theorem~\ref{thm_guarded}(a)}

Let $(E,C,R,c,i)$ be a left \gu collapsing scheme for $BM$. We can now `lift' this collapsing scheme to a collapsing scheme $(\E,\R,\C,\cc,\ii)$ for the simplicial set $\EMr$ in the following natural way. First observe that
\begin{eqnarray*}
&                 & m(m_1, m_2, ...,m_n) \ \mbox{is an $n$-cell of $EM$} \\
& \Leftrightarrow & m_i \neq 1 \ \mbox{for all $1 \leq i \leq n$} \\
& \Leftrightarrow & (m_1, m_2, ...,m_n) \ \mbox{is an $n$-cell of $BM$.}
\end{eqnarray*}

Define an $n$-cell $m(m_1, m_2, ...,m_n)$ of $\EMr$ to be essential (respectively redundant, collapsible respectively) if and only if $(m_1, m_2, ...,m_n)$ is essential (respectively redundant, collapsible respectively) in the collapsing scheme  $(E,R,C,c,i)$. This defines the partition $(\E,\R,\C)$ of the $n$-cells of $\EMr$ for each $n$. We call these sets the essential, redundant and collapsible cells, respectively, of $\EMr$. For the mappings $\cc$ and $\ii$, given $m \tau \in \EMr$ where $\tau = (m_1, ..., m_n)$ is a redundant $n$-cell of $BM$ we define
\begin{eqnarray}
\ii (m \tau) & = & i(\tau) \\
\cc (m \tau) & = & m (c(\tau)).
\end{eqnarray}
We claim that $(\E,\R,\C,\cc,\ii)$ is an $M$-equivariant collapsing scheme for the free left $M$-simplicial set $\EMr$ such that, for each $n$, the set of essential $n$-cells $\E_n$ is a free left $M$-set of rank $|E_n|$. Once proved this will complete the proof of Theorem~\ref{thm_guarded}(a).

We begin by proving that $(\E,\R,\C,\cc,\ii)$ is a collapsing scheme for $\EMr$, and then we will verify that all of the conditions (A1) to (A6) are satisfied.

\subsubsection{Proving that $(\E,\R,\C,\cc,\ii)$ is a collapsing scheme}

\begin{proposition}
\label{prop_CS}
With the above definitions,
$(\E,\R,\C,\cc,\ii)$ is a collapsing scheme for the simplicial set $\EMr$.
\end{proposition}

We have already observed that $(\E,\R,\C)$ partitions the cells of $\EMr$. To complete the proof of the proposition we need to verify that the conditions (C1) and (C2) in the definition of collapsing scheme are both satisfied. For part of the proof we will find it useful to recast the ideas in terms of certain bipartite digraphs.
The idea of viewing collapsing schemes in terms of matchings in bipartite graphs is a natural one and has been used in the literature; see for example \cite{Chari2000}.
Let us now introduce the terminology and basic observations about digraphs that we shall need.

A \emph{directed graph} $D$ consists of: a set of edges $ED$, a set of vertices $VD$ and functions $\alpha$ and $\beta$ from $ED$ to $VD$. For $e \in E$ we call $\alpha(e)$ and $\beta(e)$ the initial and terminal vertices of the directed edge $e$. A \emph{directed path} of length $n$ in $D$ is a sequence of edges $e_1 e_2 \ldots e_n$ such that $\beta(e_i) = \alpha(e_{i+1})$ for each directed edge. Note that edges in paths are allowed to repeat, and vertices can also repeat in the sense that $\beta(e_i) = \beta(e_j)$ is possible for distinct $i$ and $j$ (in graph theory literature what we call a path here is often called a walk). By an \emph{infinite directed path} we mean a path $(e_i)_{i \in \mathbb{N}}$. Note that an infinite directed path need not contain infinitely many distinct edges. For example, if a digraph contains a directed circuit $e_1 e_2 e_3$ then $e_1 e_2 e_3 e_1 e_2 e_3 \ldots$ would be an infinite directed path. A \emph{bipartite digraph} $D$ with bipartition $VD_1 \cup VD_2$ has vertex set  $VD = VD_1 \cup VD_2$ where $VD_1$ and $VD_2$ are disjoint, such that for every $e \in ED$ we have either $\alpha(e) \in VD_1$ and $\beta(e) \in VD_2$, or $\alpha(e) \in VD_2$ and $\beta(e) \in VD_1$ (i.e., there are no directed edges between vertices in the same part of the bipartition). A \emph{homomorphism} $\varphi\colon  D \rightarrow D'$ between digraphs is a map $\varphi\colon  (VD \cup ED) \rightarrow (VD' \cup ED')$ which maps vertices to vertices, edges to edges, and satisfies $\alpha (\varphi(e)) = \varphi(\alpha(e))$ and $\beta (\varphi(e)) = \varphi(\beta(e))$. If $p = e_1 e_2 \ldots e_n$ is a path of length $n$ in $D$ and $\varphi\colon  D \rightarrow D'$ is a digraph homomorphism then we define $\varphi(p) = \varphi(e_1) \varphi(e_2) \ldots \varphi(e_n)$ which is a path of length $n$ in $D'$. Note that in general a homomorphism is allowed to map two distinct vertices (resp. edges) of $D$ to the same vertex (resp. edge) of $D'$. Since digraph homomorphisms map paths to paths, we have the following basic observation.

\begin{Lemma}
\label{lem_easy}
Let $\varphi\colon D \rightarrow D'$ be a homomorphism of directed graphs. If $D$ has an infinite directed path than $D'$ has an infinite directed path.
\end{Lemma}

For each $n \in \mathbb{N}$ let $\Gamma^{(n)}(BM)$ be the directed bipartite graph defined as follows. The vertex set $V\Gamma^{(n)}(BM) = \mathcal{C}_n \cup \mathcal{C}_{n+1}$ where $\mathcal{C}_i$ denotes the set of $i$-cells $BM$. The directed edges $EV$ of $V$ are are of two types:

\

\noindent (DE1) A directed edge $\tau \longrightarrow d_j(\tau)$ (with initial vertex $\tau$ and terminal vertex $d_j(\tau)$) for every collapsible $\tau \in \mathcal{C}_{n+1}$ and $j \in \{0, \ldots, n+1\}$ such that $d_j(\tau)$ is a redundant $n$-cell (i.e., is a redundant non-degenerate $n$-simplex) and either $c(d_j(\tau))\neq \tau$ or $c(d_j(\tau))=\tau$ but $j\neq i(d_j(\tau))$;

\

\noindent (DE2) A directed edge $\sigma \longrightarrow c(\sigma)$ (with initial vertex $\sigma$ and terminal vertex $c(\sigma)$) for every redundant $\sigma \in \mathcal{C}_n$.

\

We sometimes refer to these two types of directed arcs as the ``down-arcs'' and the ``up-arcs'' (respectively) in the bipartite graph. Note that condition (C2) in the collapsing scheme definition is equivalent to saying that the digraph $D(n,n+1)$ does not contain any infinite directed path, and in particular does not contain any directed cycles. Let $\Gamma^{(n)}(\EMr)$ be the corresponding directed bipartite graph defined in the same way using $(\E,\R,\C,\cc,\ii)$, with vertex set the $n$ and $n+1$ cells of $\EMr$ and directed edges determined by the maps $\cc$ and $\ii$. To simplify notation, let us fix $n \in \mathbb{N}$ and set $\Gamma(BM) = \Gamma^{(n)}(BM)$ and $\Gamma(EM) = \Gamma^{(n)}(EM)$.

\begin{Lemma}
\label{lem_digraphhom}
Let $\pi\colon V\Gamma(\EMr) \cup E\Gamma(\EMr) \rightarrow V\Gamma(BM) \cup E\Gamma(BM)$ be defined on vertices by:
\[
m(m_1,\ldots,m_k) \mapsto (m_1,\ldots,m_k) \ (k \in \{n,n+1\})
\]
and defined on edges (DE1) and (DE2) by $\pi(x \rightarrow y) = \pi(x) \rightarrow \pi(y)$. Then $\pi$ is a digraph homomorphism.
\end{Lemma}
%
%
%
\begin{proof}
We need to prove, for each directed edge $x \rightarrow y$ from $E\Gamma(\EMr)$, that $\pi(x) \rightarrow \pi(y)$ is a directed edge in $E\Gamma(BM)$. There are two cases that need to be considered depending on arc type. The two arc types depend on whether the arc is going downwards from level $n+1$ to level $n$ (arc type 1) or
upwards from level $n$ to level $n+1$ (arc type 2).

\

\noindent \textbf{Case: Arc type 1.} Let $m(m_1,\ldots,m_{n+1})$ be a collapsible $n+1$ cell in $\EMr$ and let $j \in \{0, \ldots, n+1\}$ and suppose that
\[
m(m_1,\ldots,m_{n+1}) \longrightarrow d_j(m(m_1,\ldots,m_{n+1}))
\]
is an arc in $\Gamma(\EMr)$. This means that $d_j(m(m_1,\ldots,m_{n+1})) $ is a redundant $n$-cell in $\EMr$ and  if $\kappa(d_j(m(m_1,\ldots, m_{n+1})))= m(m_1,\ldots, m_{n+1})$, then $j \neq \iota(d_j(m(m_1,\ldots,m_{n+1})))$. Note that $j=0$ or $j=n+1$ is possible here.
We claim that
\[
\pi(m(m_1,\ldots,m_{n+1})) \longrightarrow \pi(d_j(m(m_1,\ldots,m_{n+1})))
\]
in $\Gamma(BM)$. Indeed, we saw above in Section~\ref{sec_cspaces} that the projection $\pi\colon  \EMr \rightarrow BM$ is a simplicial morphism with
\[
\pi(d_j(m(m_1,\ldots,m_{n+1}))) = d_j(m_1,\ldots,m_{n+1}) = d_j(\pi(m(m_1,\ldots,m_{n+1}))).
\]
Therefore
\begin{align*}
& \phantom{= \;} \; \pi(m(m_1,\ldots,m_{n+1})) \\
& =
(m_1,\ldots,m_{n+1})
\rightarrow
d_j(m_1,\ldots,m_{n+1}) \\
& =
\pi(d_j(m(m_1,\ldots,m_{n+1}))),
\end{align*}
in $\Gamma(BM)$, since if $c(d_j(m_1,\ldots, m_{n+1}))=(m_1,\ldots, m_{n+1})$, then
\[\kappa(d_j(m(m_1,\ldots, m_{n+1})))=\kappa(md_j(m_1,\ldots, m_{n+1})) = mc(d_j(m_1,\ldots, m_{n+1})) = m(m_1,\ldots, m_{n+1})\] and hence by definition of $\iota$ we have
\[
j \neq \iota(d_j(m(m_1,\ldots,m_{n+1}))) = i(d_j(m_1,\ldots,m_{n+1})).
\]

\

\noindent \textbf{Case: Arc type 2.} These are the arcs that go up from level $n$ to level $n+1$. A typical such arc arises as follows. Let $m\sigma \in \mathcal{C}_n$ be a redundant cell in $\EMr$ where $\sigma$ is a redundant cell in $BM$ and $m \in M$. Then $\kappa(m\sigma) \in \mathcal{C}_{n+1}$ is collapsible and
\[
d_{\iota(m\sigma)}(\kappa(m\sigma)) = m\sigma.
\]
A typical type 2 arc has the form
\[
m\sigma \longrightarrow \kappa(m\sigma) = m(c(\sigma)),
\]
by definition of $\kappa$. Applying $\pi$ to this arc gives
\[
\pi(m\sigma) = \sigma \longrightarrow c(\sigma) = \pi(m(c(\sigma))),
\]
which is a type 2 arc in $\Gamma(BM)$ completing the proof of the lemma.
\end{proof}

%

\begin{proof}[Proof of Proposition~\ref{prop_CS}.]
We must check that $(\E,\R,\C,\cc,\ii)$ satisfies the two collapsing scheme conditions (C1) and (C2).

\

\noindent \textbf{Verifying collapsing scheme condition (C1).} We must prove that the map $\cc$ defines a bijection from the redundant $n$-cells to the collapsible $n+1$-cells.

The map is injective since $\cc(m\tau) = \cc(m'\sigma)$ implies $m c(\tau) = m' c(\sigma)$ so $m=m'$ and $\tau$ = $\sigma$ since $c$ is injective by assumption. Also, given an arbitrary collapsible $n+1$ cell $m \sigma$ there exists $\sigma = c(\tau)$ where $\tau$ is a redundant $n$-cell an so $m \sigma = \cc(m \tau)$.

Moreover, for every redundant $n$-cell $m\tau$, we have
\begin{eqnarray*}
d_{\ii (m\tau)} (\cc (m\tau)) & = & d_{i(\tau)} (mc(\tau)) \ \mbox{(by definition)} \\
& = & md_{i(\tau)} (c(\tau)) \ \mbox{(since $i(\tau) \neq 0$)} \\
& = & m\tau,
\end{eqnarray*}
and this concludes the proof that collapsing scheme condition (C1) holds.

Note that it is in the second line of the above sequence of equations that we appeal to our assumption that $(E,C,R,c,i)$ is left \gu (which implies $i(\tau) \neq 0$). In fact, this will be the only place in the proof of Theorem~\ref{thm_guarded}(a) that the left \gu assumption is used.

\

\noindent \textbf{Verifying collapsing scheme condition (C2).} To see that collapsing scheme condition (C2) holds let $\Gamma(BM)$ and $\Gamma(\EMr)$ be the level ($n$, $n+1$) bipartite digraphs of $BM$ and $\EMr$, respectively, defined above.
By Lemma~\ref{lem_digraphhom} the mapping $\pi$ defines a digraph homomorphism from $\Gamma(\EMr)$ to $\Gamma(BM)$.
If $\Gamma(\EMr)$ contained an infinite directed path then by Lemma~\ref{lem_easy} the image of this path would be an infinite directed path in $\Gamma(BM)$ which is impossible since $(E,R,C,c,i)$ is a collapsing scheme. Therefore $\Gamma(\EMr)$ contains no infinite path and thus $(\E,\R,\C,\cc,\ii)$ satisfies condition (C2).

This completes the proof of Proposition~\ref{prop_CS}.
\end{proof}

\begin{remark}\label{rem_downarcs}
It follows from Proposition~\ref{prop_CS} that every down arc in $\Gamma(\EMr)$ (that is, every arc of type (DE1)) is of the form $\tau \rightarrow d_j(\tau)$ where $\tau \in \C_{n+1}$, $d_j(\tau) \in \R_n$, and $\kappa(d_j(\tau)) \neq \tau$.   
\end{remark}

\subsubsection{Proving $(\E,\R,\C,\cc,\ii)$ is a left $M$-equivariant collapsing scheme}

To prove that $(\E,\R,\C,\cc,\ii)$ is a left $M$-equivariant collapsing scheme for $\EMr$ we need to verify that conditions (A1)--(A6) hold. In Section~\ref{sec_cspaces} we proved that $\EMr$ is a rigid free left $M$-simplicial set. In addition to this, from the definitions an $n$-cell $m(m_1, m_2, ...,m_n)$ of $\EMr$ essential (redundant, collapsible respectively) if and only if $(m_1, m_2, ...,m_n)$ is essential (redundant, collapsible respectively) in the collapsing scheme  $(E,R,C,c,i)$ of $BM$. These facts prove that (A1) and (A2) both hold:
\begin{enumerate}
\item[(A1)] The action of $M$ on $\EMr$ maps $n$-simplicies to $n$-simplicies, and
commutes with $d_i$ and $s_i$, that is, $M$ is acting by simplicial morphisms on $\EMr$.
\item[(A2)] For every $n$-simplex $\sigma$ and $m \in M$, $\sigma$ is a cell (i.e. is a non-degenerate simplex) if and only if $m\sigma$ is a cell, in which case $\sigma \in \E$ (respectively $\R$, $\C$) if and only if $m\sigma \in \E$ (respectively $\R$, $\C$).
\end{enumerate}

\

\noindent The next axiom we need to verify is:

\

\begin{enumerate}
\item[(A3)] If $(\sigma,\tau) \in \R_n \times \C_{n+1}$ is a matched redundant-collapsible pair (i.e. $\tau = c(\sigma)$) then so is the pair $m(\sigma,\tau) = (m\sigma,m\tau) \in \R_n \times \C_{n+1}$.
\end{enumerate}

\

\noindent We shall prove a little more than this. We consider the bipartite digraph $\Gamma(\EMr)$ between levels $n$ and $n+1$ defined above. We want to prove that $M$ acts on this digraph, that is, that arcs are sent to arcs under the action.  We note that the action of $M$ on the vertices preserves the bipartition. As above, there are two types of directed arcs that need to be considered, the up-arcs and the down-arcs.

First consider the down-arcs. My Remark~\ref{rem_downarcs}, a typical down-arc has the form
\[
m\sigma \xrightarrow{d_j} d_j(m\sigma)
\]
where $\kappa(d_j(m\sigma)) \neq m\sigma$. 
%
%
%
Let $n \in M$ be arbitrary. Then we have
\[
nm\sigma \xrightarrow{d_j} d_j(nm\sigma) = nd_j(m\sigma),
\]
by definition of $d_j$. This is a down-arc because if $\kappa(d_j(nm\sigma))=nm\sigma$, then from $nmd_j(\sigma)=d_j(nm\sigma)$ we deduce that $nmc(d_j(\sigma))=nm\sigma$ and so $c(d_j(\sigma))=\sigma$, whence $\kappa(d_j(m\sigma))=\kappa(md_j(\sigma)) = mc(d_j(\sigma))=m\sigma$, which is a contradiction.  
%

Now consider up-arcs. A typical up-arc has the form $m\sigma \rightarrow \kappa(m\sigma)$. Let $n \in M$. Then
\[
nm\sigma \rightarrow \kappa(nm\sigma) = nm(c(\sigma)) = n\kappa(m\sigma),
\]
which is an up-arc as required.

This covers all types of arcs in $\Gamma(\EMr)$ and we conclude that $M$ acts on $\Gamma(\EMr)$ by digraph endomorphisms. This fact together with (A2) then implies property (A3), since the up-arcs in this bipartite graph are precisely the matched redundant-collapsible pairs. Next consider


\

\begin{enumerate}
\item[(A4)] There is a subset $\B \subseteq \E \cup \R \cup \C$ such that for all $n$ the set of $n$-cells is a free left $M$-set with basis $\B_n$ (the $n$-cells in $\B$). Let $\E^\B = \E \cap \B$, $\R^\B = \R \cap \B$ and $\C^\B = \C \cap \B$.
Then $\E_n$ is a free left $M$-set with basis $\E_n^\B$, and similarly for $\R_n$ and $\C_n$.
\end{enumerate}

\

\noindent We saw in Section~\ref{sec_cspaces} that the set of $n$-cells of $\EMr$ is a free left $M$-set with basis the set of $n$-cells
\[
\B = \{
(m_1,\ldots, m_n)\colon\quad m_i \neq 1 \ \mbox{for all} \ i
\}
\]
of $BM$. The last clause of (A4) then follows from (A2). Now we shall prove

\

\begin{enumerate}
\item[(A5)] For every matched pair $(\sigma,\tau) \in \R \times \C$, $\sigma \in \R^\B$ if and only if $\tau \in \C^\B$. In particular, for every matched pair $(\sigma,\tau)$ there is a unique pair $(\sigma',\tau') \in \R^\B \times \C^\B$ and $m \in M$ such that $(\sigma, \tau) = m(\sigma', \tau')$.
\end{enumerate}

\

\noindent The matched pairs are the up-arcs in the graph $\Gamma(\EMr)$. A typical up-arc has the form
\[
m(m_1,\ldots, m_n) \rightarrow
\kappa(m(m_1,\ldots, m_n) =
m c(m_1, \ldots, m_n).
\]
So this pair is
\[
m \cdot (
(m_1,\ldots, m_n) \rightarrow
c(m_1, \ldots, m_n)
)
\]
where
\[
(m_1,\ldots, m_n) \rightarrow
c(m_1, \ldots, m_n)
\]
is a uniquely determined matched basis pair. Also, if $(\sigma,\tau)$ is a matched pair then $\sigma = (m_1,\ldots, m_n)$ belongs to the basis if and only if $\kappa(m_1, \ldots, m_n) = c(m_1, \ldots, m_n)$ belongs to the basis, completing the proof of (A5). Finally we turn our attention to showing axiom

\

\begin{enumerate}
\item[(A6)] For every redundant cell $\tau$ and every $m\in M$
\[
\mathrm{height}(\tau) = \mathrm{height}(m\tau)
\]
where height is taken with respect to the collapsing scheme $(\E,\R,\C,\cc,\ii)$.
\end{enumerate}

\

The following lemma will be useful.

\begin{Lemma}[Path lifting property]\label{lem:PathLift}
Define a mapping $\pi\colon V\Gamma(\EMr) \cup E\Gamma(\EMr) \rightarrow V\Gamma(BM) \cup E\Gamma(BM)$ defined on vertices by:
\[
m(m_1,\ldots,m_k) \mapsto (m_1,\ldots,m_k) \ (k \in \{n,n+1\})
\]
and defined on edges (DE1) and (DE2) by $\pi(x \rightarrow y) = \pi(x) \rightarrow \pi(y)$.
Let $\mu$ be a redundant $n$-cell in $\EMr$. Then for each path $p$ in $\Gamma(BM)$, with initial vertex $\pi(\mu)$ there is a path $\hat{p}$ in $\Gamma(\EMr)$, with initial vertex $\mu$, such that $\pi(\hat{p}) = p$.
%
%
\end{Lemma}
\begin{proof}
We shall establish two claims, from which the lemma will be obvious by induction on path length.  First we claim that if $y=m\tau$ is a redundant $n$-cell of $\EMr$ (with $m \in M$ and $\tau$ and $n$-cell of $BM$) and $\sigma\in V\Gamma(BM)$ is a vertex such that there is a directed edge $\tau=\pi(y)\to \sigma$, then there is a vertex $z \in V\Gamma(\EMr)$ such that $y \rightarrow z$ is a directed edge in $E\Gamma(\EMr)$, and $\pi(y \rightarrow z) = \pi(y) \rightarrow \sigma$.  Indeed, set $z = \cc(y)$. Then $y \rightarrow z$ is a directed edge in $E\Gamma(\EMr)$ by definition and $\pi(z) = \pi(\cc(y)) = \pi(\cc(m\tau)) = \pi(m(c(\tau))) = c(\tau) = \sigma$.



Next we claim that if $x \rightarrow y$ is an up-arc of $E\Gamma(\EMr)$ as in (DE2) and $\sigma$ is a vertex in $V\Gamma(BM)$ such that $\pi(y) \rightarrow \sigma$ is a directed edge in $E\Gamma(BM)$, then there exists a vertex $z \in V\Gamma(\EMr)$ such that $y \rightarrow z$ is a directed edge in $E\Gamma(\EMr)$, and $\pi(y \rightarrow z) = \pi(y) \rightarrow \sigma$. Write $x = m\tau$ where $m \in M$ and $\tau$ is a redundant $n$-cell of $BM$. Then $x \rightarrow y$ is equal to $m\tau \longrightarrow \cc(m\tau) = m c(\tau)$. In $\Gamma(BM)$ we have the path $\pi(x) \rightarrow \pi(y) \rightarrow \sigma$ which equals $\tau \rightarrow c(\tau) \rightarrow \sigma$. Therefore $c(\tau) \rightarrow \sigma$ is an arc in $\Gamma(BM)$ of type 1. Therefore, by Remark~\ref{rem_downarcs} applied to the graph $\Gamma(BM)$, it follows that $\sigma$ is a redundant $n$-cell with $\sigma =  d_j(c(\tau))$ for some $j \in \{0, \ldots, n+1\}$ and 
$\sigma \neq \tau$ (using that $c$ is a bijection). Set $z = d_j(y)$. We need to show that
$\pi(z) = \sigma$ and that
$y \rightarrow z$ is a directed edge in $\Gamma(\EMr)$.

For the first claim, since $\pi$ is a simplicial morphism we have
\[
\pi(z) = \pi(d_j(y)) = d_j(\pi(y)) = d_j(c(\tau)) = \sigma.
\]

For the second claim, to verify that $y \rightarrow z$ is a directed edge in $\Gamma(\EMr)$ we just need to show that $z$ is a redundant cell and $y \neq \cc(z)$.  
From the definitions it follows that under the mapping $\pi$ any vertex in the preimage of a redundant cell is a redundant cell. Thus, since $\sigma$ is redundant and $\pi(z) = \sigma$ it follows that $z$ is redundant. Finally, if $y=\cc(z)$, then $z=x$ because $\cc$ is injective.
Therefore, $\sigma =\pi(z)=\pi(x)=\tau$, which is a contradiction. 
%
%
%
%
%
%
%
%
%
%
%
%
%
%
%
%

We can now construct the path $\hat{p}$ by induction on the length of $p$, where the inductive step uses the first claim if the lift of the previous portion ends at a redundant cell and uses the second claim if it ends at a collapsible cell.
\end{proof}

Axiom (A6) is then easily seen to be a consequence of the following lemma.

\begin{Lemma}
Let $\tau$ be a redundant cell in the collapsing scheme $(\E,\R,\C,\cc,\ii)$. Write $\tau = m \sigma$ where $\sigma$ is a redundant cell in $BM$. Let $\mathrm{height}_{\EMr}(m \sigma)$ denote the height of $m \sigma$ with respect to the collapsing scheme $(\E,\R,\C,\cc,\ii)$, and let $\mathrm{height}_{BM}(\sigma)$ denote the height of $\sigma$ with respect to the collapsing scheme $(E,R,C,c,i)$. Then $\mathrm{height}_{\EMr}(m \sigma) = \mathrm{height}_{BM}(\sigma)$.
\end{Lemma}
\begin{proof}
Let
\[
m\sigma = \tau = \tau_0 \succ \tau_1 \succ \cdots \succ \tau_k
\]
be a descending chain of redundant $n$-cells from $\R$. It follows that there is a directed path
\[
p = \tau_0 \rightarrow \cc(\tau_0) \rightarrow \cdots \rightarrow \cc(\tau_{k-1}) \rightarrow \tau_k
\]
in $\Gamma(\EMr)$. Since $\pi$ is a digraph homomorphism it follows that $\pi(p)$ is a directed path in $\Gamma(BM)$ and hence
\[
\sigma = \pi(\tau_0) \succ \pi(\tau_1) \succ \cdots \succ \pi(\tau_k)
\]
is a descending chain of redundant $n$-cells in $R$. This proves that $\mathrm{height}_{\EMr}(m \sigma) \leq \mathrm{height}_{BM}(\sigma)$.

For the converse, let
\[
\sigma = \sigma_0 \succ \sigma_1 \succ \cdots \succ \sigma_k
\]
be a descending chain of redundant $n$-cells from $R$. Then there is a directed path
\[
q = \sigma_0 \rightarrow c(\sigma_0) \rightarrow \cdots \rightarrow c(\sigma_{k-1}) \rightarrow \sigma_k
\]
in $\Gamma(BM)$. By Lemma~\ref{lem:PathLift} we can lift $q$ to a path $\hat{q}$ in $\Gamma(\EMr)$ with initial vertex $m\sigma$ and such that $\pi(\hat{q}) = q$. But then the redundant cells in the path $\hat{q}$ form a decending chain of length $k$ starting at $m \sigma$, proving that $\mathrm{height}_{\EMr}(m \sigma) \geq \mathrm{height}_{BM}(\sigma)$.
\end{proof}

This completes the proof of Theorem~\ref{thm_guarded}(a) and its dual Theorem~\ref{thm_guarded}(b).

\subsection{Proof of Theorem~\ref{thm_guarded}(c)} We shall explain how the above proof of Theorem~\ref{thm_guarded}(c) is modified to prove the two-sided analogue Theorem~\ref{thm_guarded}(c).

Let $(E,C,R,c,i)$ be a \gu collapsing scheme. Define an $n$-cell $m(m_1, m_2, ...,m_n)u$ of $\EMb$ to be essential (respectively redundant, collapsible respectively) if and only if $(m_1, m_2, ...,m_n)$ is essential (respectively redundant, collapsible respectively) in the collapsing scheme  $(E,R,C,c,i)$. This defines $(\E,\R,\C)$.

For the mappings $\cc$ and $\ii$, given $m \tau u \in \EMb$ where $\tau = (m_1, ..., m_n)$ is an $n$-cell of $BM$ we define
\begin{eqnarray}
\ii (m \tau u) & = & i(\tau) \\
\cc (m \tau u ) & = & m (c(\tau)) u.
\end{eqnarray}
With these definitions we claim that $(\E,\R,\C,\cc,\ii)$ is an $M$-equivariant collapsing scheme for the free bi-$M$-simplicial set $\EMb$ such that for each $n$ the set of essential $n$-cells $\E_n$ is a free bi-$M$-set of rank $|E_n|$.

\subsubsection{Proving that $(\E,\R,\C,\cc,\ii)$ is a collapsing scheme}

\begin{proposition}
\label{prop_CS_bi}
With the above definitions,
$(\E,\R,\C,\cc,\ii)$ is a collapsing scheme for the simplicial set $\EMb$.
\end{proposition}

The proof is analogous to the proof of Proposition~\ref{prop_CS}. As in the proof of that proposition, we have a digraph homomorphism $\pi\colon  V\Gamma(\EMb) \cup E\Gamma(\EMb) \rightarrow V\Gamma(BM) \rightarrow E\Gamma(BM)$ given by
\begin{align*}
& m(m_1, \ldots, m_k)n \rightarrow (m_1, \ldots, m_k), \\
& \pi(x \rightarrow y) = \pi(x) \rightarrow \pi(y).
\end{align*}
To verify collapsing scheme condition (C1) it suffices to observe that for every redundant $n$-cell $m\tau$ we have
\begin{eqnarray*}
d_{\ii (m\tau n)} (\cc (m\tau n)) & = & d_{i(\tau)} (mc(\tau)n) \ \mbox{(by definition)} \\
& = & md_{i(\tau)} (c(\tau))n \ \mbox{(since $i(\tau) \neq 0$ and $i(\tau) \neq n+1$)} \\
& = & m\tau n.
\end{eqnarray*}
Here the second line appeals to the fact that the original collapsing scheme is \gud.
Collapsing scheme condition (C2) holds again by applying Lemma~\ref{lem_easy} and the fact that $\pi$ is a digraph homomorphism.

\subsubsection{Proving $(\E,\R,\C,\cc,\ii)$ is an $M$-equivariant collapsing scheme}

We have already seen in Section~\ref{sec_cspaces} that $\EMb$ is a bi-$M$-simplicial set. We need to show that $(\E,\R,\C,\cc,\ii)$ is an $M$-equivariant collapsing scheme for $\EMb$. By definition, for this we need to verify that axioms (A1)-(A6) are satisfied for $\EMb$ as a left $M \times M^{op}$-simplicial set.

In Section~\ref{sec_cspaces} we saw that $\EMb$ is a rigid free left $M \times M^{op}$-simplicial set. Together with the definition of the action of $M \times M^{op}$ on $\EMb$ and the definition of $\E$, $\R$ and $\C$, axioms (A1) and (A2) both then follow. Axioms (A3)-(A6) are then proved exactly as above just with $M \times M^{op}$ in place of $M$ in the proof. This then completes the proof of Theorem~\ref{thm_guarded}(c).

\section{Monoids admitting guarded collapsing schemes}
\label{sec_admitting}

In this section we give examples of classes of monoids to which the above theory of equivariant collapsing schemes applies. In particular, this will allow us to use the theory developed above to give a topological proof of the fact that monoids which admit finite complete presentations are of type bi-$\Finfty$.

%
%

Let $M$ be a monoid defined by a finite presentation $\langle A \mid R \rangle$ with generators $A$ and defining relations $R \subseteq A^* \times A^*$. Thus, $M$ is isomorphic to $A^* /  \ThCon{R}$ where $ \ThCon{R}$ is the smallest congruence on $A^*$ containing $R$. We view $\langle A \mid R \rangle$ as a string rewriting system, writing $l \rightarrow r$ for the pair $(l,r) \in R$. We define a binary relation $\rightarrow$ on $A^*$, called a {\it
single-step reduction}, in the following way:
\[
u \rightarrow v \
\Leftrightarrow \ u \equiv w_1l w_2 \ {\rm and} \ v \equiv w_1r w_2
\]
for some $(l,r) \in R $ and $w_1,w_2\in X^* $.
A word is called \emph{irreducible} if no single-step reduction rule may be applied to it.
The
transitive and reflexive closure of $\ssr{R} $ is denoted by
$\RhCon{R}$.

This rewriting system is called \emph{noethereian} if there
are no infinite descending chains
\[
w_1\ssr{R}w_2\ssr{R}w_3\ssr{R}\cdots
\ssr{R}w_n\ssr{R}\cdots.
\]
It is called {\it confluent} if
whenever we have $u\RhCon{R}v$ and $u\RhCon{R}v'$ there is a word
$w\in X^*$ such that $v\RhCon{R}w$ and $v'\RhCon{R}w$. If $R$ is
simultaneously noetherian and confluent we say that $R$ is {\it
complete}. The presentation $\langle A \mid R \rangle$ is called complete if the rewriting system $R$ is complete.

It is well-known (see for example \cite{BookAndOtto}) that if $\langle A \mid R \rangle$ is a finite complete presentation then each  $\ThCon{R}$-class of $A^*$ contains exactly one irreducible element. So the set of irreducible elements give a set of normal forms for the elements of the monoid $M$. In particular, if a monoid admits a presentation by a finite compete rewriting system then the word problem for the monoid is decidable.

%
%
%
%

In \cite[page 145]{Brown1989} a method is given for constructing a collapsing scheme on $BM$ for any monoid $M$ that is given by a finite complete rewriting system. It is easily observed from \cite[page 145]{Brown1989}
that in the given collapsing scheme $(E,R,C,c,i)$ the function $i$ never takes either of its two extreme allowable values, that is, the collapsing scheme for $BM$ given in \cite{Brown1989} is \gu in the sense of Definition~\ref{def_guarded}. Also, as Brown observes (see \cite[page 147]{Brown1989}), it follows easily from his definition that there are only finitely many essential cells in each dimension. Thus we have:

\begin{proposition}
\label{prop_FCRS}
Let $M$ be a monoid. If $M$ admits a presentation by a finite complete rewriting system then $BM$ admits a \gu collapsing scheme of finite type.
\end{proposition}

It follows that the theory of $M$-equivariant collapsing schemes developed in the previous section applies to monoids with complete presentations, giving:

\begin{corollary}\label{corol_FCRSFn}
Let $M$ be a monoid that admits a presentation by a finite complete rewriting system. Then $M$ is of type left-$\F_\infty$, right-$\F_\infty$ and  bi-$\F_\infty$.
\end{corollary}
\begin{proof}
By Proposition~\ref{prop_FCRS} the simplicial set $BM$ admits a \gu collapsing scheme of finite type.
Then the result follows from Corollary~\ref{cor_main}.
\end{proof}

We obtain the following result of Kobayashi as a special case.

\begin{corollary}[\!\!\cite{Kobayashi2005}]
Let $M$ be a monoid that admits a presentation by a finite complete rewriting system. Then $M$ is of type bi-$\FP_\infty$.
\end{corollary}
\begin{proof}
Follows from Proposition~\ref{p:basic.props.bi.fn} and Corollary~\ref{corol_FCRSFn}.
\end{proof}

More recently the class of, so-called, factorable monoids was introduced in work of Hess and Ozornova \cite{HessArxiv}. Since it is quite technical we shall not give the definition of factorable monoid here, we refer the reader to \cite{HessArxiv} to the definition, and we shall use the same notation as there. In their work they show that a number of interesting classes of monoids admit factorability structures. In some cases (e.g. Garside groups) the rewriting systems associated with factorability structures are finite and complete, and so in these cases the monoids are bi-$\Finfty$. On the other hand, in \cite[Appendix]{HessArxiv} they give an example of a factorable monoid where the associated rewriting system is not noetherian and thus not complete (although it is not discussed there whether the monoid admits a presentation by some other finite complete presentation). So, there are some examples where factorability structures may be seen to exist, even when the given presentation is not complete. In \cite[Section~9]{HessArxiv} the authors construct a matching on the reduced, inhomogeneous bar complex of a factorable monoid. As they say in their paper, the matching they construct is very similar to the construction used by Brown giving a collapsing scheme for monoids defined by finite complete rewriting systems \cite[page 145]{Brown1989}. Details of the matching they construct for factorable monoids may be found on pages 27 and 28 of \cite{HessArxiv}. It is immediate from the definition of the mapping $\mu$ on page 28 that their construction defines a \gu collapsing scheme for the simplicial set $BM$ where $M$ is a any factorable monoid $(M, \mathcal{E}, \eta)$. If the generating set $\mathcal{E}$ for the monoid is finite, then the number of essential cells in each dimension is finite, and so we have a guarded collapsing scheme of finite type, giving:

\begin{corollary}
\label{prop_factorable}
Let $M$ be a monoid. If $M$ admits a factorability structure $(M, \mathcal{E}, \eta)$ with finite generating set $\mathcal{E}$ then $BM$ admits a \gu collapsing scheme of finite type. In particular $M$ is of type left-$\Finfty$, right-$\Finfty$ and bi-$\Finfty$.
\end{corollary}

%


\begin{thebibliography}{CTVEZ03}

\bibitem[AH03]{AlonsoHermiller2003}
J.~M. Alonso \& S.~M. Hermiller.
\newblock `Homological finite derivation type'.
\newblock {\em Internat. J. Algebra Comput.}, 13, no.~3 (2003), pp. 341--359.
\newblock {\sc doi:} \href {http://dx.doi.org/10.1142/S0218196703001407}
  {{10.1142/S0218196703001407}}.

\bibitem[Alo94]{Alonso1994}
J.~M. Alonso.
\newblock `Finiteness conditions on groups and quasi-isometries'.
\newblock {\em J. Pure Appl. Algebra}, 95, no.~2 (1994), pp. 121--129.
\newblock {\sc doi:} \href {http://dx.doi.org/10.1016/0022-4049(94)90069-8}
  {{10.1016/0022-4049(94)90069-8}}.

\bibitem[Ani86]{Anick1986}
D.~J. Anick.
\newblock `On the homology of associative algebras'.
\newblock {\em Trans. Amer. Math. Soc.}, 296, no.~2 (1986), pp. 641--659.
\newblock {\sc doi:} \href {http://dx.doi.org/10.2307/2000383}
  {{10.2307/2000383}}.

\bibitem[AR67]{Adams1967}
W.~W. Adams \& M.~A. Rieffel.
\newblock `Adjoint functors and derived functors with an application to the
  cohomology of semigroups'.
\newblock {\em J. Algebra}, 7 (1967), pp. 25--34.
\newblock {\sc doi:} \href {http://dx.doi.org/10.1016/0021-8693(67)90065-8}
  {{10.1016/0021-8693(67)90065-8}}.

\bibitem[BB97]{Bestvina1997}
M.~Bestvina \& N.~Brady.
\newblock `Morse theory and finiteness properties of groups'.
\newblock {\em Invent. Math.}, 129, no.~3 (1997), pp. 445--470.
\newblock {\sc doi:} \href {http://dx.doi.org/10.1007/s002220050168}
  {{10.1007/s002220050168}}.

\bibitem[BBG98]{Baumslag1998}
G.~Baumslag, M.~R. Bridson, \& K.~W. Gruenberg.
\newblock `On the absence of cohomological finiteness in wreath products'.
\newblock {\em J. Austral. Math. Soc. Ser. A}, 64, no.~2 (1998), pp. 222--230.

\bibitem[BG84]{BrownGeoghegan1984}
K.~S. Brown \& R.~Geoghegan.
\newblock `An infinite-dimensional torsion-free {${\rm FP}_{\infty }$} group'.
\newblock {\em Invent. Math.}, 77, no.~2 (1984), pp. 367--381.
\newblock {\sc doi:} \href {http://dx.doi.org/10.1007/BF01388451}
  {{10.1007/BF01388451}}.

\bibitem[BH01]{Bieri2001}
R.~Bieri \& J.~Harlander.
\newblock `On the {$\rm FP\sb 3$}-conjecture for metabelian groups'.
\newblock {\em J. London Math. Soc. (2)}, 64, no.~3 (2001), pp. 595--610.

\bibitem[BHMS02]{Bridson2002}
M.~R. Bridson, J.~Howie, C.~F. Miller, III, \& H.~Short.
\newblock `The subgroups of direct products of surface groups'.
\newblock {\em Geom. Dedicata}, 92 (2002), pp. 95--103.
\newblock Dedicated to John Stallings on the occasion of his 65th birthday.
\newblock {\sc doi:} \href {http://dx.doi.org/10.1023/A:1019611419598}
  {{10.1023/A:1019611419598}}.

\bibitem[Bie76]{Bieri1976}
R.~Bieri.
\newblock {\em Homological dimension of discrete groups}.
\newblock Mathematics Department, Queen Mary College, London, 1976.
\newblock Queen Mary College Mathematics Notes.

\bibitem[BO93]{BookAndOtto}
R.~V. Book \& F.~Otto.
\newblock {\em String-rewriting systems}.
\newblock Texts and Monographs in Computer Science. Springer-Verlag, New York,
  1993.

\bibitem[Bra99]{Brady1999}
N.~Brady.
\newblock `Branched coverings of cubical complexes and subgroups of hyperbolic
  groups'.
\newblock {\em J. London Math. Soc. (2)}, 60, no.~2 (1999), pp. 461--480.

\bibitem[Bro87]{Brown1987}
K.~S. Brown.
\newblock `Finiteness properties of groups'.
\newblock In {\em Proceedings of the {N}orthwestern conference on cohomology of
  groups ({E}vanston, {I}ll., 1985)}, vol.~44, pp. 45--75, 1987.
\newblock {\sc doi:} \href {http://dx.doi.org/10.1016/0022-4049(87)90015-6}
  {{10.1016/0022-4049(87)90015-6}}.

\bibitem[Bro92]{Brown1989}
K.~S. Brown.
\newblock `The geometry of rewriting systems: a proof of the
  {A}nick-{G}roves-{S}quier theorem'.
\newblock In {\em Algorithms and classification in combinatorial group theory
  ({B}erkeley, {CA}, 1989)}, vol.~23 of {\em Math. Sci. Res. Inst. Publ.}, pp.
  137--163. Springer, New York, 1992.
\newblock {\sc doi:} \href {http://dx.doi.org/10.1007/978-1-4613-9730-4-6}
  {{10.1007/978-1-4613-9730-4-6}}.

\bibitem[Bro94]{BrownCohomologyBook}
K.~S. Brown.
\newblock {\em Cohomology of groups}, vol.~87 of {\em Graduate Texts in
  Mathematics}.
\newblock Springer-Verlag, New York, 1994.
\newblock Corrected reprint of the 1982 original.

\bibitem[Bro10]{Brown2010}
K.~S. Brown.
\newblock `Lectures on the cohomology of groups'.
\newblock In {\em Cohomology of groups and algebraic {$K$}-theory}, vol.~12 of
  {\em Adv. Lect. Math. (ALM)}, pp. 131--166. Int. Press, Somerville, MA, 2010.

\bibitem[BW07]{Bux2007}
K.-U. Bux \& K.~Wortman.
\newblock `Finiteness properties of arithmetic groups over function fields'.
\newblock {\em Invent. Math.}, 167, no.~2 (2007), pp. 355--378.

\bibitem[Cha00]{Chari2000}
M.~K. Chari.
\newblock `On discrete {M}orse functions and combinatorial decompositions'.
\newblock {\em Discrete Math.}, 217, no.~1-3 (2000), pp. 101--113.
\newblock Formal power series and algebraic combinatorics (Vienna, 1997).
\newblock {\sc doi:} \href {http://dx.doi.org/10.1016/S0012-365X(99)00258-7}
  {{10.1016/S0012-365X(99)00258-7}}.

\bibitem[Che84]{Cheng1984}
C.~C.-a. Cheng.
\newblock `Separable semigroup algebras'.
\newblock {\em J. Pure Appl. Algebra}, 33, no.~2 (1984), pp. 151--158.
\newblock {\sc doi:} \href {http://dx.doi.org/10.1016/0022-4049(84)90003-3}
  {{10.1016/0022-4049(84)90003-3}}.

\bibitem[Cho06]{Chouraqui2006}
F.~Chouraqui.
\newblock `Rewriting systems in alternating knot groups'.
\newblock {\em Internat. J. Algebra Comput.}, 16, no.~4 (2006), pp. 749--769.

\bibitem[CO98]{Cremanns1998}
R.~Cremanns \& F.~Otto.
\newblock `{$\mathrm{FP}_{\infty}$} is undecidable for finitely presented
  monoids with word problems decidable in polynomial time'.
\newblock {\em Mathematische Schriften Kassel 11/98}, Universitat Kassel,
  September (1998) .

\bibitem[Coh73]{CohenBook}
M.~M. Cohen.
\newblock {\em A course in simple-homotopy theory}.
\newblock Springer-Verlag, New York, 1973.
\newblock Graduate Texts in Mathematics, Vol. 10.

\bibitem[Coh92]{Cohen1992}
D.~E. Cohen.
\newblock `A monoid which is right {$FP\sb \infty$} but not left {$FP\sb 1$}'.
\newblock {\em Bull. London Math. Soc.}, 24, no.~4 (1992), pp. 340--342.

\bibitem[Coh97]{Cohen1997}
D.~E. Cohen.
\newblock `String rewriting and homology of monoids'.
\newblock {\em Math. Structures Comput. Sci.}, 7, no.~3 (1997), pp. 207--240.

\bibitem[CS80]{Cheng1980}
C.~C.-a. Cheng \& J.~Shapiro.
\newblock `Cohomological dimension of an abelian monoid'.
\newblock {\em Proc. Amer. Math. Soc.}, 80, no.~4 (1980), pp. 547--551.
\newblock {\sc doi:} \href {http://dx.doi.org/10.2307/2043420}
  {{10.2307/2043420}}.

\bibitem[CS09a]{Silva2009_2}
J.~Cassaigne \& P.~V. Silva.
\newblock `Infinite periodic points of endomorphisms over special confluent
  rewriting systems'.
\newblock {\em Ann. Inst. Fourier}, 59, no.~2 (2009), pp. 769--810.

\bibitem[CS09b]{Silva2009_1}
J.~Cassaigne \& P.~V. Silva.
\newblock `Infinite words and confluent rewriting systems: endomorphism
  extensions'.
\newblock {\em Internat. J. Algebra Comput.}, 19, no.~4 (2009), pp. 443--490.

\bibitem[DDM09]{Miasnikov2009}
V.~Diekert, A.~Duncan, \& A.~Miasnikov.
\newblock `Geodesic rewriting systems and pregroups'.
\newblock {\em To appear in Combinatorial and Geometric Group Theory, Dortmund
  and Carleton Conferences. Series: Trends in Mathematics Bogopolski, O.;
  Bumagin, I.; Kharlampovich, O.; Ventura, E. (Eds.) 2009.},
  arXiv:0906.2223v1, 2009.

\bibitem[DFZ86]{DrorFarjounZabrodsky1986}
E.~Dror~Farjoun \& A.~Zabrodsky.
\newblock `Homotopy equivalence between diagrams of spaces'.
\newblock {\em J. Pure Appl. Algebra}, 41, no.~2-3 (1986), pp. 169--182.
\newblock {\sc doi:} \href {http://dx.doi.org/10.1016/0022-4049(86)90109-X}
  {{10.1016/0022-4049(86)90109-X}}.

\bibitem[DL98]{DavisLuck1998}
J.~F. Davis \& W.~L\"uck.
\newblock `Spaces over a category and assembly maps in isomorphism conjectures
  in {$K$}- and {$L$}-theory'.
\newblock {\em $K$-Theory}, 15, no.~3 (1998), pp. 201--252.
\newblock {\sc doi:} \href {http://dx.doi.org/10.1023/A:1007784106877}
  {{10.1023/A:1007784106877}}.

\bibitem[EG57]{Eilenberg1957}
S.~Eilenberg \& T.~Ganea.
\newblock `On the {L}usternik-{S}chnirelmann category of abstract groups'.
\newblock {\em Ann. of Math. (2)}, 65 (1957), pp. 517--518.
\newblock {\sc doi:} \href {http://dx.doi.org/10.2307/1970062}
  {{10.2307/1970062}}.

\bibitem[Fie84]{Fiedorowicz1984}
Z.~Fiedorowicz.
\newblock `Classifying spaces of topological monoids and categories'.
\newblock {\em Amer. J. Math.}, 106, no.~2 (1984), pp. 301--350.
\newblock {\sc doi:} \href {http://dx.doi.org/10.2307/2374307}
  {{10.2307/2374307}}.

\bibitem[FMWZ13]{Fluch2013}
M.~G. Fluch, M.~Marschler, S.~Witzel, \& M.~C.~B. Zaremsky.
\newblock `The {B}rin-{T}hompson groups {$sV$} are of type
  {$\text{F}_\infty$}'.
\newblock {\em Pacific J. Math.}, 266, no.~2 (2013), pp. 283--295.
\newblock {\sc doi:} \href {http://dx.doi.org/10.2140/pjm.2013.266.283}
  {{10.2140/pjm.2013.266.283}}.

\bibitem[For95]{Forman1995}
R.~Forman.
\newblock `A discrete {M}orse theory for cell complexes'.
\newblock In {\em Geometry, topology, \& physics}, Conf. Proc. Lecture Notes
  Geom. Topology, IV, pp. 112--125. Int. Press, Cambridge, MA, 1995.

\bibitem[For02]{Forman2002}
R.~Forman.
\newblock `A user's guide to discrete {M}orse theory'.
\newblock {\em S\'{e}minaire Lotharingien de Combinatoire}, B48c (2002), pp.
  1--35.
\newblock {\sc url:}
  \href{http://www.emis.ams.org/journals/SLC/wpapers/s48forman.pdf}{\nolinkurl{http://www.emis.ams.org/journals/SLC/wpapers/s48forman.pdf}}.

\bibitem[Fre09]{Freij}
R.~Freij.
\newblock `Equivariant discrete {M}orse theory'.
\newblock {\em Discrete Math.}, 309, no.~12 (2009), pp. 3821--3829.
\newblock {\sc doi:} \href {http://dx.doi.org/10.1016/j.disc.2008.10.029}
  {{10.1016/j.disc.2008.10.029}}.

\bibitem[FS05]{Farley}
D.~Farley \& L.~Sabalka.
\newblock `Discrete {M}orse theory and graph braid groups'.
\newblock {\em Algebr. Geom. Topol.}, 5 (2005), pp. 1075--1109.
\newblock {\sc doi:} \href {http://dx.doi.org/10.2140/agt.2005.5.1075}
  {{10.2140/agt.2005.5.1075}}.

\bibitem[Geo08]{GeogheganBook}
R.~Geoghegan.
\newblock {\em Topological methods in group theory}, vol. 243 of {\em Graduate
  Texts in Mathematics}.
\newblock Springer, New York, 2008.
\newblock {\sc doi:} \href {http://dx.doi.org/10.1007/978-0-387-74614-2}
  {{10.1007/978-0-387-74614-2}}.

\bibitem[GM12]{Guiraud2012}
Y.~Guiraud \& P.~Malbos.
\newblock `Higher-dimensional normalisation strategies for acyclicity'.
\newblock {\em Adv. Math.}, 231, no.~3-4 (2012), pp. 2294--2351.
\newblock {\sc doi:} \href {http://dx.doi.org/10.1016/j.aim.2012.05.010}
  {{10.1016/j.aim.2012.05.010}}.

\bibitem[GP96]{Guba1996}
V.~S. Guba \& S.~J. Pride.
\newblock `Low-dimensional (co)homology of free {B}urnside monoids'.
\newblock {\em J. Pure Appl. Algebra}, 108, no.~1 (1996), pp. 61--79.
\newblock {\sc doi:} \href {http://dx.doi.org/10.1016/0022-4049(95)00038-0}
  {{10.1016/0022-4049(95)00038-0}}.

\bibitem[GP98]{Guba1998}
V.~S. Guba \& S.~J. Pride.
\newblock `On the left and right cohomological dimension of monoids'.
\newblock {\em Bull. London Math. Soc.}, 30, no.~4 (1998), pp. 391--396.
\newblock {\sc doi:} \href {http://dx.doi.org/10.1112/S0024609398004676}
  {{10.1112/S0024609398004676}}.

\bibitem[GS]{GraySteinbergInPrep}
R.~D. Gray \& B.~Steinberg.
\newblock `Topological finiteness properties of monoids. {P}art 2: amalgamated
  free products, {HNN} extensions, and special monoids'.
\newblock {\em in preparation}.

\bibitem[GS06]{GubaSapirDirected}
V.~S. Guba \& M.~V. Sapir.
\newblock `Diagram groups and directed 2-complexes: homotopy and homology'.
\newblock {\em J. Pure Appl. Algebra}, 205, no.~1 (2006), pp. 1--47.
\newblock {\sc doi:} \href {http://dx.doi.org/10.1016/j.jpaa.2005.06.012}
  {{10.1016/j.jpaa.2005.06.012}}.

\bibitem[GS08]{Goodman2008}
O.~Goodman \& M.~Shapiro.
\newblock `On a generalization of {D}ehn's algorithm'.
\newblock {\em Internat. J. Algebra Comput.}, 18, no.~7 (2008), pp. 1137--1177.

\bibitem[GZ67]{GabrielZisman}
P.~Gabriel \& M.~Zisman.
\newblock {\em Calculus of fractions and homotopy theory}.
\newblock Ergebnisse der Mathematik und ihrer Grenzgebiete, Band 35.
  Springer-Verlag New York, Inc., New York, 1967.

\bibitem[HO14]{HessArxiv}
A.~Hess \& V.~Ozornova.
\newblock `Factorability, string rewriting and discrete morse theory'.
\newblock {\em arXiv:1412.3025},  2014.

\bibitem[Hoc45]{HochschildCoh}
G.~Hochschild.
\newblock `On the cohomology groups of an associative algebra'.
\newblock {\em Ann. of Math. (2)}, 46 (1945), pp. 58--67.
\newblock {\sc doi:} \href {http://dx.doi.org/10.2307/1969145}
  {{10.2307/1969145}}.

\bibitem[How63]{Howie1963}
J.~M. Howie.
\newblock `Embedding theorems for semigroups'.
\newblock {\em Quart. J. Math. Oxford Ser. (2)}, 14 (1963), pp. 254--258.
\newblock {\sc doi:} \href {http://dx.doi.org/10.1093/qmath/14.1.254}
  {{10.1093/qmath/14.1.254}}.

\bibitem[How95]{Howie}
J.~M. Howie.
\newblock {\em Fundamentals of semigroup theory}.
\newblock Academic Press [Harcourt Brace Jovanovich Publishers], London, 1995.
\newblock L.M.S. Monographs, No. 7.

\bibitem[HS99]{Hermiller1999}
S.~Hermiller \& M.~Shapiro.
\newblock `Rewriting systems and geometric three-manifolds'.
\newblock {\em Geom. Dedicata}, 76, no.~2 (1999), pp. 211--228.

\bibitem[Hur89]{Hurwitz1989}
C.~M. Hurwitz.
\newblock `On the homotopy theory of monoids'.
\newblock {\em J. Austral. Math. Soc. Ser. A}, 47, no.~2 (1989), pp. 171--185.

\bibitem[KM05]{Leray0}
G.~Kalai \& R.~Meshulam.
\newblock `A topological colorful {H}elly theorem'.
\newblock {\em Adv. Math.}, 191, no.~2 (2005), pp. 305--311.
\newblock {\sc doi:} \href {http://dx.doi.org/10.1016/j.aim.2004.03.009}
  {{10.1016/j.aim.2004.03.009}}.

\bibitem[Kna72]{Knauer1971}
U.~Knauer.
\newblock `Projectivity of acts and {M}orita equivalence of monoids'.
\newblock {\em Semigroup Forum}, 3, no.~4 (1971/1972), pp. 359--370.

\bibitem[KO01]{KobayashiOtto2001}
Y.~Kobayashi \& F.~Otto.
\newblock `On homotopical and homological finiteness conditions for finitely
  presented monoids'.
\newblock {\em Internat. J. Algebra Comput.}, 11, no.~3 (2001), pp. 391--403.
\newblock {\sc doi:} \href {http://dx.doi.org/10.1142/S0218196701000577}
  {{10.1142/S0218196701000577}}.

\bibitem[Kob00]{Kobayashi2000}
Y.~Kobayashi.
\newblock `Every one-relation monoid has finite derivation type'.
\newblock In {\em Proceedings of the {T}hird {S}ymposium on {A}lgebra,
  {L}anguages and {C}omputation ({O}saka, 1999)}, pp. 16--20. Shimane Univ.,
  Matsue, 2000.

\bibitem[Kob05]{Kobayashi2005}
Y.~Kobayashi.
\newblock `Gr\"obner bases of associative algebras and the {H}ochschild
  cohomology'.
\newblock {\em Trans. Amer. Math. Soc.}, 357, no.~3 (2005), pp. 1095--1124
  (electronic).
\newblock {\sc doi:} \href {http://dx.doi.org/10.1090/S0002-9947-04-03556-1}
  {{10.1090/S0002-9947-04-03556-1}}.

\bibitem[Kob07]{Kobayashi2007}
Y.~Kobayashi.
\newblock `The homological finiteness property {${\rm FP}_1$} and finite
  generation of monoids'.
\newblock {\em Internat. J. Algebra Comput.}, 17, no.~3 (2007), pp. 593--605.
\newblock {\sc doi:} \href {http://dx.doi.org/10.1142/S0218196707003743}
  {{10.1142/S0218196707003743}}.

\bibitem[Kob10]{Kobayashi2010}
Y.~Kobayashi.
\newblock `The homological finiteness properties left-, right-, and bi-{${\rm
  FP}_n$} of monoids'.
\newblock {\em Comm. Algebra}, 38, no.~11 (2010), pp. 3975--3986.
\newblock {\sc doi:} \href {http://dx.doi.org/10.1080/00927872.2010.507562}
  {{10.1080/00927872.2010.507562}}.

\bibitem[Koz05]{MorseChain}
D.~N. Kozlov.
\newblock `Discrete {M}orse theory for free chain complexes'.
\newblock {\em C. R. Math. Acad. Sci. Paris}, 340, no.~12 (2005), pp. 867--872.
\newblock {\sc doi:} \href {http://dx.doi.org/10.1016/j.crma.2005.04.036}
  {{10.1016/j.crma.2005.04.036}}.

\bibitem[KT76]{KanThurston1976}
D.~M. Kan \& W.~P. Thurston.
\newblock `Every connected space has the homology of a {$K(\pi ,1)$}'.
\newblock {\em Topology}, 15, no.~3 (1976), pp. 253--258.

\bibitem[Lam99]{LamBook}
T.~Y. Lam.
\newblock {\em Lectures on modules and rings}, vol. 189 of {\em Graduate Texts
  in Mathematics}.
\newblock Springer-Verlag, New York, 1999.
\newblock {\sc doi:} \href {http://dx.doi.org/10.1007/978-1-4612-0525-8}
  {{10.1007/978-1-4612-0525-8}}.

\bibitem[LN01]{LearyNucinkis2001}
I.~J. Leary \& B.~E.~A. Nucinkis.
\newblock `Every {CW}-complex is a classifying space for proper bundles'.
\newblock {\em Topology}, 40, no.~3 (2001), pp. 539--550.
\newblock {\sc doi:} \href {http://dx.doi.org/10.1016/S0040-9383(99)00073-7}
  {{10.1016/S0040-9383(99)00073-7}}.

\bibitem[LS06]{Leary2006}
I.~J. Leary \& M.~Saadeto{\u{g}}lu.
\newblock `Some groups of finite homological type'.
\newblock {\em Geom. Dedicata}, 119 (2006), pp. 113--120.

\bibitem[May67]{MaySimplicialObjects}
J.~P. May.
\newblock {\em Simplicial objects in algebraic topology}.
\newblock Van Nostrand Mathematical Studies, No. 11. D. Van Nostrand Co., Inc.,
  Princeton, N.J.-Toronto, Ont.-London, 1967.

\bibitem[May72]{May1972}
J.~P. May.
\newblock {\em The geometry of iterated loop spaces}.
\newblock Springer-Verlag, Berlin-New York, 1972.
\newblock Lectures Notes in Mathematics, Vol. 271.

\bibitem[May75]{May1975}
J.~P. May.
\newblock `Classifying spaces and fibrations'.
\newblock {\em Mem. Amer. Math. Soc.}, 1, no.~1, 155 (1975), pp. xiii+98.

\bibitem[May96]{May1996}
J.~P. May.
\newblock {\em Equivariant homotopy and cohomology theory}, vol.~91 of {\em
  CBMS Regional Conference Series in Mathematics}.
\newblock Published for the Conference Board of the Mathematical Sciences,
  Washington, DC; by the American Mathematical Society, Providence, RI, 1996.
\newblock With contributions by M. Cole, G. Comeza\~na, S. Costenoble, A. D.
  Elmendorf, J. P. C. Greenlees, L. G. Lewis, Jr., R. J. Piacenza, G.
  Triantafillou, and S. Waner.
\newblock {\sc doi:} \href {http://dx.doi.org/10.1090/cbms/091}
  {{10.1090/cbms/091}}.

\bibitem[May99]{May1999}
J.~P. May.
\newblock {\em A concise course in algebraic topology}.
\newblock Chicago Lectures in Mathematics. University of Chicago Press,
  Chicago, IL, 1999.

\bibitem[McC69]{McCord1969}
M.~C. McCord.
\newblock `Classifying spaces and infinite symmetric products'.
\newblock {\em Trans. Amer. Math. Soc.}, 146 (1969), pp. 273--298.

\bibitem[McD79]{McDuff1979}
D.~McDuff.
\newblock `On the classifying spaces of discrete monoids'.
\newblock {\em Topology}, 18, no.~4 (1979), pp. 313--320.
\newblock {\sc doi:} \href {http://dx.doi.org/10.1016/0040-9383(79)90022-3}
  {{10.1016/0040-9383(79)90022-3}}.

\bibitem[Mil57]{Milnor1957}
J.~Milnor.
\newblock `The geometric realization of a semi-simplicial complex'.
\newblock {\em Ann. of Math. (2)}, 65 (1957), pp. 357--362.

\bibitem[Mit72]{Mitchell1972}
B.~Mitchell.
\newblock `Rings with several objects'.
\newblock {\em Advances in Math.}, 8 (1972), pp. 1--161.
\newblock {\sc doi:} \href {http://dx.doi.org/10.1016/0001-8708(72)90002-3}
  {{10.1016/0001-8708(72)90002-3}}.

\bibitem[MS76]{McDuffSegal1975}
D.~McDuff \& G.~Segal.
\newblock `Homology fibrations and the ``group-completion'' theorem'.
\newblock {\em Invent. Math.}, 31, no.~3 (1975/76), pp. 279--284.

\bibitem[MSS15a]{MSS2015}
S.~{Margolis}, F.~{Saliola}, \& B.~{Steinberg}.
\newblock `{Cell complexes, poset topology and the representation theory of
  algebras arising in algebraic combinatorics and discrete geometry}'.
\newblock {\em ArXiv e-prints},  2015.
\newblock arXiv:~\href {http://arxiv.org/abs/1508.05446} {{1508.05446}}.

\bibitem[MSS15b]{Margolis2015}
S.~Margolis, F.~Saliola, \& B.~Steinberg.
\newblock `Combinatorial topology and the global dimension of algebras arising
  in combinatorics'.
\newblock {\em J. Eur. Math. Soc. (JEMS)}, 17, no.~12 (2015), pp. 3037--3080.
\newblock {\sc doi:} \href {http://dx.doi.org/10.4171/JEMS/579}
  {{10.4171/JEMS/579}}.

\bibitem[Nic69]{Nico1969}
W.~R. Nico.
\newblock `On the cohomology of finite semigroups'.
\newblock {\em J. Algebra}, 11 (1969), pp. 598--612.
\newblock {\sc doi:} \href {http://dx.doi.org/10.1016/0021-8693(69)90093-3}
  {{10.1016/0021-8693(69)90093-3}}.

\bibitem[Nic72]{Nico1972}
W.~R. Nico.
\newblock `A counterexample in the cohomology of monoids'.
\newblock {\em Semigroup Forum}, 4 (1972), pp. 93--94.
\newblock {\sc doi:} \href {http://dx.doi.org/10.1007/BF02570775}
  {{10.1007/BF02570775}}.

\bibitem[Nie]{248371}
P.~Nielsen.
\newblock `Can the trivial module be stably free for a monoid ring?'.
\newblock MathOverflow.
\newblock URL:https://mathoverflow.net/q/248371 (version: 2016-08-26).
\newblock {\sc url:}
  \href{https://mathoverflow.net/q/248371}{\nolinkurl{https://mathoverflow.net/q/248371}}.

\bibitem[Nov98]{Novikov1998}
B.~V. Novikov.
\newblock `Semigroups of cohomological dimension one'.
\newblock {\em J. Algebra}, 204, no.~2 (1998), pp. 386--393.
\newblock {\sc doi:} \href {http://dx.doi.org/10.1006/jabr.1997.7363}
  {{10.1006/jabr.1997.7363}}.

\bibitem[Nun95]{Nunes1995}
M.~Nunes.
\newblock `Cohomological results in monoid and category theory via classifying
  spaces'.
\newblock {\em J. Pure Appl. Algebra}, 101, no.~2 (1995), pp. 213--244.
\newblock {\sc doi:} \href {http://dx.doi.org/10.1016/0022-4049(94)00056-O}
  {{10.1016/0022-4049(94)00056-O}}.

\bibitem[OK97]{Otto1997}
F.~Otto \& Y.~Kobayashi.
\newblock `Properties of monoids that are presented by finite convergent
  string-rewriting systems---a survey'.
\newblock In {\em Advances in algorithms, languages, and complexity}, pp.
  225--266. Kluwer Acad. Publ., Dordrecht, 1997.

\bibitem[Pas08]{Pasku2008}
E.~Pasku.
\newblock `On some homotopical and homological properties of monoid
  presentations'.
\newblock {\em Semigroup Forum}, 76, no.~3 (2008), pp. 427--468.
\newblock {\sc doi:} \href {http://dx.doi.org/10.1007/s00233-007-9037-1}
  {{10.1007/s00233-007-9037-1}}.

\bibitem[PO04]{Pride2004}
S.~J. Pride \& F.~Otto.
\newblock `For rewriting systems the topological finiteness conditions {FDT}
  and {FHT} are not equivalent'.
\newblock {\em J. London Math. Soc. (2)}, 69, no.~2 (2004), pp. 363--382.
\newblock {\sc doi:} \href {http://dx.doi.org/10.1112/S0024610703004903}
  {{10.1112/S0024610703004903}}.

\bibitem[PO05]{Pride2005}
S.~J. Pride \& F.~Otto.
\newblock `On higher order homological finiteness of rewriting systems'.
\newblock {\em J. Pure Appl. Algebra}, 200, no.~1-2 (2005), pp. 149--161.

\bibitem[Pri06]{Pride2006}
S.~J. Pride.
\newblock `Homological finiteness conditions for groups, monoids, and
  algebras'.
\newblock {\em Comm. Algebra}, 34, no.~10 (2006), pp. 3525--3536.
\newblock {\sc doi:} \href {http://dx.doi.org/10.1080/00927870600796110}
  {{10.1080/00927870600796110}}.

\bibitem[Pup58]{Puppe1958}
D.~Puppe.
\newblock `Homotopie und {H}omologie in abelschen {G}ruppen- und
  {M}onoidkomplexen. {I}. {II}'.
\newblock {\em Math. Z.}, 68 (1958), pp. 367--406, 407--421.

\bibitem[Pup59]{Puppe1959}
D.~Puppe.
\newblock `A theorem on semi-simplicial monoid complexes'.
\newblock {\em Ann. of Math. (2)}, 70 (1959), pp. 379--394.

\bibitem[RT89]{RhodesTilson1989}
J.~Rhodes \& B.~Tilson.
\newblock `The kernel of monoid morphisms'.
\newblock {\em J. Pure Appl. Algebra}, 62, no.~3 (1989), pp. 227--268.
\newblock {\sc doi:} \href {http://dx.doi.org/10.1016/0022-4049(89)90137-0}
  {{10.1016/0022-4049(89)90137-0}}.

\bibitem[Sch75]{Schein1975}
B.~M. Schein.
\newblock `Free inverse semigroups are not finitely presentable'.
\newblock {\em Acta Math. Acad. Sci. Hungar.}, 26 (1975), pp. 41--52.
\newblock {\sc doi:} \href {http://dx.doi.org/10.1007/BF01895947}
  {{10.1007/BF01895947}}.

\bibitem[Seg78]{Segal1978}
G.~Segal.
\newblock `Classifying spaces related to foliations'.
\newblock {\em Topology}, 17, no.~4 (1978), pp. 367--382.
\newblock {\sc doi:} \href {http://dx.doi.org/10.1016/0040-9383(78)90004-6}
  {{10.1016/0040-9383(78)90004-6}}.

\bibitem[Ser71]{Serre1971}
J.-P. Serre.
\newblock `Cohomologie des groupes discrets'.
\newblock  (1971), pp. 77--169. Ann. of Math. Studies, No. 70.

\bibitem[SOK94]{Squier1994}
C.~C. Squier, F.~Otto, \& Y.~Kobayashi.
\newblock `A finiteness condition for rewriting systems'.
\newblock {\em Theoret. Comput. Sci.}, 131, no.~2 (1994), pp. 271--294.
\newblock {\sc doi:} \href {http://dx.doi.org/10.1016/0304-3975(94)90175-9}
  {{10.1016/0304-3975(94)90175-9}}.

\bibitem[Squ87]{Squier1987}
C.~C. Squier.
\newblock `Word problems and a homological finiteness condition for monoids'.
\newblock {\em J. Pure Appl. Algebra}, 49, no.~1-2 (1987), pp. 201--217.
\newblock {\sc doi:} \href {http://dx.doi.org/10.1016/0022-4049(87)90129-0}
  {{10.1016/0022-4049(87)90129-0}}.

\bibitem[Sri96]{SrinivasBook}
V.~Srinivas.
\newblock {\em Algebraic {$K$}-theory}, vol.~90 of {\em Progress in
  Mathematics}.
\newblock Birkh\"auser Boston, Inc., Boston, MA, second edition, 1996.
\newblock {\sc doi:} \href {http://dx.doi.org/10.1007/978-0-8176-4739-1}
  {{10.1007/978-0-8176-4739-1}}.

\bibitem[Sta68]{Stallings1968}
J.~R. Stallings.
\newblock `On torsion-free groups with infinitely many ends'.
\newblock {\em Ann. of Math. (2)}, 88 (1968), pp. 312--334.
\newblock {\sc doi:} \href {http://dx.doi.org/10.2307/1970577}
  {{10.2307/1970577}}.

\bibitem[Ste92]{Stein1992}
M.~Stein.
\newblock `Groups of piecewise linear homeomorphisms'.
\newblock {\em Trans. Amer. Math. Soc.}, 332, no.~2 (1992), pp. 477--514.
\newblock {\sc doi:} \href {http://dx.doi.org/10.2307/2154179}
  {{10.2307/2154179}}.

\bibitem[Swa69]{Swan1969}
R.~G. Swan.
\newblock `Groups of cohomological dimension one'.
\newblock {\em J. Algebra}, 12 (1969), pp. 585--610.
\newblock {\sc doi:} \href {http://dx.doi.org/10.1016/0021-8693(69)90030-1}
  {{10.1016/0021-8693(69)90030-1}}.

\bibitem[Wal65]{Wall1965}
C.~T.~C. Wall.
\newblock `Finiteness conditions for {${\rm CW}$}-complexes'.
\newblock {\em Ann. of Math. (2)}, 81 (1965), pp. 56--69.

\bibitem[Wei94]{Weibel1994}
C.~A. Weibel.
\newblock {\em An introduction to homological algebra}, vol.~38 of {\em
  Cambridge Studies in Advanced Mathematics}.
\newblock Cambridge University Press, Cambridge, 1994.
\newblock {\sc doi:} \href {http://dx.doi.org/10.1017/CBO9781139644136}
  {{10.1017/CBO9781139644136}}.

\bibitem[Wei13]{WeibelKBook}
C.~A. Weibel.
\newblock {\em The {$K$}-book}, vol. 145 of {\em Graduate Studies in
  Mathematics}.
\newblock American Mathematical Society, Providence, RI, 2013.
\newblock An introduction to algebraic $K$-theory.

\end{thebibliography}
\end{document}